\documentclass[a4paper, 12pt, tikz]{emilsart}
\title{Inversion monotonicity in subclasses\\ of the 1324-avoiders}
\author{Anders Claesson, Svante Linusson, Henning Ulfarsson, and Emil Verkama}
\date{}

\usepackage[backend=biber, doi=false, isbn=false, giveninits=true, style=alphabetic, maxbibnames=10]{biblatex}
\DeclareFieldFormat*{title}{``#1''}
\DeclareMathOperator{\SPM}{SPM}
\newcommand{\delete}{\smallsetminus}

\begin{document}
\maketitle

\begin{abstract}
  We prove injectively that the pattern-avoiding permutation classes \(\Av(1324, 231)\) and \(\Av(1324, 2314, 3214, 4213)\) are inversion monotone: for every \(k\), the number of length-\(n\) permutations with \(k\) inversions in either class is weakly increasing in \(n\). These are the first known nontrivial examples of the phenomenon. The latter is a special case of a general construction. We also characterize the finitely based permutation classes in which the number of length-\(n\) members with a fixed number of inversions stabilizes as \(n\) grows, and determine the limiting values for all classes \(\Av(1324, p)\) with \(p \in S_4\). Connections to several nice families of integer partitions are revealed. Lastly, we show that for asymptotically almost every pattern \(p \in \Av(1324)\), the class \(\Av(1324, p)\) is inversion monotone in the region \(n \geq \frac{k+7}{2}\), expanding on the work of Linusson and Verkama. This work is motivated by the conjectures of Claesson, Jelínek and Steingrímsson that \(\Av(1324)\), and more generally \(\Av(p)\) for every nonidentity pattern \(p\), is inversion monotone.
\end{abstract}

\begingroup
\hypersetup{linkcolor=black}
\tableofcontents
\endgroup

\clearpage

\section{Introduction}

Let \(\Av_n^k(p)\) denote the set of length-\(n\) permutations with exactly \(k\) inversions that avoid the pattern \(p\), and let \(\av_n^k(p) = |{\Av_n^k(p)}|\). We say that \(\Av(p)\) is \emph{inversion monotone} if
\begin{equation} \label{eq:inv mono}
  \av_n^k(p) \leq \av_{n+1}^k(p) \quad \text{for all } n, k.
\end{equation}
With minor abuse of terminology, we call \(p\) itself inversion monotone if \(\Av(p)\) is inversion monotone.

The notion was introduced in 2012 by Claesson, Jelínek and Steingrímsson \cite{claesson_upper_2012} in their study of the notorious class \(\Av(1324)\), whose exact enumeration and asymptotic growth rate (\emph{Stanley--Wilf limit}, see \cite{marcus_excluded_2004}) are long-standing open problems. They showed that in this case, the property \eqref{eq:inv mono} would have a special consequence: if \(1324\) is inversion monotone, then its Stanley--Wilf limit is less than \(13.002\). Currently, the best known upper and lower bounds for the Stanley--Wilf limit of \(1324\) are \(13.5\) and \(10.27\), respectively \cite{bevan_structural_2020}, whereas the actual value is estimated to be \(11.600 \pm 0.003\) \cite{conway_1324-avoiding_2018}. However, inversion monotonicity itself appears to be a more general phenomenon.

\begin{conjecture}[Conjecture 20 in \cite{claesson_upper_2012}] \label{conj:invomono}
  All patterns, except for the identity patterns, are inversion monotone. 
\end{conjecture}

The identity pattern \(\id\) is not inversion monotone, as it is easy to show that for every \(k\), \(\av_n^k(\id) = 0\) whenever \(n\) is large enough (see e.g.\ \eqref{eq:comp inv length} in Section \ref{sec:preliminaries}). Furthermore, if \(p\) is a pattern that does not start with \(1\) or does not end with its largest entry, then \(p\) is trivially inversion monotone; in the latter case, setting \(n+1\) as the last entry defines an injection from \(\Av_n^k(p)\) to \(\Av_{n+1}^k(p)\), and the former case is symmetric. However, for \emph{all} nontrivial patterns, Conjecture~\ref{conj:invomono} remains open. Linusson and Verkama recently made progress on the important case \(1324\) (which is also the shortest nontrivial pattern) by proving that \eqref{eq:inv mono} holds for all \(k\) and \(n \geq \frac{k+7}{2}\) \cite{linusson_enumerating_2025}. The distribution \(\av_n^k(1324)\) was also considered by Johansson and Nakamura \cite{johansson_using_2014}.
 
This paper initiates the study of inversion monotonicity in general permutation classes \(\Av(B)\), with particular attention to subclasses of \(\Av(1324)\). A source of motivation for this setting is that such subclasses can approximate \(\Av(1324)\): if for each \(N\) one can find a collection \(B_N\) consisting of patterns of length at least \(N\) such that \(\{1324\} \cup B_N\) is inversion monotone, then \(1324\) itself is inversion monotone. 

Conversely, the study of inversion monotonicity often yields new structural and enumerative insights into the classes \(\Av(B)\) themselves. For example, we investigate when \(\av_n^k(B)\) stabilizes as \(n\) grows and \(k\) is fixed, and determine the resulting limit sequence in specific classes.

\paragraph{Inversion monotone collections}

Our main result provides the first known nontrivial examples of the inversion monotonicity phenomenon.

\begin{theorem*}
  The collections \(\{1324, 231\}\) and \(\{1324, 2314, 3214, 4213\}\) are inversion monotone.
\end{theorem*}

The former is proved by an explicit, intricate injection (Theorem \ref{thm:1324 231 inv mono}), but the latter is a special case of a more general construction. To recover it, set \(B = \{213\}\) in the following result (Propositions \ref{prop:building B} and \ref{prop:bb inv mono 1324}).

\begin{proposition*}
  Let \(B^{\ell}\) denote the collection of patterns obtained from the patterns of \(B\) by inserting a new first entry in all possible ways. If \(B\) is inversion monotone, then so is \(B^{\ell}\).
\end{proposition*}

\paragraph{Limit sequences}

As mentioned, Claesson, Jelínek and Steingrímsson \cite{claesson_upper_2012} showed that if \(1324\) is inversion monotone, then its Stanley--Wilf limit is less than \(13.002\). This estimate is based on the fact that if \(n \geq k+2\), then
\begin{equation*}
  \av_n^k(1324) = \sum_{i = 0}^k p(i) p(k-i) \eqqcolon a(k),
\end{equation*}
where \(p(k)\) is the number of integer partitions of \(k\) (see Section \ref{sec:preliminaries}). In particular, \(a(k)\) is independent of \(n\). Therefore, if \(1324\) is inversion monotone, then the number of \(1324\)-avoiders of length \(n\) is
\begin{equation*}
  \sum_{k=0}^{\binom n2} \av_n^k(1324) \leq \sum_{k=0}^{\binom n2} a(k) \leq \left(\tbinom n2 + 1\right) a \left(\tbinom n2\right),
\end{equation*}
on which classical estimates may be applied. This is why the \emph{limit sequence} \(a(0), a(1), a(2), \ldots\) of \(1324\) is important.

We extend the study of limit sequences to arbitrary permutation classes \(\Av(B)\). The key result, extending part of \cite[Proposition~21]{claesson_upper_2012}, is the following (Proposition~\ref{prop:lim seq existence}).

\begin{proposition*}
  A finite collection \(B\) of patterns has a limit sequence if and only if it contains a pattern \(p\) with \(\inv(p) \leq 1\).
\end{proposition*}

We are able to determine the limit sequences of every pair \(\{1324, p\}\), where \(p\) is a pattern of length four (Theorem \ref{thm:1324 pair limits}). The tool that makes this systematic is a factorization result for the limit sequence of an arbitrary finite collection containing \(1324\) (Proposition \ref{prop:limit factorization} and Corollary \ref{cor:limit factorization}). The problem reduces to counting \emph{indecomposable} permutations avoiding \(132\) and one further pattern by the number of inversions, expanding on work by Franklín \cite{franklin_pattern_2025} (Section \ref{sec:132 indecomposable}). This, in turn, is equivalent to enumerating restricted integer partitions. Many interesting classes arise: the sand pile model, overpartitions, convex and steep partitions, and penny arrangements.

\paragraph{Compatible patterns and half-monotonicity}

In \cite{linusson_enumerating_2025}, Linusson and Verkama defined a certain injective mapping on the \emph{almost decomposable} \(1324\)-avoiding permutations and proved that if \(k \leq 2n-7\), every permutation in \(\Av_n^k(1324)\) is decomposable or almost decomposable. As a consequence, \(\av_n^k(1324) \leq \av_{n+1}^k(1324)\) for every \(n \geq \frac{k+7}{2}\); we say that \(1324\) is \emph{half-monotone}. 

We investigate which additional avoidance conditions are preserved by the Linusson--Verkama injection. We call such patterns \emph{compatible}; if \(p\) is compatible, then \(\{1324, p\}\) is half-monotone. More generally, if \(B\) is a collection consisting of compatible patterns, then \(\{1324\} \cup B\) is half-monotone. Theorems \ref{thm:incomp necessary} and \ref{thm:incomp sufficient} give, respectively, sufficient and necessary conditions for a pattern to be compatible. As a consequence, we obtain the following result (Theorem \ref{thm:incomp asymptotics}). 

\begin{theorem*}
  Almost every pattern \(p \in \Av(1324)\) satisfies
  \begin{equation*}
    \av_n^k(1324, p) \leq \av_{n+1}^k(1324, p) \quad \text{for all } k \text{ and } n \geq \tfrac{k+7}{2}.
  \end{equation*}
\end{theorem*}

This can be lifted to sets of patterns; Corollary \ref{cor:almost all sets comp} says roughly that for almost every collection \(B \subseteq \Av(1324)\) of bounded cardinality, \(\{1324\} \cup B\) is half-monotone. Theorem \ref{thm:incomp asymptotics} also gives precise asymptotics for the number of compatible patterns.

Lastly, we connect the method to enumeration with the pair \(\{1324, 1342\}\). The pattern \(1342\) is not compatible, but all incompatibilities occur in the region \(n < \frac{k+7}{2}\) (Theorem \ref{thm:1342 almost compatible}), so half-monotonicity holds. We enumerate the differences \(\av_{n+1}^k(1324, 1342) - \av_n^k(1324, 1342)\) when \(k \leq 2n-7\) by analyzing the permutations outside of the image of the Linusson--Verkama injection. Using the previously determined limit sequence, this leads to the following result (Theorem~\ref{thm:1342 enum}).

\begin{theorem*} 
  For every \(n \geq \frac{k+7}{2}\), the difference \(\av_{n+1}^k(1324, 1342) - \av_n^k(1324, 1342)\) is nonnegative, and is equal to the coefficient of \(x^k\) in the generating function
  \begin{equation*}
    x^{n-1} (2+2x) \cdot \prod_{i \geq 1} \frac{1 + x^i}{1 - x^i}.
  \end{equation*}
  In particular, when \(n \geq \frac{k+7}{2}\), 
  \begin{equation*}
    \av_n^k(1324, 1342) = [x^k] \left( \frac{1 - x - x^{n-1}(2 + 2x)}{1 - x} \cdot \prod_{i \geq 1} \frac{1 + x^i}{1 - x^i} \right).
  \end{equation*}
\end{theorem*}

\paragraph{Outline}

Section \ref{sec:preliminaries} recalls the necessary preliminaries. In Section \ref{sec:1324 231} we prove that \(\{1324, 231\}\) is inversion monotone, whereas Section \ref{sec:building bases} describes the procedure for building new, larger inversion monotone collections. The limit sequence theory is built in Section \ref{sec:1324 pairs}. The enumeration results for indecomposable permutations avoiding \(\{132, p\}\) needed for the limit sequences are in Section \ref{sec:132 indecomposable}. Section \ref{sec:almost decomp} discusses almost decomposable permutations and compatible patterns. Finally, Section~\ref{sec:conclusions} contains some concluding remarks and open problems.

\section{Preliminaries} \label{sec:preliminaries}

A permutation \(\pi \in S_n\) \emph{contains} a pattern \(p \in S_m\) if \(\pi\) has a subsequence that is order-isomorphic to \(p\). If \(\pi\) does not contain \(p\), then \(\pi\) \emph{avoids} \(p\). If \(B\) is a set of patterns, we say that \(\pi\) avoids \(B\) if \(\pi\) avoids every pattern contained in \(B\). We write \(\Av(B)\) for the set of all permutations that avoid \(B\), \(\Av_n(B) = \Av(B) \cap S_n\) and \(\av_n(B) = |{\Av_n(B)}|\). Furthermore, if \(\pi \in S_n\), we write \(|\pi| = n\).

It is convenient for us to have notation for the pattern obtained by deleting certain entries. If \(\pi \in S_n\) and \(A \subseteq [n]\), let \(\pi \delete A\) denote the permutation that is order-isomorphic to the subsequence obtained by removing all entries with values in \(A\) from the one-line notation of \(\pi\). For singletons, we write \(\pi \delete e = \pi \delete \{e\}\).

\paragraph{Inversions and decomposability}

An \emph{inversion} in a permutation \(\pi\) is a pair \((i,j)\) of indices such that \(i < j\) and \(\pi_i > \pi_j\). Let \(\inv(\pi)\) denote the number of inversions in \(\pi\). We write \(\Av_n^k(B)\) for the set of \(B\)-avoiding permutations of length \(n\) with exactly \(k\) inversions, and \(\av_n^k(B) = |{\Av_n^k(B)}|\). There are three (nontrivial) symmetries preserving the number of inversions of a permutation \(\pi \in S_n\): the inverse \(\pi^{-1}\), the reverse-complement \(\pi^{\rc}\), which is given by \(\pi^{\rc}_i = n + 1 - \pi_{n+1-i}\), as well as their composition \((\pi^{-1})^{\rc}\). The reverse-complement is the composition of two symmetries that sends inversions to noninversions, namely the reverse \(\pi^{\rev}_i = \pi_{n+1 - i}\) and the complement \(\pi^{\compl}_i = n + 1 - \pi_i\).

We define the \emph{direct sum} of \(\sigma \in S_n\) and \(\tau \in S_m\) as the permutation \(\sigma \oplus \tau \in S_{n+m}\) given by
\begin{equation*}
  (\sigma \oplus \tau)(i) = \begin{cases}
    \sigma(i) & \text{if } i \leq n, \\
    n + \tau(i-n) & \text{if } i > n.
  \end{cases}
\end{equation*}
Similarly, the \emph{skew sum} \(\sigma \ominus \tau\) is defined by
\begin{equation*}
  (\sigma \ominus \tau)(i) = \begin{cases}
    m + \sigma(i) & \text{if } i \leq n, \\
    \tau(i-n) & \text{if } i > n.
  \end{cases}
\end{equation*}
There is a nice visual interpretation. We will illustrate permutations using their plots in cartesian coordinates: \(\pi \in S_n\) becomes \(\{(i,\pi_i) : i \in [n]\}\). See Figure \ref{fig:perm sums} for an example: if \(\sigma = 21\) and \(\tau = 231\), then \(\sigma \oplus \tau = 21453\) and \(\sigma \ominus \tau = 54231\).

A permutation is called \emph{indecomposable} if it cannot be written as the direct sum of two nonempty permutations. Every permutation \(\pi\) can be written uniquely as a direct sum
\begin{equation*}
  \pi = \pi^{(1)} \oplus \pi^{(2)} \oplus \ldots \oplus \pi^{(r)}
\end{equation*}
of indecomposable permutations \(\pi^{(i)}\), called the \emph{components} of \(\pi\). Let \(\comp(\pi) = r\) denote the number of components of \(\pi\). The identity permutation of length \(n\) is denoted by \(\id_n\). If the length is irrelevant and can be left unspecified, we simply write \(\id\). Note that
\begin{equation*}
  \id_n = {\underbrace{1 \oplus 1 \oplus \ldots \oplus 1}_{n \text{ times}}}.
\end{equation*}
The number of inversions and components of a permutation \(\pi \in S_n\) are related by the Erd\H{o}s--Szekeres-type inequality
\begin{equation} \label{eq:comp inv length}
  \inv(\pi) + \comp(\pi) \geq n,
\end{equation}
see \cite[Lemma~8]{claesson_upper_2012}. In particular, if \(n \geq \inv(\pi) + 2\), then \(\pi\) must be decomposable. An important observation is that decomposable \(1324\)-avoiders have a simple form.

\begin{lemma}[Lemma 14 in \cite{claesson_upper_2012}] \label{lem:1324 decomp form}
  A decomposable permutation \(\pi\) avoids \(1324\) if and only if it is of the form \(\pi = \sigma \oplus {\id} \oplus \tau\), where \(\sigma\) is an indecomposable \(132\)-avoider and \(\tau\) is an indecomposable \(213\)-avoider.
\end{lemma}

\begin{figure}[t]
  \centering
  \begin{tikzpicture}[scale=0.5]
    \draw (0,0) rectangle (2,2);
    \foreach \x\y in {1/2,2/1}
      \node[dot] at (\x-0.5,\y-0.5) {};
    \node at (1,-0.8) {\(\sigma\)};

    \begin{scope}[shift={(3.5,0)}]
      \draw (0,0) rectangle (3,3);
      \foreach \x\y in {1/2,2/3,3/1}
        \node[dot] at (\x-0.5,\y-0.5) {};
      \node at (1.5,-0.8) {\(\tau\)};
    \end{scope}
    
    \begin{scope}[shift={(8,0)}]
      \draw (0,0) rectangle (5,5);
      \draw 
        (0,2) -- (5,2)
        (2,0) -- (2,5);
      \foreach \x\y in {1/2,2/1,3/4,4/5,5/3}
        \node[dot] at (\x-0.5,\y-0.5) {};
      \node at (2.5,-0.8) {\(\sigma \oplus \tau\)};
    \end{scope}
    
    \begin{scope}[shift={(14.5,0)}]
      \draw (0,0) rectangle (5,5);
      \draw
        (0,3) -- (5,3)
        (2,0) -- (2,5);
      \foreach \x\y in {1/5,2/4,3/2,4/3,5/1}
        \node[dot] at (\x-0.5,\y-0.5) {};
      \node at (2.5,-0.8) {\(\sigma \ominus \tau\)};
    \end{scope}
  \end{tikzpicture}
  \caption{The direct sum and skew sum of \(\sigma = 21\) and \(\tau = 231\).}
  \label{fig:perm sums}
\end{figure}
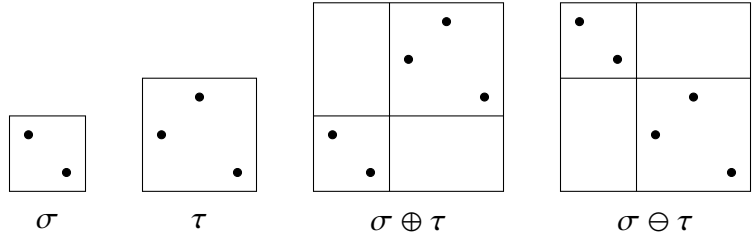

\begin{definition}
  A collection \(B\) of patterns is called \emph{inversion monotone} if for all \(k\) and \(n\),
  \begin{equation*}
    \av_n^k(B) \leq \av_{n+1}^k(B).
  \end{equation*}
  If \(B = \{p\}\), we say that \(p\) itself is inversion monotone.
\end{definition}

\paragraph{Limit sequences}

We often find that the sequences \((\av_n^k(B))_n\) converge to constant values \(c_k\). In these cases, analyzing the limiting sequence \((c_k)_k\) can be interesting.

\begin{definition}
  If \(B\) is a set of patterns such that for every \(k\) there exist constants \(m_k\) and \(c_k\) for which
  \begin{equation*}
    \av_n^k(B) = c_k \quad \text{for all } n \geq m_k,
  \end{equation*}  
  we say that \((c_k)_{k \geq 0}\) is the \emph{limit sequence} of \(B\), and denote \(c_k(B) = c_k\). We call
  \begin{equation*}
    C_B(x) = \sum_{k \geq 0} c_k(B) x^k
  \end{equation*}
  the \emph{limit generating function} of \(B\).
\end{definition}

Few limit sequences have been explicitly determined. A simple argument shows that \(C_{132}(x) = P(x)\) (using the inversion sequence; see \eqref{eq:132 partition bijection} below), from which it follows that \(C_{1324}(x) = P(x)^2\), where \(P(x)\) is the generating function for the partition numbers. However, for example \(C_{1243}(x)\) is unknown. Existence is better understood. Claesson, Jelínek and Steingrímsson proved that a pattern \(p\) has a limit sequence if and only if \(\inv(p) \leq 1\) \cite[Proposition~21]{claesson_upper_2012}. We provide a simple generalization in Proposition \ref{prop:lim seq existence}.

\begin{definition}
  Let \(B\) and \(B'\) be collections of patterns. We say that \(B\) and \(B'\) are \emph{inv-Wilf-equivalent} if \begin{equation*}
    \av_n^k(B) = \av_n^k(B') \quad \text{for all } n \text{ and } k.
  \end{equation*}
\end{definition}

\begin{remark}
  A collection \(B\) of patterns is trivially inv-Wilf-equivalent to
  \begin{equation*}
    \{p^{-1} : p \in B\} \quad \text{and} \quad \{p^{\rc} : p \in B\}.
  \end{equation*}
  Inv-Wilf-equivalence was introduced by Dokos, Dwyer, Johnson, Sagan and Selsor \cite{dokos_permutation_2012}. Chan proved that nontrivial inv-Wilf-equivalences exist, even for principal classes \cite{chan_infinite_2015}. If \(B\) and \(B'\) are inv-Wilf-equivalent and \(B\) has a limit sequence, then \(B'\) obviously has the same limit sequence.
\end{remark}

\paragraph{132-avoiders and partitions}

 An essential tool in our analysis of limit sequences is the bijection between indecomposable \(132\)-avoiders and integer partitions, encoding the inversions of the permutation. The \emph{inversion table} or \emph{Lehmer code} \((b_1, \ldots, b_n)\) of a permutation \(\pi \in S_n\) is defined by
\begin{equation*}
  b_i = |\{j \in \{i+1, \ldots, n\} : \pi_i > \pi_j\}|
\end{equation*}
for all \(i \in [n]\). This mapping gives a bijection
\begin{equation*}
  S_n \longrightarrow \{0, \ldots, n-1\} \times \{0, \ldots, n-2\} \times \ldots \times \{0\}.
\end{equation*}

It is easy to show that the inversion table of a permutation \(\pi\) is weakly decreasing if and only if \(\pi\) avoids \(132\) \cite[Section~1.5]{stanley_enumerative_2012}. Therefore, the inversion table induces a mapping \(\Lambda\) from the \(132\)-avoiders to the integer partitions, where \(\Lambda(\pi)\) is the partition obtained from the inversion table of \(\pi\) by deleting the trailing zeroes. The restriction
\begin{equation*}
  \Lambda : \{\pi \in \Av(132) : \comp(\pi) = 1\} \longrightarrow \{\text{integer partitions}\}
\end{equation*}
is a bijection. If we fix \(k\) and \(n \geq k+1\), we therefore also have that
\begin{equation} \label{eq:132 partition bijection}
  \Lambda : \Av_n^k(132) \longrightarrow \{\text{integer partitions of } k\}
\end{equation}
is a bijection. If \(\lambda\) is a partition, we will write \(\Lambda^{-1}(\lambda)\) for the unique indecomposable \(132\)-avoider \(\pi\) such that \(\Lambda(\pi) = \lambda\). The following elementary result collects some facts about the connection between \(\pi\) and \(\lambda = \Lambda(\pi)\). 

\begin{lemma}\label{lem:132 partition properties}
  Let \(\pi\) be an indecomposable \(132\)-avoiding permutation and \(\lambda = \Lambda(\pi)\). 
  \begin{enumerate}[(a)]
    \item\label{lem:132 partition property def} For every \(i\), \(\pi_i\) is the smallest positive integer \(\ell\) such that \(\ell > \lambda_i\) and \(\ell \neq \pi_j\) for all \(j < i\).
    \item\label{lem:132 partition property equal} \(\lambda_i = \lambda_{i+1}\) if and only if \(\pi_i < \pi_{i+1}\).
    \item\label{lem:132 partition property drop} \(\lambda_i > \lambda_{i+1}\) if and only if \(\pi_{i+1} = \lambda_{i+1} + 1\).
  \end{enumerate}
\end{lemma}

Lastly, we show how to determine the limit sequence of \(1324\). Let \(n \geq k+2\). By \eqref{eq:comp inv length}, every permutation \(\pi \in \Av_n^k(1324)\) is decomposable, and may be written as \(\pi = \sigma \oplus {\id} \oplus \tau\), where \(\sigma\) is an indecomposable \(132\)-avoider and \(\tau\) is an indecomposable \(213\)-avoider. Note that \(\inv(\pi) = \inv(\sigma) + \inv(\tau)\). The mapping \(\pi \mapsto (\Lambda(\sigma), \Lambda(\tau^{\rc}))\) is a bijection between \(\Av_n^k(1324)\) and pairs \((\lambda, \mu)\) of integer partitions such that \(|\lambda| + |\mu| = k\). Therefore,
\begin{equation*}
  \av_n^k(1324) = \sum_{i=0}^k p(i) p(k-i)
\end{equation*}
for all \(n \geq k+2\), and hence \(C_{1324}(x) = P(x)^2\).

\section[\texorpdfstring{\(\{1324, 231\}\)}{\{1324, 231\}} is inversion monotone]{\texorpdfstring{\(\bm{\{1324, 231\}}\)}{\{1324, 231\}} is inversion monotone} \label{sec:1324 231}

The pattern \(1324\) contains three distinct patterns of length three: \(123\), \(132\) and \(213\). Furthermore, \(\{1324, 321\}\) is \emph{not} inversion monotone, since
\begin{equation*}
  \av_{10}^{15}(1324, 321) = 60 > 52 = \av_{11}^{15}(1324, 321).
\end{equation*} 
This is interesting in itself, since the decreasing patterns intuitively seem highly inversion monotone. Nevertheless, the only remaining pattern is \(231\) (and the symmetric \(312\)). We shall prove that the collection \(\{1324, 231\}\) is inversion
monotone by constructing an explicit injection 
\begin{equation*}
  \Av_n^k(1324, 231) \longrightarrow \Av_{n+1}^k(1324, 231).
\end{equation*}

\begin{lemma} \label{lem:213 231}
  Let \(\pi \in \Av_n^{k}(213, 231)\).
  \begin{enumerate}[(a)]
  \item \label{lem:213 231 insert} If \(\pi\) is indecomposable then \(\pi_1 = n\), and for each \(r \in \{0, \ldots, n\}\) there exists a unique
    \(\sigma \in \Av_{n+1}^{k+r}(213, 231)\) such that
    \(\sigma \delete i = \pi\) for some \(i\).
  \item \label{lem:213 231 delete} If \(i\) and \(j\) are such that \(\inv(\pi \delete i) = \inv(\pi \delete j)\), then \(\pi \delete i = \pi \delete j\).
  \end{enumerate}
\end{lemma}

\begin{figure}[!ht]
  \centering
  \begin{tikzpicture}[scale=0.95]
    \draw (0,0) rectangle (4,4);
    \begin{scope}
      \clip (0,0) rectangle (4,4);
      \draw[ultra thick] (0,0) -- (4,2) -- (0,4);
    \end{scope}

    \begin{scope}[shift={(4.75,-0.25)}, scale=0.5]
      \draw (0.5,0.5) rectangle (8.5,8.5);
      \foreach \x\y in {1/8,2/7,3/1,4/6,5/2,6/3,7/5,8/4}
        \node[dot] at (\x,\y) {};
      \foreach \x in {1,2,4,7} {
        \draw[dashed] (\x,8.5) -- (\x,0.5);
        \draw (\x,0.5) -- (\x,0.3);
      }
      \foreach \x\i in {0.5/0,1.5/1,3/2,5.5/3,8/4}
        \node at (\x,0) {\tiny\(+\i\)};
    \end{scope}
    
    \begin{scope}[shift={(9.75,-0.25)}, scale=0.5]
      \draw (0.5,0.5) rectangle (8.5,8.5);
      \foreach \x\y in {1/8,2/7,3/1,4/6,5/2,6/3,7/5,8/4}
        \node[dot] at (\x,\y) {};
      \foreach \x in {3,5,6,8} {
        \draw[dashed] (\x,8.5) -- (\x,0.5);
        \draw (\x,8.5) -- (\x,8.7);
      }
      \foreach \x\i in {7/5,5.5/6,4/7,1.75/8}
        \node at (\x,9) {\tiny\(+\i\)};
    \end{scope}
  \end{tikzpicture}
  \caption{The structure of a \(\{213, 231\}\)-avoider (left). The different ways to insert a point in the lower arm (middle) and in the upper arm (right), with the increase in inversions indicated.}
  \label{fig:213 231 structure}
\end{figure}
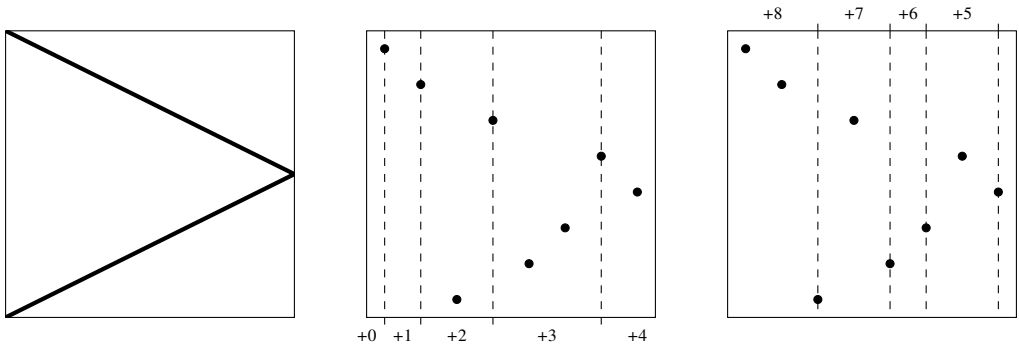

\begin{proof}
  The inverse permutation \(\pi^{-1}\) is \(\{213, 312\}\)-avoiding,
  thus unimodal.  Therefore \(\pi^{-1}\) (and \(\pi\)) is
  indecomposable if and only if \(\pi_n^{-1} = 1\), that is,
  \(\pi_1 = n\). Hence \(\pi\) is of the form indicated in Figure \ref{fig:213 231 structure} (left): the subsequence of entries \(\pi_i\) with \(\pi_i \leq \pi_n\) is increasing (the `lower arm'), and the subsequence of entries \(\pi_i\) with \(\pi_i > \pi_n\) is decreasing (the `upper arm').
  
  Denote the number of points in the upper (resp.\ lower) arm by
  \(a\) (resp.\ \(b\)), so that \(a + b = n\). The rightmost point is
  considered to be part of the lower arm.  There can be several ways to
  add a point to the lower arm in a given interval between two
  points in the upper arm, but each insertion yields the same permutation. A point added to the \(r\)-th such
  interval adds exactly \(r\) inversions, for
  \(r \in \{0, \ldots, a\}\); see Figure \ref{fig:213 231 structure} (middle).

  Similarly, adding a point to the upper arm in the \(r\)-th
  interval (from the right) between two points in the lower arm
  adds exactly \(a + r\) inversions, for \(r \in \{1, \ldots, b\}\); see Figure \ref{fig:213 231 structure} (right). This proves part \ref{lem:213 231 insert}. Part \ref{lem:213 231 delete} is similar: we delete a point instead of adding a point.
\end{proof}

\begin{figure}[p]
  \centering
  \begin{tikzpicture}[scale=0.8]
    \draw (-1.5,0) rectangle (1.5,3);
    \node at (0,1.5) {\(\Av \begin{pmatrix}
      1324, \\ 231
    \end{pmatrix}\)};
    \draw[decorate, very thick, decoration={calligraphic brace,amplitude=6pt}, yshift=-1mm] (-2,-0.5) -- (2,-0.5);
    \node at (-2,-1.2) {\(\epsilon\)};
    \node[rotate=-105] at (-2.25,-2) {\(\longmapsto\)};
    \node[dot] at (-2.42, -2.7) {};

    \begin{scope}[shift={(1,-3)}]
      \draw (0,0) rectangle (2,2);
      \node at (1,1) {\(n \geq 1\)};
      \draw[decorate, very thick, decoration={calligraphic brace,amplitude=6pt}, yshift=-1mm] (-1.5,-0.5) -- (3.5,-0.5);  
    \end{scope}

    \begin{scope}[shift={(-1.5,-6)}]
      \begin{scope}
        \clip[draw] (0,0) rectangle (2,2);
        \draw[ultra thick] (0,2) -- (2,0);
      \end{scope}
      \node[rotate=-100] at (0.2, -0.77) {\(\longmapsto\)};
    \end{scope}

    \begin{scope}[shift={(-3,-9.7)}]
      \begin{scope}
        \clip[draw] (0,0) rectangle (2,2);
        \draw[ultra thick] (0,2) -- (2,0);
      \end{scope}
      \node[dot] at (2.15,2.15) {};
      \node at (-1,1) {(1)};
    \end{scope}

    \begin{scope}[shift={(2.5,-8)}]
      \draw (0,0) rectangle (4,4);
      \draw
        (4/3,0) -- (4/3,4)
        (8/3,0) -- (8/3,4)
        (0,4/3) -- (4,4/3)
        (0,8/3) -- (4,8/3);

      \begin{scope}
        \clip (0,4) rectangle (4/3,8/3);
        \draw[ultra thick] (0,4) -- (4/3,8/3);
      \end{scope}

      \begin{scope}
        \clip (4/3,0) rectangle (8/3,4/3);
        \draw[ultra thick] (4/3,4/3) -- (2,0) -- (8/3,4/3);
      \end{scope}

      \begin{scope}
        \clip (8/3,4/3) rectangle (4,8/3);
        \draw[ultra thick] (8/3,4/3) -- (4,2) -- (8/3,8/3);
      \end{scope}

      \draw (4.1,4/3) -- (4.2,4/3) -- (4.2,8/3) -- (4.1,8/3);
      \node[anchor=west] at (4.15,2) {\scriptsize\(m > 0\)};

      \draw[decorate, very thick, decoration={calligraphic brace,amplitude=6pt}, yshift=-1mm] (-0.5,-0.5) -- (4.5,-0.5);      
    \end{scope}
    
    \begin{scope}[shift={(0,-12)}]
      \begin{scope}
        \clip[draw] (0,0) rectangle (1.5,1.5);
        \draw[ultra thick] (0,1.5) -- (0.75,0) -- (1.5,1.5);
      \end{scope}

      \begin{scope}
        \clip[draw] (1.5,1.5) rectangle (3,3);
        \draw[ultra thick] (1.5,1.5) -- (3,2.25) -- (1.5,3);
      \end{scope}

      \node[rotate=-95] at (1,-0.77) {\(\longmapsto\)};
    \end{scope}

    \begin{scope}[shift={(-1,-16.8)}]
      \draw (1.8,1.8) rectangle (3.3,3.3);
      \node[dot] at (1.65,1.65) {};

      \begin{scope}
        \clip[draw] (0,0) rectangle (1.5,1.5);
        \draw[ultra thick] (0,1.5) -- (0.75,0) -- (1.5,1.5);
      \end{scope}

      \begin{scope}
        \clip[draw] (1.8,1.8) rectangle (3.3,3.3);
        \draw[ultra thick] (1.8,1.8) -- (3.3,2.55) -- (1.8,3.3);
      \end{scope}
      
      \node at (-1,1.65) {(2)};
    \end{scope}
    
    \begin{scope}[shift={(5,-13)}]
      \draw (0,0) rectangle (4,4);
      \draw
        (4/3,0) -- (4/3,4)
        (8/3,0) -- (8/3,4)
        (0,4/3) -- (4,4/3)
        (0,8/3) -- (4,8/3);

      \begin{scope}
        \clip (0,4) rectangle (4/3,8/3);
        \draw[ultra thick] (0,4) -- (4/3,8/3);
      \end{scope}

      \begin{scope}
        \clip (4/3,0) rectangle (8/3,4/3);
        \draw[ultra thick] (4/3,4/3) -- (2,0) -- (8/3,4/3);
      \end{scope}

      \begin{scope}
        \clip (8/3,4/3) rectangle (4,8/3);
        \draw[ultra thick] (8/3,4/3) -- (4,2) -- (8/3,8/3);
      \end{scope}

      \draw 
        (-0.1,8/3) -- (-0.2,8/3) -- (-0.2,4) -- (-0.1,4)
        (4.1,4/3) -- (4.2,4/3) -- (4.2,8/3) -- (4.1,8/3);
      \node[anchor=east] at (-0.15,10/3) {\scriptsize\(0 < \ell\)};
      \node[anchor=west] at (4.15,2) {\scriptsize\(m > 0\)};

      \node[rotate=-90] at (2,-0.77) {\(\longmapsto\)};
    \end{scope}

    \begin{scope}[shift={(5,-18.5)}]
      \draw (0,0) rectangle (4,4);
      \draw (0,0) rectangle (4,4);
      \draw
        (4/3,0) -- (4/3,4)
        (8/3,0) -- (8/3,4)
        (0,4/3) -- (4,4/3)
        (4/3-0.5,0) -- (4/3-0.5,4)
        (0,4/3+0.5) -- (4,4/3+0.5)
        (0,8/3+0.5) -- (4,8/3+0.5);

      \begin{scope}
        \clip (0,4) rectangle (4/3-0.5,8/3+0.5);
        \draw[ultra thick] (0,4) -- (4/3-0.5,8/3+0.5);
      \end{scope}

      \begin{scope}
        \clip (4/3-0.5,4/3) rectangle (4/3,4/3+0.5);
        \draw[ultra thick] (4/3-0.5,4/3+0.5) -- (4/3,4/3);
      \end{scope}

      \begin{scope}
        \clip (4/3,0) rectangle (8/3,4/3);
        \draw[ultra thick] (4/3,4/3) -- (2,0) -- (8/3,4/3);
      \end{scope}

      \begin{scope}
        \clip (8/3,4/3+0.5) rectangle (4,8/3+0.5);
        \draw[ultra thick] (8/3,4/3+0.5) -- (4,2.5) -- (8/3,8/3+0.5);
        \node[dot] at (10/3,2.5) {};
       \end{scope}

      \draw
        (4.1,4/3+0.5) -- (4.2,4/3+0.5) -- (4.2,8/3+0.5) -- (4.1,8/3+0.5)
        (-0.1,4/3) -- (-0.2,4/3) -- (-0.2,4/3+0.5) -- (-0.1,4/3+0.5)
        (-0.1,8/3+0.5) -- (-0.2,8/3+0.5) -- (-0.2,4) -- (-0.1,4);
      \node[anchor=east] at (-0.15,4/3+0.25) {\scriptsize\(q\)};
      \node[anchor=east] at (-0.15,10/3+0.25) {\scriptsize\(\ell - q\)};
      \node[anchor=west] at (4.15,2.5) {\scriptsize\(m+1\)};
      \node[anchor=north] at (2,-0.5) {\parbox{6.7cm}{Compute \(\ell = q(m+1) - r\), where \(r \leq m\). Insert a point in the rightmost component so that \(\ell + r\) inversions are created, then shift down \(q\) points from the top left.}};
      \node at (-1.3,2) {(3)};
    \end{scope}
  \end{tikzpicture}
  \vspace{5mm}
  \caption{A schematic of the injection \(\Av_n^k(1324, 231) \to \Av_{n+1}^k(1324, 231)\).}
  \label{fig:1324 231 inj}
\end{figure}

\begin{theorem} \label{thm:1324 231 inv mono}
  The collection \(\{1324, 231\}\) is inversion monotone.
\end{theorem}

\begin{proof}
  We start by analyzing the structure of an arbitrary permutation \(\pi \in \Av_n(1324, 231)\). Write \(\pi = L n R\), where \(L\) (resp.\ \(R\)) is the subsequence of entries to the left (resp.\ right) of \(n\). Every entry of \(L\) must be smaller than every entry of \(R\), since \(\pi\) avoids \(231\). Furthermore, \(L \in \Av(132, 231)\), and if \(L\) is nonempty then \(R \in \Av(213, 231)\). If \(L\) is empty (i.e.\ \(n\) is the first entry of \(\pi\)), then \(R\) is an arbitrary \(\{1324, 231\}\)-avoider. We get the following recursive characterization: \(\pi \in S_n\) avoids \(\{1324, 231\}\) if and only if \(n=0\), or \(\pi_1 = n\) and \(\pi \delete n \in \Av(1324, 231)\), or \(\pi = \pi^{(1)} \oplus \pi^{(2)}\) where \(\pi^{(1)} \in \Av(132, 231)\) and \(\pi^{(2)} \in \Av(213, 231)\). If \(\pi^{(2)}\) is indecomposable then it begins with its largest entry, which is the entry \(n\) of \(\pi\).

  Hence, if \(\pi \in \Av_n(1324, 231)\) is indecomposable and not equal to \(\id_n^{\rev}\), then there is some positive integer \(\ell < n\) such that \(n, n-1, \ldots, n-\ell+1\) appear consecutively in decreasing order as the first \(\ell\) elements of \(\pi\), and the remaining \(n - \ell\) elements form an arbitrary decomposable \(\{1324, 231\}\)-avoider. We will refer to the last component of this decomposable permutation as the last component of \(\pi\). Observe that the component avoids \(\{213, 231\}\).

  With these preliminaries in mind, we can define a mapping from \(\Av_n^{k}(1324, 231)\) to \(\Av_{n+1}^{k}(1324, 231)\) according to Figure \ref{fig:1324 231 inj}, which we shall prove is injective. For branches (1) and (2), the injection is straightforward. In branch (3), the mapping is defined as follows for a permutation \(\pi \in \Av(1324, 231)\).
  \begin{itemize}
    \item There are unique integers \(q > 0\) and \(0 \leq r \leq m\) such that \(\ell = q(m+1) - r\), where \(\ell > 0\) is as before and \(m\) is the number of points in the last component of \(\pi\).
    \item Insert a point into the last component of \(\pi\) so that its \(\{213, 231\}\)-avoidance is maintained and \(r\) inversions are created within the component, as in Lemma \ref{lem:213 231} \ref{lem:213 231 insert}. The new point also creates inversions with the \(\ell\) first points of the permutation, so in total, \(\ell + r\) new inversions are created. 
    \item Shift the last \(q\) of the \(\ell\) first points down by \(m+1\) steps each. Each such shift removes \(m+1\) inversions, so a total of \(q(m+1)\) inversions are removed. Since \(\ell + r = q(m+1)\), the resulting permutation has equally many inversions as \(\pi\).
  \end{itemize}
  Here, \emph{shifting} an entry \(e\) of a permutation \(\pi\) down (resp.\ up) by one step means mapping \(\pi \mapsto (e\ e-1) \pi\) (resp.\ \(\pi \mapsto (e\ e+1) \pi\)), where \((i\ j)\) is the transposition swapping \(i\) and \(j\). Shifting \(e\) down by \(m\) steps is the composition
  \begin{equation*}
    \pi \longmapsto (e-m+1\ e-m) \cdots (e-1\ e-2)(e\ e-1) \pi.
  \end{equation*}
  When shifting down several entries, start from the smallest and proceed in increasing order. When shifting up, start from the largest.

  Let \(F\) denote our mapping \(\Av_n^k(1324, 231) \to \Av_{n+1}^k(1324, 231)\). The restriction of \(F\) to branches (1) and (2), i.e.\ to the set
  \begin{equation*}
    \left\{\pi \in \Av_n^{k}(1324, 231) : \comp(\pi) \geq 2 \text{ or } \pi = \id_n^{\rev}\right\}
  \end{equation*}
  is clearly injective, and its image is
  \begin{equation*}
    \left\{\pi \in \Av_{n+1}^{k}(1324, 231) : \comp(\pi) \geq 3 \text{ or } \pi = {\id_n^{\rev}} \oplus 1\right\}. 
  \end{equation*}
  It suffices to show that \(F\) restricted to branch (3) is
  injective, and that its image is disjoint from the above. Denote
  \begin{equation*}
    \mathcal{I}_n^{k} = \left\{\pi \in \Av_n^{k}(1324, 231) : \comp(\pi) = 1\right\} \setminus \left\{\id_n^{\rev}\right\},
  \end{equation*}
  and let \(\sigma \in F(\mathcal{I}_n^{k})\). We want to construct
  the inverse of \(F|_{\mathcal{I}_n^{k}}\). To do this, we need to
  recover the numbers \(m\), \(q\), and \(r\) used to construct
  \(\sigma\). Firstly:
  \begin{itemize}
  \item \(\ell - q\) is the smallest nonnegative integer such that
    \(\sigma_{\ell-q+1} \sigma_{\ell-q+2} \ldots \sigma_{n+1}\) is
    decomposable. For now, denote \(\ell' = \ell - q\).
  \item \(m + 1 = n + 1 - \sigma_{\ell'+1} - \ell'\). This is
    the number of entries larger than \(\sigma_{\ell'+1}\), excluding
    \(\sigma_1, \ldots, \sigma_{\ell'}\).
  \end{itemize}
  It remains to determine the value of \(q\), the number of points that were shifted down as the last step in the construction of \(\sigma\). Observe that there exists some largest positive integer \(q'\) such that the entries \(\sigma_{\ell' + 1}, \sigma_{\ell' + 2}, \ldots, \sigma_{\ell' + q'}\) are decreasing and have consecutive values, since we know that \(q > 0\). The first of these entries, \(\sigma_{\ell' + 1}\), is equal to \(\sigma_{\ell'} - m - 2\) (if \(\ell' = 0\), set \(\sigma_{\ell'} = n+2\)). Call these \(q'\) entries \emph{eligible}; we must have \(q \leq q'\).

  Since \(\sigma\) was constructed using \(F\), there must exist a positive integer \(q\) such that after shifting the first \(q\) eligible points up by \(m+1\) steps, the resulting permutation \(\sigma'\) satisfies
  \begin{equation} \label{eq:q and r}
    \inv(\sigma') = k + \ell + r = k + \ell' + q + r,
  \end{equation}
  where \(0 \leq r \leq m\). On the other hand, since shifting the first eligible point up by \(m + 1\) steps creates \(m + 1\) new inversions, we have \(\inv(\sigma') = k + q(m + 1)\). Combining these two identities yields
  \begin{equation*}
    \ell' \leq qm \leq \ell' + m.
  \end{equation*}
  If there is a unique candidate value \(q\) satisfying the inequalities, then we are done: \(r\) is given by \eqref{eq:q and r}. 
  
  Otherwise, there are exactly two candidate values \(q_1 < q_2\) with \(q_2 = q_1 + 1\) for \(q\), corresponding to \(r = 0\) and \(r = m\), respectively. However, these two cases can be distinguished: \(r = 0\) if and only if the last \(m+1\) points of \(\sigma\) form a decomposable permutation. In conclusion, the parameters \(m\), \(q\) and \(r\) are uniquely determined.

  Finally, by Lemma \ref{lem:213 231} \ref{lem:213 231 delete} there exists a unique permutation \(\pi \in \mathcal{I}_n^{k}\) such that \(\sigma' \delete e = \pi\) for some \(e\) among the last \(m+1\) entries of \(\sigma'\). It is clear that \(\pi\) is the unique preimage of \(\sigma\) under \(F|_{\mathcal{I}_n^{k}}\), so \(F|_{\mathcal{I}_n^k}\) is injective.

  It remains to show that \(F\) as a whole is injective.
  \begin{itemize}
  \item Each permutation \(\sigma \in F(\mathcal{I}_n^{k})\) has at
  most two components. This is clear whenever the inserted point in branch (3) adds at least one inversion to the last component. If it adds zero inversions to the last component, that is, \(r = 0\), then \(\ell = q(m+1) \geq 2q\), so \(\sigma_1 = n+1\) and \(\comp(\sigma) = 1\). 
  \item \({\id_n^{\rev}} \oplus 1 \notin F(\mathcal{I}_n^{k})\), since every decomposable permutation in the image of branch (3) has a final component of size at least two. \qedhere
  \end{itemize}
\end{proof}

\begin{figure}[tbp]
  \centering
  \begin{tikzpicture}[scale=0.5]
    \draw (0,0) rectangle (12,12);
    \draw
      (5,0) -- (5,12)
      (10,0) -- (10,12)
      (0,5) -- (12,5)
      (0,7) -- (12,7);

    \foreach \x\y in {1/12,2/11,3/10,4/9,5/8,6/5,7/3,8/1,9/2,10/4,11/7,12/6} 
      \node[dot] at (\x-0.5,\y-0.5) {}; 
    
    \begin{scope}[shift={(14,0)}, scale=12/13]
      \draw (0,0) rectangle (13,13);
      \draw
        (3,0) -- (3,13)
        (5,0) -- (5,13)
        (10,0) -- (10,13)
        (0,5) -- (13,5)
        (0,7) -- (13,7)
        (0,10) -- (13,10);

      \foreach \x\y in {1/13,2/12,3/11,4/7,5/6,6/5,7/3,8/1,9/2,10/4,11/10,12/8,13/9}
        \node[dot] at (\x-0.5,\y-0.5) {};
      \foreach \x\y in {4/7,5/6,12/8}
        \node[circ] at (\x-0.5,\y-0.5) {};
    \end{scope}
  \end{tikzpicture}
  \vspace{2mm}
  \caption{A \(\{1324, 231\}\)-avoiding permutation \(\pi\) (left) and its image \(F(\pi)\) (right).}
  \label{fig:1324 231 inj example}
\end{figure}
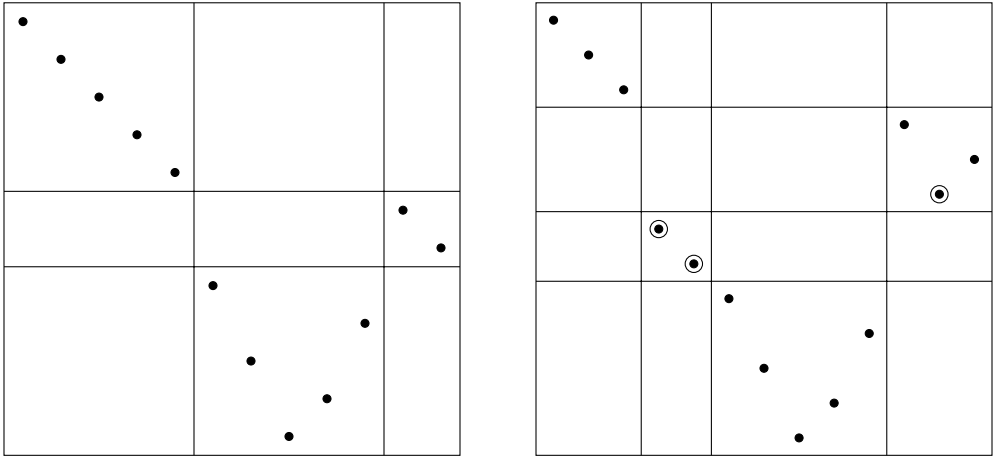

\begin{example}
  Let \(\pi\) be as in Figure \ref{fig:1324 231 inj example} (left). We have \(n = 12\) and \(k = 52\). The last component of \(\pi\) is \(21\), so \(m = 2\). The number of consecutive decreasing entries at the start of \(\pi\) is \(\ell = 5\). Hence, we get \(\ell = q(m+1) - r\) with \(q = 2\) and \(r = 1\). 
  
  As per Figure \ref{fig:1324 231 inj}, we insert a point into the last component of \(\pi\) so that it creates \(r = 1\) inversion; the only way to do this results in the permutation \(312\). This new point also creates \(\ell = 5\) inversions with the first five entries of \(\pi\), so in total, \(6\) new inversions are created. Finally, we shift down the last \(q = 2\) of the first \(\ell = 5\) entries of \(\pi\) by \(m+1 = 3\) steps each. The resulting permutation is \(F(\pi)\), shown in Figure \ref{fig:1324 231 inj example} (right). It has exactly \(k = 52\) inversions. The three circled points in the figure are the two points that were shifted down, and the new point that was inserted.
\end{example}

\section{Building inversion-monotone sets} \label{sec:building bases}

We say that a collection \(B\) is \emph{trivially} inversion monotone if
\begin{enumerate}[(a)]
  \item \(p_1 \neq 1\) for all \(p \in B\), or
  \item \(p_m \neq m\) for all \(p \in B\), where \(m = |p|\).
\end{enumerate}
In this section, we will demonstrate how to build a nontrivially inversion monotone set \(B'\) from an inversion monotone set \(B\), such that
\begin{equation*}
  \min \{|p| : p \in B'\} = \min \{|p| : p \in B\} + 1. 
\end{equation*} 
For a collection \(B\) of patterns, we define \(B^{\ell}\), \(B^r\), \(B^u\) and \(B^d\) as follows:
\begin{align*}
  B^{\ell} &= \left\{p : p \delete p_1 \in B\right\}, \\
  B^r &= \left\{p : p \delete p_{|p|} \in B\right\}, \\
  B^u &= \left\{p : p \delete |p| \in B\right\}, \\
  B^d &= \left\{p : p \delete 1 \in B\right\}.
\end{align*}
In other words, \(B^{\ell}\) consists of the patterns obtained by inserting a new first entry into each pattern in \(B\) in all possible ways. The other three collections are similar.

\begin{proposition} \label{prop:building B}
  If \(B\) is inversion monotone, then so are \(B^{\ell}\), \(B^r\), \(B^u\) and \(B^d\).
\end{proposition}

\begin{proof}
  It suffices to prove that \(B^{\ell}\) is inversion monotone, as the other cases are symmetric. Since \(B\) is inversion monotone, there is an injection \(f : \Av(B) \to \Av(B)\) such that \(|f(\pi)| = |\pi| + 1\) and \(\inv(f(\pi)) = \inv(\pi)\) for all \(\pi \in \Av(B)\). We will first show that if \(\pi \in \Av(B^{\ell})\), then \(\pi \delete \pi_1 \in \Av(B)\). Indeed, if \(\pi \delete \pi_1\) contains a pattern \(p \in B\), then obviously this occurrence of \(p\) in \(\pi\) together with \(\pi_1\) forms a pattern \(q\) such that \(q \delete q_1 = p\). So \(q \in B^{\ell}\) by construction, meaning that \(\pi\) does not avoid \(B^{\ell}\).

  Hence, we can define a map \(g : \Av(B^{\ell}) \to \Av(B^{\ell})\) by setting \(g(\pi)_1 = \pi_1\) and \(g(\pi) \delete \pi_1 = f(\pi \delete \pi_1)\). Observe that since \(f(\pi \delete \pi_1) \in \Av(B)\), we must have \(g(\pi) \in \Av(B^{\ell})\) by the same argument as above. Clearly \(g\) is injective, \(|g(\pi)| = |\pi| + 1\) and \(\inv(g(\pi)) = \inv(\pi)\), so \(B^{\ell}\) is inversion monotone.
\end{proof}

The result is somewhat surprising, since \(B^{\ell}\) is much larger than \(B\). If we let \(B_0 = B \subseteq S_m\) (for example with \(B = \{213\} \subseteq S_3\)) and iteratively \(B_{i+1} = B_i^{\ell}\), then
\begin{equation*}
  |B_i| = (m+i) \cdot |B_{i-1}| = \frac{(m+i)!}{m!} |B|. 
\end{equation*}

\begin{example}
  Since \(213\) is inversion monotone, so are
  \begin{align*}
    \{213\}^{\ell} &= \{1324, 2314, 3214, 4213\}, \\
    \{213\}^r &= \{3241, 3142, 2143, 2134\}, \\
    \{213\}^u &= \{4213, 2413, 2143, 2134\}, \\
    \{213\}^d &= \{1324, 3124, 3214, 3241\}.
  \end{align*}
  Note in particular the collections \(B^{\ell}\) and \(B^d\), which contain \(1324\). We record this as a proposition.
\end{example}

\begin{proposition} \label{prop:bb inv mono 1324}
  The collection \(\{1324, 2314, 3214, 4213\}\) is inversion monotone.
\end{proposition}

Inspired by the approach in Proposition \ref{prop:building B} and the discussion right after, one is tempted to try to fix a pattern in the basis and increase the length of the rest. In other words, we would fix some pattern \(p\) and collection \(B_0\), and iteratively set \(B_{i+1}\) to be either \(B_i^{\ell}\), \(B_i^r\), \(B_i^u\) or \(B_i^d\). Clearly, if \(\{p\} \cup B_i\) is inversion monotone for all \(i\), then \(p\) is inversion monotone: we have
\begin{equation*}
  \av_n^k(p) = \av_n^k(\{p\} \cup B_{n+2}) \leq \av_{n+1}^k(\{p\} \cup B_{n+2}) = \av_{n+1}^k(p)
\end{equation*}
for all \(n\) and \(k\). This is a potential avenue towards the inversion monotonicity conjecture for \(1324\). We have computational support for the following claim.

\begin{conjecture} \label{conj:building bases}
  Let \(B_0 = \{213\}\) and \(B_{i+1} = B_i^{\ell}\) for all \(i \geq 0\). The collection \(\{1324\} \cup B_i\) is inversion monotone for all \(i\). In particular, \(1324\) is inversion monotone.
\end{conjecture}

\section{Limit sequences} \label{sec:1324 pairs}

In this section, we develop some basic theory for limit sequences of permutation classes. We characterize the finite collections that have limit sequences, and observe that the limit sequence of a finite collection containing \(1324\) can be described in terms of indecomposable \(132\)- and \(213\)-avoiders. When every other pattern in the collection is indecomposable, this description becomes a product formula. We apply these results to determine the limit sequences of all pairs \(\{1324, p\}\) with \(p \in S_4\).

\subsection{Basic theory} \label{subsec:limit basic theory}

Claesson, Jelínek and Steingrímsson showed that an individual pattern has a limit sequence if and only if it has at most one inversion \cite[Proposition~21]{claesson_upper_2012}. The following result extends their characterization to finite collections.

\begin{proposition} \label{prop:lim seq existence}
  A finite collection \(B\) of patterns has a limit sequence if and only if \(B\) contains a pattern \(p\) such that \(\inv(p) \leq 1\).
\end{proposition}

\begin{proof}
  Suppose first that all patterns in \(B\) have at least two inversions. Then clearly
  \begin{equation*}
    \Av_n^1(B) = \{\pi \in S_n : \inv(\pi) = 1\},
  \end{equation*}
  so \(\av_n^1(B) = n - 1\) and \(B\) does not have a limit sequence.

  Conversely, suppose that \(B\) contains a pattern \(p\) with at most one inversion. If \(\inv(p) = 0\), then \(p = \id_m\) for some \(m\). Every permutation \(\pi \in S_n\) with \(\inv(\pi) = k\) and \(n \geq k + m\) has at least \(m\) components by \eqref{eq:comp inv length}, and hence contains \(\id_m\). It follows that \(\av_n^k(B) = 0\) for all \(n \geq k + m\).

  It remains to consider the case \(\inv(p) = 1\). Write \(p = {\id_a} \oplus 21 \oplus {\id_b}\), where \(a, b \geq 0\). We will show that \(\av_n^k(B) = \av_{n+1}^k(B)\) for all
  \begin{equation} \label{eq:inv bound}
    n \geq k + a + b + \max \{|q| : q \in B\}.
  \end{equation}
  Indeed, any permutation \(\pi \in S_n\) with \(\inv(\pi) = k\) satisfying \eqref{eq:inv bound} must have
  \begin{equation*}
    \comp(\pi) \geq a + b + \max \{|q| : q \in B\}
  \end{equation*}
  by \eqref{eq:comp inv length}. Write \(\pi = \pi^{(1)} \oplus \pi^{(2)} \oplus \ldots \oplus \pi^{(r)}\), where each \(\pi^{(i)}\) is indecomposable. If \(\pi\) avoids \(B\) (in particular, \(\pi\) avoids \(p\)), then we must have \(\pi^{(i)} = 1\) for all \(i\) such that \(a + 1 \leq i \leq r - b\). In other words,
  \begin{equation*}
    \pi = \pi^{(1)} \oplus \ldots \oplus \pi^{(a)} \oplus {\id_m} \oplus \pi^{(r-b+1)} \oplus \ldots \oplus \pi^{(r)},
  \end{equation*}
  where
  \begin{equation*}
    m \geq r - a - b \geq \max \{|q| : q \in B\}.
  \end{equation*}
  Hence, we can define a bijection \(\Av_n^k(B) \to \Av_{n+1}^k(B)\) by
  \begin{equation*}
    \pi \longmapsto \pi^{(1)} \oplus \ldots \oplus \pi^{(a)} \oplus {\id_{m+1}} \oplus \pi^{(r-b+1)} \oplus \ldots \oplus \pi^{(r)} \eqqcolon \sigma.
  \end{equation*}
  Observe that \(\sigma\) avoids \(B\), since an occurrence of a pattern \(q \in B\) in \(\sigma\) would have to use all of the entries of \(\id_{m+1}\), contradicting the fact that \(m + 1 > |q|\). This concludes the proof.
\end{proof}

The next result shows that the limit sequence (excluding the first term) can be \(0, 0, 0, \ldots\) even if \(B\) does not contain an identity pattern.

\begin{proposition} \label{prop:21 id}
  We have that
  \begin{equation*}
    \av_n^k \left({\id_a} \oplus 21, 21 \oplus {\id_b}\right) = 0
  \end{equation*}
  for all \(k \geq 1\) and \(n \geq k + a + b\). In particular, \(\{{\id_a} \oplus 21, 21 \oplus {\id_b}\}\) is not inversion monotone for any \(a, b \geq 1\).
\end{proposition}

\begin{remark}
  Thus, the pair in Proposition \ref{prop:21 id} has limit sequence \(1,0,0,\ldots .\) For example, taking \(a = b = 2\) gives the pair \(\{1243, 2134\}\).
\end{remark}

\begin{proof}[Proof of Proposition \ref{prop:21 id}]
  Suppose that \(k \geq 1\) and \(n \geq k + a + b\). Every permutation \(\pi \in \Av_n^k\left({\id_a} \oplus 21, 21 \oplus {\id_b}\right)\) satisfies \(\comp(\pi) \geq n - k \geq a + b\) by \eqref{eq:comp inv length}. Decompose
  \begin{equation*}
    \pi = \pi^{(1)} \oplus \pi^{(2)} \oplus \ldots \oplus \pi^{(c)},
  \end{equation*}
  and suppose that a component \(\pi^{(i)}\) contains an inversion. We must have \(i \leq a\), or
  \begin{equation*}
    \pi^{(1)} \oplus \pi^{(2)} \oplus \ldots \oplus \pi^{(a)} \oplus \pi^{(i)}
  \end{equation*}
  contains \({\id_a} \oplus 21\). Similarly, we must have \(i \geq c - b + 1\), or
  \begin{equation*}
    \pi^{(i)} \oplus \pi^{(c-b+1)} \oplus \pi^{(c-b+2)} \oplus \ldots \oplus \pi^{(c)}
  \end{equation*}
  contains \(21 \oplus {\id_b}\). But then \(c \leq i + b - 1 \leq a + b - 1\), a contradiction.
\end{proof}

We close this subsection with a general principle that is used to compute every limit sequence in this paper. Recall that by Lemma \ref{lem:1324 decomp form}, a decomposable \(1324\)-avoider is determined by a pair \((\sigma, \tau)\), where \(\sigma\) is an indecomposable \(132\)-avoider and \(\tau\) is an indecomposable \(213\)-avoider. This gives a general description of the limit sequence of any finite collection containing \(1324\).

\begin{proposition} \label{prop:limit factorization}
  Let \(B\) be a finite collection of patterns with \(1324 \in B\), and put \(N = \max \{|q| : q \in B\}\). If \(n \geq k + N + 2\), then \(\pi = \sigma \oplus {\id_m} \oplus \tau \mapsto (\sigma, \tau)\) is a bijection between \(\Av_n^k(B)\) and the set of pairs \((\sigma, \tau)\) such that
  \begin{enumerate}[(a)]
    \item \label{lim fact:shape} \(\sigma\) is an indecomposable \(132\)-avoider and \(\tau\) is an indecomposable \(213\)-avoider;
    \item \label{lim fact:inv} \(\inv(\sigma) + \inv(\tau) = k\); and
    \item \label{lim fact:split} for every \(p \in B\) and every way of writing \(p = \alpha \oplus {\id_\ell} \oplus \beta\) with \(\ell \geq 0\) and \(\alpha, \beta\) possibly empty, \(\sigma\) avoids \(\alpha\) or \(\tau\) avoids \(\beta\).
  \end{enumerate}
\end{proposition}

\begin{proof}
  This is the bijection of \cite[Proposition~15]{claesson_upper_2012} restricted to \(\Av(B)\). It remains to prove that, for \(m \geq N\), the permutation \(\pi = \sigma \oplus {\id_m} \oplus \tau\) with \(\sigma\) and \(\tau\) indecomposable avoids a pattern \(p\) with \(|p| \leq N\) if and only if condition \ref{lim fact:split} holds for \(p\).

  To this end, observe that an occurrence of \(p\) in \(\pi\) decomposes as \(p = \alpha \oplus {\id_\ell} \oplus \beta\), where the three blocks (each possibly empty) are contained within the blocks \(\sigma\), \(\id_m\) and \(\tau\) of \(\pi\), respectively. We require only that \(\ell \leq m\), which holds by our assumption:
  \begin{equation*}
    m \geq \comp(\pi) - 2 \geq n - k - 2 \geq N \geq \ell.
  \end{equation*}
  The second inequality uses \eqref{eq:comp inv length}.
\end{proof}

Condition \ref{lim fact:split} is automatic for \(1324\) by condition \ref{lim fact:shape}. When every other pattern is indecomposable, it simplifies considerably for the remaining patterns. Indeed, the only decompositions of an indecomposable pattern \(p\) of length at least two are
\begin{equation*}
  p = p \oplus \epsilon \oplus \epsilon
  \quad \text{and} \quad
  p = \epsilon \oplus \epsilon \oplus p,
\end{equation*}
where \(\epsilon\) denotes the empty permutation. Thus, the condition asks precisely that both \(\sigma\) and \(\tau\) avoid \(p\). The limit generating function factors.

\begin{corollary} \label{cor:limit factorization}
  Let \(B\) be a finite collection of patterns. If every pattern in \(B\) is indecomposable, then
  \begin{equation*}
    C_{1324, B}(x) = C_{132, B}(x) \cdot C_{213, B}(x) = C_{132, B}(x) \cdot C_{132, B^{*}}(x),
  \end{equation*}
  where \(B^{*} = \left\{(p^{-1})^{\rc} : p \in B\right\}\).
\end{corollary}

\begin{proof}
  See the preceding discussion.
\end{proof}

\subsection{Pairs containing 1324} \label{subsec:limit seqs}

\paragraph{Limit sequences} 

We now apply the preceding results to pairs \(\{1324, p\}\) with \(p \in S_4\). Inversion-preserving symmetries reduce the \(23\) such pairs to the following twelve representatives for \(p\):
\begin{equation} \label{eq:1324 pairs}
  1234, 1243, 1342, 1432, 2143, 2341, 2413, 2431, 3412, 3421, 4231, 4321.
\end{equation}
These twelve pairs have distinct limit sequences. The pair \(\{1324, 1234\}\) has limit sequence \(0, 0, 0, \ldots,\) so we omit it from what follows.

Among all \(\binom{24}{2} = 276\) pairs of distinct patterns in \(S_4\), \(215\) are trivially inversion monotone and \(33\) are not inversion monotone. The remaining \(28\) pairs are open, and all but nine of them are of the form \(\{1324, p\}\). Eight of the pairs \(\{1324, p\}\) are shown to be inversion monotone whenever \(n \geq \frac{k+7}{2}\) in Section \ref{sec:almost decomp}. Table S1 in the supplementary material for this article contains a status table for all pairs, and Table S2 shows where inversion monotonicity breaks for the negative cases. For the representatives in \eqref{eq:1324 pairs}, the relevant status and limit sequences are summarized in Table \ref{tab:lim seqs}.

Several of these classes have been previously enumerated: \(p = 4231\) \cite{albert_counting_2009}, \(4321\) \cite{vatter_finding_2012}, \(2341\) \cite{miner_enumeration_2016}, \(2431\) \cite{albert_inflations_2014}, \(3412\) \cite{albert_enumeration_2011} and \(3421\) \cite{albert_enumeration_2012}, while \(\Av(1324, 2143)\) is the class of smooth permutations (after reversal) \cite{lakshmibai_criterion_1990} and is equinumerous with \(\Av(1324, 2413)\) \cite{bona_smooth_1998}; \href{https://permpal.com}{PermPAL} \cite{combinatorial-exploration} is a useful resource for all of them. The Wilf-equivalence classification for the pairs of length-four patterns was completed by Le \cite{le_wilf_2005}.

The generating functions for the limit sequences are collected in the following theorem. The less familiar partition classes occurring in the statement are described partly in the proof and partly in Section \ref{sec:132 indecomposable}. Let \(P(x)\) be the generating function for the number of integer partitions.

\begin{table}[t]
  \centering
  \small
  \caption{Limit sequences of pairs containing \(1324\). Here `half' means inversion monotone whenever \(n \geq \frac{k+7}{2}\); all entries marked `open' have been checked computationally for \(n \leq 15\).}
  \begin{booktabs}{rowsep=2pt}
    \toprule
    Pair & Inv-monotone & Limit sequence \\ \midrule
    \(1243\) & No & Partition numbers \\
    \(2143\) & No & \(2 p(k)\) for \(k \geq 1\) \\ \midrule
    \(1342\) & Half & Overpartitions \\
    \(1432\) & Half & Partitions with a divider \\
    \(4231\) & Half & \((\text{g.f.\ of convex partitions})^2\) \\
    \(4321\) & Half & \((\text{g.f.\ of partitions with at most two distinct parts})^2\) \\ \midrule
    \(2341\) & Open & \((\text{g.f.\ of the sand pile model})^2\) \\
    \(2413\) & Open & \(P(x)^2\) \\
    \(2431\) & Open & \(P(x) \cdot (\text{g.f.\ of steep partitions})\) \\
    \(3412\) & Open & \((\text{g.f.\ of concave partitions})^2\) \\
    \(3421\) & Open & \((\text{g.f.\ of partitions with distinct parts, except the smallest})^2\) \\ \bottomrule
  \end{booktabs}
  \label{tab:lim seqs}
\end{table}

\begin{theorem} \label{thm:1324 pair limits}
  The limit generating functions of the pairs in Table \ref{tab:lim seqs} are as follows.
  \begin{enumerate}[(a)]
    \item \label{lim pair:1243}
    For \(p = 1243\),
    \begin{equation*}
      C_{1324, p}(x) = P(x).
    \end{equation*}
    %where \(P(x)\) is the generating function for the partition numbers.

    \item \label{lim pair:2143}
    For \(p = 2143\),
    \begin{equation*}
      C_{1324, p}(x) = 2 P(x) - 1.
    \end{equation*}

    \item \label{lim pair:1342}
    For \(p = 1342\),
    \begin{equation*}
      C_{1324, p}(x) = \prod_{k \geq 1} \frac{1 + x^k}{1 - x^k}.
    \end{equation*}

    \item \label{lim pair:1432}
    For \(p = 1432\),
    \begin{equation*}
      C_{1324, p}(x) = \sum_{k \geq 0} (k+1) x^k \cdot \prod_{i=1}^k \frac{1}{1 - x^i}.
    \end{equation*}

    \item \label{lim pair:4231}
    For \(p = 4231\),
    \begin{equation*}
      C_{1324, p}(x) = \left(\prod_{i \geq 1} (1+x^i) + \sum_{a, b \geq 0} x^{(a+2)(b+1)} \cdot \prod_{i=1}^a (1+x^i) \cdot \prod_{i=1}^b (1+x^i)\right)^2.
    \end{equation*}

    \item \label{lim pair:4321}
    For \(p = 4321\),
    \begin{equation*}
      C_{1324, p}(x) = \left(1 + \sum_{k \geq 1} \frac{x^k}{1 - x^k} + \sum_{k \geq 1} \sum_{i \geq k+1} \frac{x^{k+i}}{(1 - x^k)(1 - x^i)}\right)^2.
    \end{equation*}

    \item \label{lim pair:2341}
    For \(p = 2341\),
    \begin{equation*}
      C_{1324, p}(x) = \left(1 + \sum_{k \geq 1} x^{\frac{k(k+1)}{2}} \cdot \prod_{i=1}^k \left(x + \frac{1}{1 - x^i}\right)\right)^2.
    \end{equation*}

    \item \label{lim pair:2413}
    For \(p = 2413\),
    \begin{equation*}
      C_{1324, p}(x) = P(x)^2.
    \end{equation*}

    \item \label{lim pair:2431}
    For \(p = 2431\),
    \begin{equation*}
      C_{1324, p}(x) = P(x) \cdot \sum_{k \geq 0} \operatorname{steep}(k) x^k,
    \end{equation*}
    where \(\operatorname{steep}(k)\) is the number of steep partitions of \(k\).

    \item \label{lim pair:3412}
    For \(p = 3412\),
    \begin{equation*}
      C_{1324, p}(x) = \left(\,\prod_{k \geq 1}(1 + x^k) \cdot \sum_{c \geq 0} \frac{x^{\frac{c(c+3)}{2}}}{\prod_{i=1}^c(1 - x^{2i})}\right)^2.
    \end{equation*}

    \item \label{lim pair:3421}
    For \(p = 3421\),
    \begin{equation*}
      C_{1324, p}(x) = \left(1 + \sum_{k \geq 1} \frac{x^k}{1 - x^k} \cdot \prod_{i \geq k+1} \big(1+x^i\big)\right)^2.
    \end{equation*}
  \end{enumerate}
\end{theorem}

\begin{proof}[Proof of Theorem \ref{thm:1324 pair limits}]
  We begin with the four decomposable patterns
  \begin{equation*}
    1243, 2143, 1342, 1432,
  \end{equation*}
  to which Corollary \ref{cor:limit factorization} does not apply.

  If \(n \geq k + 3\), then every permutation in \(S_n\) with \(k\) inversions has at least three components by \eqref{eq:comp inv length}. A permutation with at least three components avoids \(\{1324, 1243\}\) if and only if it avoids \(132\). Consequently,
  \begin{equation*}
    \av_n^k(1324, 1243) = \av_n^k(132) = p(k),
  \end{equation*}
  which proves part \ref{lim pair:1243}.

  Next, let \(k \geq 1\) and \(n \geq k + 2\). Write a permutation \(\pi \in \Av_n^k(1324)\) as
  \begin{equation*}
    \pi = \sigma \oplus {\id_m} \oplus \tau,
  \end{equation*}
  as in Lemma \ref{lem:1324 decomp form}. Since \(2143 = 21 \oplus 21\), the permutation \(\pi\) avoids \(2143\) if and only if \(\sigma = 1\) or \(\tau = 1\). Thus, all \(k\) inversions occur either in \(\sigma\) or \(\tau\). Each possibility is counted by \(p(k)\), and therefore
  \begin{equation*}
    \av_n^k(1324, 2143) = 2 p(k).
  \end{equation*}
  For \(k = 0\) there is only the identity permutation, giving \(C_{1324, 2143}(x) = 2 P(x) - 1\) and proving part \ref{lim pair:2143}.

  The pattern \(1342\) decomposes as \(1 \oplus 231\). Hence, a decomposable permutation \(\pi = \sigma \oplus {\id_m} \oplus \tau\) avoids \(\{1324, 1342\}\) if and only if
  \begin{equation*}
    \sigma \in \Av(132) \quad \text{and} \quad \tau \in \Av(213, 231).
  \end{equation*}
  Indecomposable \(\{213, 231\}\)-avoiders with \(\ell\) inversions are in bijection with partitions of \(\ell\) into distinct parts \cite{franklin_pattern_2025}. It follows that \(\Av_n^k(1324, 1342)\), for \(n \geq k + 2\), is in bijection with pairs \((\lambda, \mu)\) of partitions such that \(|\lambda| + |\mu| = k\) and all parts of \(\mu\) are distinct. Overlining the parts of \(\mu\) and merging the two partitions gives an \emph{overpartition} of \(k\) (see sequence \href{https://oeis.org/A015128}{A015128} in the OEIS), proving part \ref{lim pair:1342}.

  Similarly, \(1432 = 1 \oplus 321\), and therefore \(\pi\) avoids \(\{1324, 1432\}\) if and only if
  \begin{equation*}
    \sigma \in \Av(132) \quad \text{and} \quad \tau \in \Av(213, 321).
  \end{equation*}
  Indecomposable \(\{213, 321\}\)-avoiders with \(\ell\) inversions correspond to partitions of \(\ell\) into equal parts \cite{franklin_pattern_2025}. Thus, we obtain pairs \((\lambda, \mu)\) of partitions in which all parts of \(\mu\) are equal. If \(\mu\) is empty, color all parts of \(\lambda\) blue. Otherwise, let \(i\) be the common size of the parts of \(\mu\). Color the parts of \(\lambda\) greater than or equal to \(i\) blue, and color the remaining parts of \(\lambda\), together with the parts of \(\mu\), red. Merging the two partitions gives a partition in which every blue part is greater than or equal to every red part. This is reversible and gives the generating function in part \ref{lim pair:1432}; see sequence \href{https://oeis.org/A093694}{A093694} in the OEIS.

  It remains to consider
  \begin{equation} \label{eq:indecomposable p}
    2341, 2413, 2431, 3412, 3421, 4231, 4321.
  \end{equation}
  These patterns are indecomposable, so Corollary \ref{cor:limit factorization} applies. The patterns
  \begin{equation*}
    2341, 3412, 3421, 4231, 4321
  \end{equation*}
  are fixed by \(p \mapsto (p^{-1})^{\rc}\), and hence
  \begin{equation*}
    C_{1324, p}(x) = C_{132, p}(x)^2
  \end{equation*}
  in each of these five cases. The required generating functions are given by Propositions \ref{prop:132 2341 spm}, \ref{prop:132 3412 penny}, \ref{prop:132 3421}, \ref{prop:132 4231} and \ref{prop:132 decreasing}.

  The pattern \(2413\) contains both \(132\) and \(213\). Thus, neither factor acquires an additional restriction, and
  \begin{equation*}
    C_{1324, 2413}(x) = C_{132}(x) \cdot C_{213}(x) = P(x)^2.
  \end{equation*}
  Finally, the mapping \(\pi \mapsto (\pi^{-1})^{\rc}\) sends \(\{213, 2431\}\) to \(\{132, 3241\}\). Proposition \ref{prop:132 3241 steep} therefore gives
  \begin{equation*}
    C_{1324, 2431}(x) = P(x) \cdot C_{132, 3241}(x) = P(x) \cdot \sum_{k \geq 0} \operatorname{steep}(k) x^k.
  \end{equation*}

  This concludes the proof. Further details can be found in Verkama's licentiate thesis \cite[Paper~B]{verkama_inversions_2026}, which includes the first version of this article.
\end{proof}

\paragraph{Higher limit sequences}

Linusson and Verkama \cite{linusson_enumerating_2025} showed that \(1324\) exhibits a higher-order stabilization phenomenon, and the numerical data indicate that the same is true for many of our pairs. Let \(a_p(n,k) = \av_n^k(1324, p)\). For
\begin{equation*}
  p \in \{1243, 1342, 1432, 2341, 3421, 4321\},
\end{equation*}
the values
\begin{equation*}
  a_p(n+1+i,k+i) - a_p(n+i,k+i)
\end{equation*}
appear to be eventually constant in \(i\). In other words, the first differences stabilize along every diagonal. Indexing the resulting diagonal limits by \(j = k-n+1\), we obtain what we call the \emph{secondary limit sequence}. We prove this for \(p = 1342\) in Theorem \ref{thm:1342 enum}; the generating function for its secondary limit sequence is
\begin{equation*}
  (2 + 2x) C_{1324, 1342}(x).
\end{equation*}

For the remaining representatives
\begin{equation*}
  p \in \{2143, 2413, 2431, 3412, 4231\},
\end{equation*}
the data instead suggest stabilization one level later. More precisely, if
\begin{equation*}
  b_p(n,k) = \big(a_p(n+2,k+1) - a_p(n+1,k+1)\big) - \big(a_p(n+1,k) - a_p(n,k)\big),
\end{equation*}
then \(b_p(n+i,k+i)\) appears to be eventually constant in \(i\). We refer to the corresponding sequence of constants as the \emph{tertiary limit sequence}. The data supporting both observations are given in Tables~S3--S24 of the supplementary material.

\section{Partitions from indecomposable 132-avoiders} \label{sec:132 indecomposable}

In this section, we complete the limit sequence analysis of Section \ref{sec:1324 pairs} by enumerating indecomposable permutations avoiding \(132\) and certain patterns of length four by the number of inversions. A similar analysis has been carried out before for collections of length-three patterns \cite{franklin_restricted_2024, franklin_pattern_2025}. For a collection \(B\) of patterns, let
\begin{equation*}
  I_k(B) = \left\{\pi \in \Av(B) : \comp(\pi) = 1 \text{ and } \inv(\pi) = k \right\} \quad \text{and} \quad
  i_k(B) = |I_k(B)|.
\end{equation*}
Since a \(132\)-avoiding permutation is decomposable if and only if it ends with its largest entry, we have
\begin{equation*}
  C_{132, B}(x) = \sum_{k \geq 0} i_k(\{132\} \cup B) x^k
\end{equation*}
whenever every pattern in \(B\) is indecomposable; this is the connection to limit sequences.

Recall that taking the inversion table defines a bijection \(\Lambda\) from \(I_k(132)\) to the partitions of \(k\). When restricted to \(I_k(\{132\} \cup B)\), the image is some restricted class of partitions. For \(I_k(132, p)\) with \(p \in S_4\), these partitions often have a nice description. In a few more interesting cases, an additional step is required to pass to some other previously studied partition-like object.

We will prove the characterization for \(I_k(132, 4231)\) in detail to illustrate the method. The remaining results follow from similar arguments; complete proofs are given in \cite[Paper~B, Section~7]{verkama_inversions_2026}. Throughout, if \(\lambda = (\lambda_1, \ldots, \lambda_\ell)\) is a partition, we set \(\lambda_i = 0\) for all \(i > \ell\).

\subsection{4231: convex partitions} \label{subsec:132 4231}

We call a partition \(\lambda = (\lambda_1, \ldots, \lambda_{\ell})\) \emph{convex} if it satisfies the following condition: if \(\lambda_i \geq \lambda_{i + 1} + 2\), then \(\lambda_{i+1} > \lambda_{i+2} > \ldots > \lambda_{\ell}\). We remark that the set of convex partitions is closed under conjugation, even if this is not evident from the definition.

\begin{proposition} \label{prop:132 4231}
  An indecomposable \(132\)-avoider \(\pi\) avoids \(4231\) if and only if \(\Lambda(\pi)\) is convex. In particular,
  \begin{equation*}
    C_{132, 4231}(x) = \prod_{i\geq 1} (1+x^i)\ + \sum_{a, b \geq 0} x^{(a+2)(b + 1)} \cdot \prod_{i=1}^a (1+x^i) \cdot \prod_{i=1}^b (1+x^i).
  \end{equation*}
\end{proposition}

\begin{proof}
  Suppose first that \(\lambda \coloneqq \Lambda(\pi)\) is not convex, i.e.\ that there are some \(i_1 < i_2\) with \(\lambda_{i_1} \geq \lambda_{i_1+1} + 2\) and \(\lambda_{i_2} = \lambda_{i_2+1}\). Suppose further that no \(i\) with \(i_1 < i < i_2\) has either of the two properties. Since \(\lambda_{i_1} - \lambda_{i_1+1} \geq 2\), there must be some \(i_3 > i_2\) such that \(\pi_{i_1+1} < \pi_{i_3} < \pi_{i_1}\). If \(i_2+1\) does not satisfy this property, i.e.\ \(\pi_{i_2+1} > \pi_{i_1}\), then \(\pi_{i_1} \pi_{i_2+1} \pi_{i_3}\) forms a \(132\) pattern. Therefore \(\pi_{i_2+1} < \pi_{i_1}\), and clearly \(\pi_{i_2} < \pi_{i_2+1}\). The entry \(1\) must come after \(\pi_{i_1}\) in \(\pi\), so \(\pi_{i_1} \pi_{i_2} \pi_{i_2+1} 1\) forms a \(4231\) pattern in \(\pi\).

  Conversely, suppose \(\pi\) contains \(4231\). We can assume that the entry \(1\) of \(\pi\) serves as the \(1\) in the pattern. Otherwise the entry \(1\) of \(\pi\) occurs before the \(31\) of the pattern, and the three entries together form \(132\). Furthermore, we can choose the \(2\) and \(3\) of the pattern to be adjacent. Here is why: say \(\pi_{i_1} \pi_{i_2} \pi_{i_3} 1\) is the occurrence of \(4231\) in \(\pi\). None of the entries \(\pi_i\) with \(i_2 < i < i_3\) can satisfy \(\pi_i > \pi_{i_3}\), or \(\pi_{i_2} \pi_i \pi_{i_3}\) is a \(132\) pattern. So, we have an occurrence \(\pi_{i_1} \pi_{i_2} \pi_{i_2+1} 1\) of \(4231\) in \(\pi\). Maximize \(i_1\), so that \(\pi_i < \pi_{i_2+1}\) for all \(i_1 < i < i_2\). In particular, \(\pi_{i_1+1} < \pi_{i_2+1} < \pi_{i_1}\), implying that \(\lambda_{i_1} \geq \lambda_{i_1+1} + 2\) and \(\lambda_{i_2} = \lambda_{i_2+1}\). Thus, \(\lambda\) is not convex.

  Figure \ref{fig:convex concave partition} (left) shows how to enumerate the convex partitions. If \(\lambda\) has no gap of size greater than one, then its conjugate has distinct parts, giving the first term of the generating function. Otherwise, let \(i\) be the first index such that \(\lambda_i - \lambda_{i+1} \geq 2\), and put \(a = \lambda_i - 2\) and \(b = i - 1\). Removing the rectangle of size \((a+2)(b+1)\) leaves two partitions \(\mu\) and \(\nu\), as in the figure. The partition \(\mu\) has no gap of size greater than one and at most \(b\) parts, whereas \(\nu\) has distinct parts and largest part at most \(a\). Conjugating \(\mu\) gives a partition with distinct parts and largest part at most \(b\). The two products in the summand follow, and the factor \(x^{(a+2)(b+1)}\) accounts for the rectangle.
\end{proof}

The enumeration sequence of the convex partitions does not appear in the OEIS.

\begin{figure}[t]
  \begin{tikzpicture}[scale=0.6]
    \draw[fill=black!20] 
      (0,7) -- (0,10) -- (1,10) -- (1,9) -- (3,9) -- (3,8) -- (4,8) -- (4,7) -- cycle;
    \draw[fill=black!20] 
      (5,5) -- (6,5) -- (6,4) -- (7,4) -- (7,2) -- (8,2) -- (8,1) -- (9,1) -- (9,0) -- (5,0) -- cycle;
    \draw 
      (-0.2,7) -- (5,7) -- (5,-0.2)
      (0,-0.2) -- (0,7)
      (-0.2,0) -- (5,0)
      (-0.2,5) -- (0,5)
      (4,-0.2) -- (4,0);
    \draw[dashed]
      (4,5) -- (0,5)
      (4,5) -- (5,5)
      (4,5) -- (4,0)
      (4,5) -- (4,7);
    \node at (-0.5,2.5) {\(a\)};
    \node at (2,-0.5) {\(b\)};
    \node at (-0.5,6) {\(2\)};
    \node at (4.5,-0.5) {\(1\)};
    \node at (1.6,8) {\(\mu\)};
    \node at (6,1.8) {\(\nu\)};

    \begin{scope}[shift={(11,0)}]
      \draw[fill=black!20]
        (0,4) -- (0,9) -- (1,9) -- (1,8) -- (2,8) -- (2,6) -- (3,6) -- (3,5) -- (4,5) -- (4,4) -- cycle;
      \draw[fill=black!20]
        (6,3) -- (7,3) -- (7,2) -- (9,2) -- (9,1) -- (10,1) -- (10,0) -- (6,0) -- cycle;
      \draw
        (-0.2,4) -- (6,4) -- (6,-0.2)
        (-0.2,0) -- (6,0)
        (0,0) -- (0,4)
        (4,-0.2) -- (4,0);
      \draw[dashed]
        (4,4) -- (4,0);
      \node at (-0.5,2) {\(c\)};
      \node at (5,-0.5) {\(\geq 2\)};
      \node at (7.6,1) {\(\mu\)};
      \node at (1,5.8) {\(\nu\)};
    \end{scope}
    
  \end{tikzpicture}
  \caption{Decomposing a convex partition (left) and a concave partition (right). The partition \(\mu\) has no gaps greater than one, whereas all parts of \(\nu\) are distinct.}
  \label{fig:convex concave partition}
\end{figure}

\subsection{3412: concave partitions, convex penny arrangements} \label{subsec:132 3412}

Avoiding \(3412\) leads to partitions that are, in a precise sense, dual to the convex partitions discussed above. We call a partition \(\lambda = (\lambda_1, \ldots, \lambda_{\ell})\) \emph{concave} if there are no two indices \(i < j \leq \ell\) such that \(\lambda_i = \lambda_{i+1}\) and \(\lambda_j - \lambda_{j+1} \geq 2\). (Note that if \(j = \ell\), then \(\lambda_{j+1} = 0\).) The set of concave partitions is also closed under conjugation.

\begin{remark}
  Let \(\lambda = (\lambda_1, \ldots, \lambda_{\ell})\) be a partition, and write \(\mu\) for the complement of \(\lambda\) inside the rectangle with \(\ell + 1\) rows and \(\lambda_1 + 1\) columns:
  \begin{equation*}
    \mu = (\lambda_1 + 1, \lambda_1 + 1 - \lambda_{\ell}, \ldots, \lambda_1 + 1 - \lambda_1).
  \end{equation*}
  Then \(\lambda\) is concave if and only if \(\mu\) is convex.
\end{remark}

These partitions appear in the OEIS, but under a different guise. A \emph{penny arrangement} is a configuration of pennies in rows in the plane such that the bottom row is contiguous, and each penny in a higher row touches two pennies in the row below it. A penny arrangement is called \emph{convex} if every row is contiguous. If \(\alpha\) has \(r\) pennies in the bottom row, let \(\alpha^{\compl}\) denote its complement with respect to the full triangular arrangement of \(r(r+1)/2\) pennies. Figure \ref{fig:penny complements} shows three examples with the complement shaded.

\begin{figure}[t]
  \centering
  \begin{tikzpicture}[scale=0.38]
    \foreach \x in {0,2,4,6,8}
      \draw (\x,0) circle (1);
    \foreach \x in {1,3,5}
      \draw (\x,1.732) circle (1);
    \draw (4,3.464) circle (1);
    \draw[fill=black!20] (7,1.732) circle (1);
    \foreach \x in {2,6}
      \draw[fill=black!20] (\x,3.464) circle (1);
    \foreach \x in {3,5}
      \draw[fill=black!20] (\x,5.196) circle (1);
    \draw[fill=black!20] (4,6.928) circle (1);

    \begin{scope}[shift={(12,0)}]
      \foreach \x in {0,2,4,6,8}
        \draw (\x,0) circle (1);
      \foreach \x in {3,5,7}
        \draw (\x,1.732) circle (1);
      \draw (6,3.464) circle (1);
      \draw [fill=black!20] (1,1.732) circle (1);
      \foreach \x in {2,4}
        \draw[fill=black!20] (\x,3.464) circle (1);
      \foreach \x in {3,5}
        \draw[fill=black!20] (\x,5.196) circle (1);
      \draw[fill=black!20] (4,6.928) circle (1);
    \end{scope}

    \begin{scope}[shift={(24,0)}]
      \foreach \x in {0,2,4,6,8}
        \draw (\x,0) circle (1);
      \foreach \x in {1,3,7}
        \draw (\x,1.732) circle (1);
      \draw (2,3.464) circle (1);
      \draw[fill=black!20] (5,1.732) circle (1);
      \foreach \x in {4,6}
        \draw[fill=black!20] (\x,3.464) circle (1);
      \foreach \x in {3,5}
        \draw[fill=black!20] (\x,5.196) circle (1);
      \draw[fill=black!20] (4,6.928) circle (1);
    \end{scope}
  \end{tikzpicture}
  \caption{Two convex penny arrangements and a non-convex one, with their complements shaded. From left to right, the corresponding partitions are \((4,1,1)\), \((2,2,1,1)\) and \((3,3)\).}
  \label{fig:penny complements}
\end{figure}

Let \(\Pi(\alpha)\) be the partition formed by the lengths of the top-to-bottom-right diagonals of \(\alpha^{\compl}\), read from right to left. The complement occupies the topmost positions of each diagonal, so the inverse mapping removes \(\lambda_i\) pennies from the top of the \(i\)-th diagonal. Since \(\lambda\) is weakly decreasing, every remaining penny above the bottom row is supported by two pennies in the row below it. For fixed \(k\) and \(r \geq k+1\), every partition of \(k\) fits above the bottom row, and this gives a bijection between penny arrangements with \(r\) pennies in the bottom row and \(k\) pennies in the complement, and the partitions of \(k\). Restricting to convex penny arrangements gives the following result.

\begin{proposition} \label{prop:132 3412 penny}
  Let \(\pi \in I_k(132)\), \(\lambda = \Lambda(\pi)\), and take \(r \geq k+1\) in the definition of \(\Pi\). The following are equivalent:
  \begin{enumerate}[(a)]
    \item\label{prop:penny avoid} \(\pi\) avoids \(3412\).
    \item\label{prop:penny partition} \(\lambda\) is a concave partition.
    \item\label{prop:penny convex} \(\Pi^{-1}(\lambda)\) is a convex penny arrangement.
  \end{enumerate}
  In particular,
  \begin{equation*}
    C_{132,3412}(x) = \prod_{k \geq 1}(1 + x^k) \cdot \sum_{c \geq 0} \frac{x^{\frac{c(c+3)}{2}}}{\prod_{i=1}^c(1 - x^{2i})}.
  \end{equation*}
\end{proposition}

The equivalence of \ref{prop:penny convex} and \ref{prop:penny partition} is easy to see from the picture. It follows that \(i_k(132, 3412)\) counts concave partitions of \(k\), as well as convex penny arrangements with \(k\) pennies in the complement and any fixed bottom row of size \(r \geq k+1\). These numbers form sequence \href{https://oeis.org/A005576}{A005576} in the OEIS, and the generating function is easy to derive in the same way as for the convex partitions; use the decomposition in Figure \ref{fig:convex concave partition} (right).

\subsection{The decreasing pattern}

\begin{proposition} \label{prop:132 decreasing}
  For every \(m \geq 2\), an indecomposable \(132\)-avoider \(\pi\) avoids \(\id_m^{\rev}\) if and only if \(\Lambda(\pi)\) has at most \(m-2\) distinct parts. In particular,
  \begin{equation*}
    C_{132, 4321}(x) = 1 + \sum_{k \geq 1} \frac{x^k}{1 - x^k} + \sum_{k \geq 1} \sum_{i \geq k+1} \frac{x^{k + i}}{(1 - x^k)(1 - x^i)}.
  \end{equation*}
\end{proposition}

\begin{proof}
  The entries between consecutive left-to-right minima of a \(132\)-avoider are increasing, so the length of a longest decreasing subsequence equals the number of left-to-right minima. By Lemma \ref{lem:132 partition properties} \ref{lem:132 partition property drop}, this is one more than the number of distinct parts of \(\Lambda(\pi)\). The generating function for \(m = 4\) follows by choosing zero, one or two distinct part sizes.
\end{proof}

For \(m = 4\), the enumeration sequence is \href{https://oeis.org/A265250}{A265250} in the OEIS.

\subsection{2341: the sand pile model} \label{subsec:132 2341 spm}

The \emph{sand pile model} is a discrete dynamical system originating in physics \cite{bak_criticality_1988}. In combinatorial terms, it consists of a set \(\SPM(k)\) of partitions of \(k\) defined recursively as follows: \((k) \in \SPM(k)\), and if \(\lambda \in \SPM(k)\) then every partition that can be obtained from \(\lambda\) by subtracting one from a part and adding one to the next part is also in \(\SPM(k)\). These partitions admit the following characterization.

\begin{proposition}[\cite{phan_structures_1999}] \label{prop:spm}
  A partition \(\lambda\) of \(k\) is contained in \(\SPM(k)\) if and only if
  \begin{enumerate}[(a)]
    \item\label{prop:spm three equal} \(\lambda\) has no three equal parts, and
    \item\label{prop:spm two two} if \(\lambda_{i_1} = \lambda_{i_1+1}\) and \(\lambda_{i_3} = \lambda_{i_3+1} > 0\) for some \(i_1 < i_3\), then there exists some \(i_2\) such that \(i_1 < i_2 < i_3\) and \(\lambda_{i_2} - \lambda_{i_2+1} \geq 2\).
  \end{enumerate}
\end{proposition}

\(2341\)-avoidance imposes precisely the same conditions on \(\Lambda(\pi)\). 

\begin{proposition} \label{prop:132 2341 spm}
  For every \(k\), we have \(\Lambda(I_k(132, 2341)) = \SPM(k)\). In particular,
  \begin{equation*}
    C_{132, 2341}(x) = 1 + \sum_{k \geq 1} x^{\frac{k(k+1)}{2}} \cdot \prod_{i=1}^k \left(x + \frac{1}{1 - x^i}\right).
  \end{equation*}
\end{proposition}

The generating function is due to Corteel and Gouyou-Beauchamps \cite{corteel_enumeration_2002}; see also sequence \href{https://oeis.org/A056219}{A056219} in the OEIS.

\subsection{3241: steep partitions}

We call a partition \(\lambda\) \emph{steep} if the difference between any two consecutive distinct parts is greater than or equal to the multiplicity of the smaller part. For example, \((5,5,3,3,2,1)\) is steep, but \((3,2,2)\) is not.

\begin{proposition} \label{prop:132 3241 steep}
  An indecomposable \(132\)-avoider \(\pi\) avoids \(3241\) if and only if \(\Lambda(\pi)\) is steep.
\end{proposition}

We were not able to find a nice expression for the generating function of the steep partitions.

\subsection{3421: distinct parts, except the smallest}

\begin{proposition} \label{prop:132 3421}
  An indecomposable \(132\)-avoider \(\pi\) avoids \(3421\) if and only if all parts of \(\Lambda(\pi)\), except possibly the smallest, have multiplicity one. In particular,
  \begin{equation*}
    C_{132, 3421}(x) = 1 + \sum_{k \geq 1} \frac{x^k}{1 - x^k} \cdot \prod_{i \geq k+1} \big(1+ x^i\big).
  \end{equation*}
\end{proposition}

The generating function follows by choosing the smallest part \(k\), its multiplicity, and a set of larger parts. This is sequence \href{https://oeis.org/A115029}{A115029} in the OEIS.

\section{Almost decomposability} \label{sec:almost decomp}

In this section, we show that almost every pattern \(p \in \Av(1324)\) satisfies
\begin{equation*}
  \av_n^k(1324, p) \leq \av_{n+1}^k(1324, p) \quad \text{for all } k \text{ and } n \geq \tfrac{k+7}{2}.
\end{equation*}
In particular, this holds for \(p \in \{1342, 1432, 4231, 4321\}\), which are precisely the half-monotone pairs in Table \ref{tab:lim seqs}. In the case of \(p = 1342\), we also obtain an enumeration of \(\av_n^k(1324, p)\) for \(n \geq \frac{k+7}{2}\). Our methods are based on the notion of \emph{almost decomposability} introduced in \cite{linusson_enumerating_2025}, as well as the accompanying injective mapping \(f\) defined on almost decomposable \(1324\)-avoiders, with the following properties: \(f(\pi)\) avoids \(1324\), \(\inv(f(\pi)) = \inv(\pi)\) and \(|f(\pi)| = |\pi| + 1\).

Let us recall the preliminaries. We say that \(\pi \in S_n\) is \emph{almost decomposable} if it is indecomposable, but there exists an entry \(e \in \{1, \pi_1, n, \pi_n\}\) such that \(\pi \delete e\) is decomposable.

\begin{theorem}[Theorem 8 in \cite{linusson_enumerating_2025}] \label{thm:almost decomp 1324}
  If \(\pi\) avoids \(1324\) and \(\inv(\pi) \leq 2|\pi| - 7\), then \(\pi\) is decomposable or almost decomposable.
\end{theorem}

By Lemma \ref{lem:1324 decomp form}, a decomposable \(1324\)-avoider \(\pi\) is of the form \(\pi = \sigma \oplus {\id_m} \oplus \tau\), where \(\sigma\) is an indecomposable \(132\)-avoider and \(\tau\) is an indecomposable \(213\)-avoider. This allows us to define
\begin{equation*}
  \widetilde f(\pi) = \sigma \oplus {\id_{m+1}} \oplus \tau,
\end{equation*}
which clearly preserves \(1324\)-avoidance and the number of inversions. If \(\pi \in \Av_n(1324)\) is almost decomposable, we define \(f(\pi)\) as follows.

\begin{enumerate}[F1]
  \item \label{inj:p1} If \(\pi \delete \pi_1\) is decomposable, let \(f(\pi)_1 = \pi_1\) and \(f(\pi) \delete \pi_1 = \widetilde f(\pi \delete \pi_1)\).
  \item \label{inj:inv} If \(\pi \delete 1\) is decomposable, let \(f(\pi) = f(\pi^{-1})^{-1}\).
  \item \label{inj:rc} Otherwise, let \(f(\pi) = f(\pi^{\rc})^{\rc}\).
\end{enumerate}

In other words, we remove a point from the boundary of the plot of \(\pi\) to make it decomposable, use \(\widetilde f\) on the resulting permutation, and then insert the point back in its original position.

The mapping is presented as above in \cite{linusson_enumerating_2025}, but the second and third cases can be written out more explicitly.

\begin{enumerate}[F1\('\), start=2]
  \item If \(\pi \delete 1\) is decomposable, let \(f(\pi)^{-1}_1 = \pi^{-1}_1\) and \(f(\pi) \delete 1 = \widetilde f(\pi \delete 1)\).
  \item \label{inj:pn} If \(\pi \delete \pi_n\) is decomposable, let \(f(\pi)_{n+1} = \pi_n + 1\) and \(f(\pi) \delete f(\pi)_{n+1} = \widetilde f(\pi \delete \pi_n)\).
  \item \label{inj:n} If \(\pi \delete n\) is decomposable, let \(f(\pi)^{-1}_{n+1} = \pi^{-1}_n + 1\) and \(f(\pi) \delete \{n+1\} = \widetilde f(\pi \delete n)\).
\end{enumerate}
For notational convenience, we set \(f(\pi) \coloneqq \widetilde f(\pi
)\) when \(\pi\) is decomposable.

\begin{remark} \label{rmk:f priority}
  It is important to note that a certain priority order is chosen in the definition of \(f\). If \(\pi \delete \pi_1\) and, say, \(\pi \delete \pi_n\) are \emph{both} decomposable, we define \(f(\pi)\) according to case \ref{inj:p1}. We only use case \ref{inj:rc} if \(\pi \delete \pi_1\) and \(\pi \delete 1\) are both indecomposable. It is, however, impossible for both \(\pi \delete \pi_1\) and \(\pi \delete 1\) to be decomposable simultaneously, so there is no priority issue between cases \ref{inj:p1} and \ref{inj:inv}, or between \ref{inj:pn} and \ref{inj:n} \cite[Proposition~7]{linusson_enumerating_2025}. 

  The choice of priority is arbitrary, and the mapping could just as well have been defined by prioritizing \ref{inj:rc} over \ref{inj:p1} and \ref{inj:inv}. Some results regarding \(f\) in this section are, for this reason, not closed under taking reverse complements of the patterns appearing in them, but similar results valid for those patterns can be obtained by changing priority.
\end{remark}
 
\begin{theorem}[Theorem 19 in \cite{linusson_enumerating_2025}] \label{thm:inj 1324}
  The mapping \(f\) is injective, \(\inv(f(\pi)) = \inv(\pi)\), and \(f(\pi) \in \Av(1324)\) for any decomposable or almost decomposable \(1324\)-avoider \(\pi\). In particular, if \(n \geq \frac{k+7}{2}\), then
  \begin{equation*}
    \av_n^k(1324) \leq \av_{n+1}^k(1324).
  \end{equation*}
\end{theorem}

\subsection{Compatible patterns} \label{subsec:compatible patterns}

This subsection examines patterns \(p\) for which \(f\) always preserves \(p\)-avoidance.

\begin{definition}
  A collection \(B\) of patterns is called \emph{\(f\)-compatible} if \(f(\pi)\) avoids \(B\) for any decomposable or almost decomposable permutation \(\pi \in \Av(1324, B)\). If \(B = \{p\}\) is \(f\)-compatible, we say that \(p\) itself is \(f\)-compatible. If \(B\) is not \(f\)-compatible, we say that it is \emph{\(f\)-incompatible}.
\end{definition}

The following lemma collects some simple properties of \(f\)-compatible sets.

\begin{lemma} \label{lem:f compatible properties} 
  Let \(p \in \Av(1324)\) and \(B, B' \subseteq \Av(1324)\).
  \begin{enumerate}[(a)]
    \item \label{lem:f compatible properties:inv} If \(p\) is \(f\)-compatible, then \(p^{-1}\) is also \(f\)-compatible.
    \item \label{lem:f compatible properties:union} If \(B\) and \(B'\) are \(f\)-compatible, then \(B \cup B'\) is \(f\)-compatible. In particular, a collection of \(f\)-compatible patterns is \(f\)-compatible.
    \item \label{lem:f compatible properties:injection} If \(B\) or \(B^{\rc}\) is \(f\)-compatible, then 
    \begin{equation*}
      \av_n^k(1324, B) \leq \av_{n+1}^k(1324, B)
    \end{equation*}
    for all \(k\) and \(n \geq \frac{k+7}{2}\).
  \end{enumerate}
\end{lemma}

\begin{proof}
  Part \ref{lem:f compatible properties:inv} follows from the symmetry in the definition of \(f\). Parts \ref{lem:f compatible properties:union} and \ref{lem:f compatible properties:injection} are obvious from the definition of \(f\)-compatibility and Remark \ref{rmk:f priority}.
\end{proof}

Part \ref{lem:f compatible properties:injection} is key: if \(B\) is \(f\)-compatible, then \(\{1324\} \cup B\) is inversion monotone when \(n \geq \frac{k+7}{2}\). We may assume throughout that \(B\) consists of \(1324\)-avoiding patterns; if \(p\) contains \(1324\), then \(\Av(1324, p) = \Av(1324)\).

The main result of this section provides asymptotics for the number of \(f\)-compatible patterns. In particular, we prove that almost every \(1324\)-avoider is \(f\)-compatible.

\begin{theorem} \label{thm:incomp asymptotics}
  Let \(E_n\) be the set of \(f\)-incompatible patterns in \(\Av_n(1324)\). We have
  \begin{equation*}
    |E_n| = \Theta \left(\frac{4^n}{\sqrt n}\right),
  \end{equation*}
  and hence
  \begin{equation*}
    \lim_{n \to \infty} \left(\frac{|E_n|}{\av_n(1324)}\right)^{\frac 1n} = \frac{4}{\lambda} < 0.39,
  \end{equation*}
  where \(\lambda > 10.27\) is the Stanley--Wilf limit of \(1324\).
\end{theorem}

\begin{corollary} \label{cor:almost all sets comp}
  Let \(\Av_{\leq n}(1324) = \bigcup_{m = 1}^n \Av_m(1324)\). For every fixed \(N\), the proportion of \(f\)-compatible sets among the subsets of \(\Av_{\leq n}(1324)\) of cardinality at most \(N\) tends to \(1\) as \(n \to \infty\).
\end{corollary}

\begin{table}[t]
  \centering
  \caption{Lower and upper bounds for the number of \(f\)-compatible patterns provided by Theorems \ref{thm:incomp necessary} and \ref{thm:incomp sufficient}, compared with the computational upper bound. The last column is the computational bound divided by \(\av_n(1324)\).}
  \begin{booktabs}{
    column{2,3,5} = {leftsep=14pt},
    rowsep = 2pt
  }
    \toprule
    \SetCell[r=2]{c} \(n\) & Lower bound & \SetCell[c=2]{c} Upper bound & & \SetCell[r=2]{c} Proportion \\ \cmidrule[lr]{2-2} \cmidrule[lr]{3-4}
    & Theorem \ref{thm:incomp necessary} & Computational & Theorem \ref{thm:incomp sufficient} & \\
    \midrule
    4 & 3 & 3 & 5 & 0.130 \\
    5 & 12 & 12 & 16 & 0.117 \\
    6 & 62 & 66 & 88 & 0.129 \\
    7 & 640 & 675 & 789 & 0.244 \\
    8 & 6409 & 6599 & 7113 & 0.418 \\
    9 & 55017 & 55875 & 57995 & 0.590 \\
    10 & 427939 & 431491 & 439797 & 0.729 \\
    \bottomrule
  \end{booktabs}
  \label{tab:comp counts}
\end{table}

Theorem \ref{thm:incomp asymptotics} follows directly from two technical results (Theorems \ref{thm:incomp necessary} and \ref{thm:incomp sufficient} below) giving, respectively, sufficient and necessary conditions for a pattern to be \(f\)-compatible. The results do not, however, fully characterize the \(f\)-compatible patterns. In Table \ref{tab:comp counts} we compare, for \(4 \leq n \leq 10\), the lower bound on the number of \(f\)-compatible patterns from Theorem \ref{thm:incomp necessary} with the upper bound from Theorem \ref{thm:incomp sufficient}, as well as a computational upper bound. The latter was obtained as follows: for each \(p \in \Av_n(1324)\), we computed \(f(\pi)\) for all decomposable and almost decomposable permutations \(\pi \in \Av_m(1324, p)\) with \(n-1 \leq m \leq n+2\), and checked if \(f(\pi)\) avoids \(p\). If any \(f(\pi)\) contained \(p\), then \(p\) was declared \(f\)-incompatible.

We also include the proportion of all permutations in \(\Av_n(1324)\) that pass the computational \(f\)-compatibility check. Note that as \(n\) grows, this proportion will eventually exceed \(1 - 0.39^n\) by Theorem \ref{thm:incomp asymptotics}.

For \(n \leq 5\), our bounds are tight. The \(f\)-compatible patterns of length four are
\begin{equation*}
  1432,\ 4231,\ 4321,
\end{equation*}
and the \(f\)-compatible patterns of length five are
\begin{align*}
  & 14523,\ 14532,\ 15342,\ 15423,\ 15432,\ 34125, \\
  & 52341,\ 52431,\ 53241,\ 53421,\ 54231,\ 54321.
\end{align*}

\begin{theorem} \label{thm:incomp necessary}
  If \(p \in \Av_n(1324)\) is \(f\)-incompatible, then at least one of the patterns \(q \in \left\{p, p^{-1}, p^{\rc}, (p^{\rc})^{-1}\right\}\) satisfies one of the following conditions:
  \begin{enumerate}[(a)]
    \item \label{incomp cond:3 comp}\(\comp(q) \geq 3\).
    \item \label{incomp cond:almost decomp} \(\comp(q \delete q_1) > \comp(q)\).
    \item \(q_1 > 1\), \label{incomp cond:2 comp} \(\comp(q) = 2\), and the first component of \(q\) begins with its largest entry. Furthermore, if \(q \in \left\{p^{\rc}, (p^{\rc})^{-1} \right\}\), then \(q_n < n\).
    \item \label{incomp cond:av 213}\(1 < q_1 < n\) and \(q \delete q_1\) avoids \(213\).
  \end{enumerate}
\end{theorem}

\begin{proof}
  Since \(p\) is not \(f\)-compatible, there exists a decomposable or almost decomposable permutation \(\pi \in \Av(1324, p)\) such that \(f(\pi)\) contains \(p\). Recall first that if \(\pi\) is decomposable, then \(\pi = \sigma \oplus {\id_m} \oplus \tau\), where \(\sigma\) and \(\tau\) are indecomposable, and \(f(\pi) = \sigma \oplus {\id_{m+1}} \oplus \tau\). Since \(\pi\) avoids \(p\), it must be that the occurrence of \(p\) in \(f(\pi)\) uses every entry of \(\id_{m+1}\). This implies that \(\comp(p) \geq 3\), condition \ref{incomp cond:3 comp}.

  Assume instead that \(\pi\) is almost decomposable, and that \(\pi \delete \pi_1\) is decomposable (so \(f(\pi)\) is defined as in case \ref{inj:p1}). Write \(\pi \delete \pi_1 = \sigma \oplus {\id_m} \oplus \tau\), where \(\sigma\) and \(\tau\) are indecomposable. We analyze how \(p\) can occur in \(f(\pi)\).

  \begin{itemize}
    \item If \(p\) is contained in \(f(\pi \delete \pi_1)\), then \(\comp(p) \geq 3\) as above. So, assume instead that every occurrence of \(p\) in \(f(\pi)\) uses the entry \(f(\pi)_1 = \pi_1\).
    \item If an occurrence of \(p\) uses only \(\pi_1\) and points of \(\sigma\), then \(\pi\) also contains \(p\). So, every occurrence of \(p\) in \(f(\pi)\) must use at least one point of \(\id_{m+1}\) or \(\tau\).
    \item If an occurrence of \(p\) uses \(\pi_1\) as well as points from both \(\sigma\) and \(\id_{m+1}\) (and possibly \(\tau\)), then \(\comp(p \delete p_1) > \comp(p)\) since \(\pi_1 \geq |\sigma| + m+1\). This is condition \ref{incomp cond:almost decomp}.
    \item If an occurrence of \(p\) uses only \(\pi_1\) and points from \(\id_{m+1}\), then there is another occurrence of \(p\) that uses points from \(\sigma\): replace \(p_2\) with any point from \(\sigma\). This case was handled above.
    \item Suppose an occurrence of \(p\) uses \(\pi_1\) and points only from \(\tau\). It is clear that if \(p_1 = n\) then \(\pi\) also contains \(p\). Furthermore, if \(p_1 = 1\) then \(\pi\) contains \(p\) by taking \(p \delete p_1\) from \(\tau\) together with any point from \(\sigma\). Hence, \(1 < p_1 < n\) and condition \ref{incomp cond:av 213} holds.
    \item Suppose an occurrence of \(p\) uses \(\pi_1\) as well as points from both \(\id_{m+1}\) and \(\tau\), none from \(\sigma\). Again, \(p \delete p_1\) avoids \(213\). If \(p_1 = n\), then \(\pi\) also contains \(p\): take \(p_1\) as \(\pi_1\), \(p_2\) arbitrarily from \(\sigma\), and \(p \delete \{p_1, p_2\}\) from \({\id_m} \oplus \tau\). Since we clearly have \(p_1 > 1\), condition \ref{incomp cond:av 213} holds.
    \item Finally, suppose that an occurrence of \(p\) uses \(\pi_1\) as well as points from \(\sigma\) and \(\tau\), but not from \(\id_{m+1}\). We have \(\comp(p \delete p_1) \geq 2\), so if \(\comp(p) = 1\) we have condition \ref{incomp cond:almost decomp}, and if \(\comp(p) \geq 3\) we have condition \ref{incomp cond:3 comp}. Assume that \(\comp(p) = 2\). The first component of \(p\) is contained in \(\pi_1 \sigma\) and uses \(\pi_1\), which is larger than every entry of \(\sigma\). Hence, the first component of \(p\) begins with its largest entry, condition \ref{incomp cond:2 comp}. We know that \(p_1 > 1\), since \(p\) uses entries from \(\sigma\).
  \end{itemize}
  
  If instead \(\pi \delete 1\) is decomposable, the same arguments apply to \(\pi^{-1}\) and \(q = p^{-1}\). If both \(\pi \delete \pi_1\) and \(\pi \delete 1\) are indecomposable, then take \(\pi^{\rc}\) and \(q = p^{\rc}\) instead. We only need to show that the additional property holds in condition \ref{incomp cond:2 comp}. So, suppose \(q = p^{\rc}\) is contained in \(f(\pi)\), where \(\pi \delete \pi_1 = \sigma \oplus {\id_m} \oplus \tau\) is decomposable, and the occurrence of \(q\) in \(f(\pi)\) uses only \(f(\pi)_1 = \pi_1\) and points from both \(\sigma\) and \(\tau\). Again, we may assume that \(q\) has exactly two components. If \(q_n = n\), then the second component is simply \(1\). If our occurrence of \(q\) contains any other entries of \(\tau\), then \(q \delete q_1\) has at least three components, which gives us condition \ref{incomp cond:almost decomp}. 
  
  Hence, we assume that the occurrence of \(q\) uses only \(\pi_1\) and points from \(\sigma\), as well as one point from \(\tau\). Now, observe that \(\pi_1 < n\): otherwise the last entry of \(\pi^{\rc}\) is \(1\), \(\pi^{\rc} \delete 1\) is decomposable, and \(f(\pi^{\rc})\) is defined according to case \ref{inj:inv} instead of \ref{inj:rc}. This means that there exists an entry of \(\tau\) that is larger than \(\pi_1\) in \(\pi\). Taking this entry together with our occurrence of \(q \delete q_n\) in \(f(\pi)\) creates an occurrence of \(q\) in \(\pi\). This is a contradiction, so we must have \(q_n \neq n\), concluding the proof.
\end{proof}

\begin{theorem} \label{thm:incomp sufficient}
  Let \(p \in \Av_n(1324)\) be a pattern. If at least one of the patterns \(q \in \left\{p, p^{-1}, p^{\rc}, (p^{\rc})^{-1}\right\}\) satisfies one of the following conditions, then \(p\) is \(f\)-incompatible:
  \begin{enumerate}[(a)]
    \item \label{incomp suff cond:3 comp} \(\comp(q) \geq 3\).
    \item \label{incomp suff cond:almost decomp}\(q_1 < n\) and \(\comp(q \delete q_1) > \comp(q)\).
    \item \label{incomp suff cond:2 comp} \(q_1 > 1\), \(\comp(q) = 2\), and the first component of \(q\) begins with its largest entry.
    \item \label{incomp suff cond:av 213}\(1 < q_1 < n\) and \(q \delete q_1\) avoids \(213\).
  \end{enumerate}
  In cases \ref{incomp suff cond:almost decomp}, \ref{incomp suff cond:2 comp} and \ref{incomp suff cond:av 213}, if \(q \in \left\{p^{\rc}, (p^{\rc})^{-1}\right\}\), we require further that \(q_1 < n-1\).
\end{theorem}

\begin{proof}
  In each case, it suffices to find a permutation \(\pi \in \Av(1324, p)\) such that \(f(\pi)\) contains \(p\). We begin with the case  \(q = p\). The case \(q = p^{-1}\) is analogous -- take \(\pi^{-1}\) instead of \(\pi\). In \ref{incomp suff cond:3 comp} we have \(p = \sigma \oplus {\id_m} \oplus \tau\) with \(m \geq 1\), so we can take \(\pi = \sigma \oplus {\id_{m-1}} \oplus \tau\). Then \(\pi\) avoids \(p\), and \(f(\pi) = p\). Henceforth, assume that \(\comp(p) \leq 2\).

  Suppose that \(p\) satisfies \ref{incomp suff cond:almost decomp}. Assume first that \(\comp(p) = 2\), say \(p = \sigma \oplus \tau\). Here \(\sigma\) is an indecomposable \(132\)-avoider such that \(\comp(\sigma \delete \sigma_1) \geq 2\), and it is easy to see that such a permutation must begin with its largest entry -- condition \ref{incomp suff cond:2 comp} holds, and this case will be handled in the next paragraph. Assume instead that \(p\) is indecomposable. If \(\comp(p \delete p_1) \geq 3\), we can let \(\pi\) be the preimage \(f^{-1}(p)\). Otherwise, let \(\pi\) be defined by \(\pi \delete \pi_1 = p \delete p_1\) and \(\pi_1 = p_1 + 1\) (since \(p_1 < n\), this is possible). Clearly \(\pi\) avoids \(p\), and \(f(\pi)\) contains \(p\).

  Now, assume \(p\) satisfies \ref{incomp suff cond:2 comp}. As above, define \(\pi\) by \(\pi \delete \pi_1 = p \delete p_1\), \(\pi_1 = p_1 + 1\). Clearly \(\pi\) is indecomposable, \(\pi\) avoids \(p\), and \(f(\pi)\) contains \(p\).

  Finally, suppose that \(p\) satisfies \ref{incomp suff cond:av 213}. Define \(\pi\) by \(\pi \delete \pi_1 = 1 \oplus (p \delete p_1)\) and \(\pi_1 = p_1 + 2\), so that \(\pi \delete 1\) is obtained from \(p\) by shifting \(p_1\) up by one step. Then \(\pi\) avoids \(p\) and \(f(\pi)\) contains \(p\).

  If instead \(q \in \left\{p^{\rc}, (p^{\rc})^{-1}\right\}\) satisfies one of the conditions, we want to construct \(\pi\) in an analogous manner. However, care must be taken to ensure that case \ref{inj:rc} is used when \(f\) is applied. Concretely, we want to show that if \(q = p^{\rc}\) satisfies \ref{incomp suff cond:almost decomp}, \ref{incomp suff cond:2 comp} or \ref{incomp suff cond:av 213}, and \(q_1 < n-1\), then there exists a permutation \(\pi \in \Av(1324, q)\) such that \(f(\pi)\) contains \(q\), and both \(\pi \delete |\pi|\) and \(\pi \delete \pi_{|\pi|}\) are indecomposable. (When we return to \(p = q^{\rc}\) the relevant permutation is \(\pi^{\rc}\), and the extra indecomposability conditions imply that \(f(\pi^{\rc})\) is defined according to \ref{inj:rc}.)

  First, suppose \(q = p^{\rc}\) satisfies condition \ref{incomp suff cond:almost decomp} or \ref{incomp suff cond:2 comp}: either \(q\) is indecomposable and \(\comp(q \delete q_1) \geq 2\), or \(q\) has exactly two components, the first of which starts with its largest entry. In either case, we may write \(q \delete q_1 = \sigma \oplus \tau\), where \(\tau \in \Av_m(213)\) is indecomposable. We define
  \begin{equation} \label{eq:annoying tau}
    \tau' = \begin{cases}
      \tau & \text{if } \tau_1 = m, \tau_m = 1, \\
      1 \ominus \tau & \text{if } \tau_1 \neq m, \tau_m = 1, \\
      \tau \ominus 1 & \text{if } \tau_1 = m, \tau_m \neq 1, \\
      1 \ominus \tau \ominus 1 & \text{if }\tau_1 \neq m, \tau_m \neq 1,
    \end{cases}
  \end{equation}
  in each case ensuring that \(\tau'\) begins with its largest entry and ends with \(1\), so that in particular \(\tau'\) is not almost decomposable. Now, define \(\pi\) by \(\pi \delete \pi_1 = \sigma \oplus \tau'\) and
  \begin{equation*}
    \pi_1 = \begin{cases}
      q_1 + 1 & \text{if } \tau_m = 1, \\
      q_1 + 2 & \text{if } \tau_m \neq 1.
    \end{cases}
  \end{equation*}
  We can check that \(\pi\) avoids \(q\): an occurrence of \(q\) in \(\pi\) cannot use the new entries added to \(\tau\), so it must consist of \(\pi_1\) and \(\sigma \oplus \tau\) -- but \(\pi_1\) was shifted up one step to prevent this. Furthermore, \(f(\pi)\) contains \(q\) and both \(\pi \delete |\pi|\) and \(\pi \delete \pi_{|\pi|}\) are indecomposable, since \(\tau'\) is not almost decomposable and, critically, \(\pi_1\) is not the largest entry of \(\pi\). (This is where we use the assumption that \(q_1 < n-1\).)

  Lastly, suppose that \(q = p^{\rc}\) satisfies condition \ref{incomp suff cond:av 213}. Let \(q \delete q_1 = \tau \in \Av_m(213)\) and define \(\tau'\) as in \eqref{eq:annoying tau}. Define \(\pi\) by \(\pi \delete \pi_1\) = \(1 \oplus \tau'\) and
  \begin{equation*}
    \pi_1 = \begin{cases}
      q_1 + 2 & \text{if } \tau_m = 1, \\
      q_1 + 3 & \text{if } \tau_m \neq 1.
    \end{cases}
  \end{equation*}
  The argument is analogous to the previous case. Since the case \(q = (p^{\rc})^{-1}\) is symmetrical, this concludes the proof.
\end{proof}

The two preceding results provide a simple proof of Theorem \ref{thm:incomp asymptotics}, which implies Corollary \ref{cor:almost all sets comp}.

\begin{proof}[Proof of Theorem \ref{thm:incomp asymptotics}]
  Suppose that \(p \in \Av_n(1324)\) is \(f\)-incompatible. By Theorem \ref{thm:incomp necessary}, some \(q \in \left\{p, p^{-1}, p^{\rc}, (p^{\rc})^{-1}\right\}\) satisfies one of its conditions. This implies that at least one of the following holds: \(q\) is decomposable; \(q \delete q_1\) is decomposable; \(q \delete q_1 \in \Av(213)\). Hence, if \(E_n\) is the set of \(f\)-incompatible patterns in \(\Av_n(1324)\), then
  \begin{equation} \label{eq:f incomp bound}
    |E_n| \leq 4(D_n + nD_{n-1} + n C_{n-1}),
  \end{equation}
  where \(D_n\) is the number of decomposable permutations in \(\Av_n(1324)\) and 
  \begin{equation*}
    C_n = \frac{1}{n+1} \binom{2n}{n}
  \end{equation*}
  is the \(n\)-th Catalan number. By Lemma \ref{lem:1324 decomp form} we have
  \begin{equation*}
    D_n \leq \sum_{i = 1}^{n-1} C_i C_{n-i} = C_{n+1} - 2 C_n < 2 C_n,
  \end{equation*}
  so \eqref{eq:f incomp bound} gives
  \begin{equation*}
    |E_n| < (12 n + 32) C_n = \Theta (n C_{n-1}) = \Theta \left(\frac{4^n}{\sqrt n}\right).
  \end{equation*}

  For the lower bound, let \(p \in S_n\) be any permutation with \(1 < p_1 < n\) and \(p \delete p_1 \in \Av(213)\). Then \(p\) is \(f\)-incompatible by Theorem \ref{thm:incomp sufficient}, so 
  \begin{equation*}
    |E_n| \geq (n-2) C_{n-1} = \Theta (n C_{n-1}) = \Theta \left(\frac{4^n}{\sqrt n}\right).
  \end{equation*}
  This proves the claim.
\end{proof}   

\begin{proof}[Proof of Corollary \ref{cor:almost all sets comp}]
  Let \(B \subseteq \Av_{\leq n}(1324)\) with \(|B| \leq N\). By Lemma \ref{lem:f compatible properties} \ref{lem:f compatible properties:union}, \(B\) is \(f\)-compatible if all of its elements are \(f\)-compatible. Thus,
  \begin{equation*}
    \mathbb P(B \text{ is \(f\)-incompatible}) \leq N \frac{\left|\bigcup_{m \leq n} E_m\right|}{|{\Av_{\leq n}(1324)}|} = O\left(\frac{4^n}{\sqrt n \av_n(1324)}\right) \longrightarrow 0. \qedhere
  \end{equation*}
\end{proof}

\subsection{1342 is almost compatible} \label{subsec:1342 almost compatible}

Out of the pairs declared half-monotone in Section \ref{sec:1324 pairs}, only \(\{1324, 1342\}\) remains. This is because the pattern \(1342\) is not \(f\)-compatible: if \(\pi = 34152\) then \(f(\pi) = 241563\), whose subsequence \(2563\) forms a \(1342\)-pattern. However, observe that \(\inv(\pi) = 5 = 2 \cdot 5 - 5\) here. We can show that every such counterexample has at least this many inversions.

\begin{theorem} \label{thm:1342 almost compatible}
  If \(\pi \in \Av_n(1324, 1342)\) is decomposable or almost decomposable and \(\inv(\pi) \leq 2n-6\), then \(f(\pi)\) avoids \(1342\). In particular, for all \(k\) and \(n \geq \frac{k+7}{2}\), we have
  \begin{equation*}
    \av_n^k(1324, 1342) \leq \av_{n+1}^k(1324, 1342).
  \end{equation*}
\end{theorem}

\begin{remark} \label{rmk:1342 sharp}
  The bound in Theorem \ref{thm:1342 almost compatible} is sharp. For every \(n \geq 5\), the permutation
  \begin{equation*}
    \pi = 3\,4 \cdots (n-1)\,1\,n\,2
  \end{equation*}
  is an almost decomposable element of \(\Av_n(1324, 1342)\) with \(\inv(\pi) = 2(n-3) + 1 = 2n-5\), and \(f(\pi)\) contains \(1342\). For \(n = 5\) this is the permutation \(34152\) above.
\end{remark}

\begin{remark}
  The pattern \(1342\) is the only such case among the length-four patterns. All others are either \(f\)-compatible, or there are counterexamples \(\pi \in \Av_n(1324, p)\) with \(n\), \(n+1\) or \(n+2\) inversions, so that \(f(\pi)\) contains \(p\).
\end{remark}

\begin{proof}[Proof of Theorem \ref{thm:1342 almost compatible}]
  First of all, it is clear that if \(\pi\) is decomposable, then \(f(\pi) = \widetilde f(\pi)\) avoids \(1342\). So, let us assume that \(\pi\) is an almost decomposable \(\{1324,1342\}\)-avoiding permutation of length \(n\). We will examine how the \(1342\)-pattern can appear in \(f(\pi)\). The goal is to prove that \(f(\pi)\) must have been constructed according to case \ref{inj:pn} in the definition of \(f\).

  \begin{itemize}
    \item Suppose \(\pi \delete \pi_1\) is decomposable. Since \(\widetilde f(\pi \delete \pi_1)\) avoids \(1342\), we know that if \(f(\pi)\) contains \(1342\), then \(f(\pi)_1\) must be the \(1\) in the pattern. However, \(\pi_1 > \pi_2\) and hence \(f(\pi)_1 > f(\pi)_2\), so \(\widetilde f(\pi \delete \pi_1)\) must also contain \(1342\). This is a contradiction, so \(\pi \delete \pi_1\) must be indecomposable.
    \item By an identical argument, \(\pi \delete 1\) must be indecomposable.
    \item Suppose that \(\pi \delete \pi_1\) and \(\pi \delete 1\) are indecomposable, and \(\pi \delete n\) is decomposable. Again, \(\widetilde f(\pi \delete n)\) avoids \(1342\), so the entry \(n+1\) of \(f(\pi)\) must be the \(4\) in any occurrence of \(1342\). Since \(\pi^{-1}_n < \pi^{-1}_{n-1}\) we have that \(f(\pi)^{-1}_{n+1} < f(\pi)^{-1}_n\). Now, fix an occurrence \(a b c d\) of \(1342\) in \(f(\pi)\), where \(c = n+1\). Note that \(d\) must come after \(n\) in \(f(\pi)\), as otherwise \(a b d n\) forms a \(1324\)-pattern. But this means that \(a b n d\) is an occurrence of \(1342\) in \(f(\pi)\) that does not use the entry \(n+1\), which is a contradiction. Hence, \(\pi \delete n\) must be indecomposable.
  \end{itemize}
  
  In conclusion, we have shown that if \(f(\pi)\) contains \(1342\), then \(\pi \delete \pi_1\), \(\pi \delete 1\) and \(\pi \delete n\) are all indecomposable. So, \(\pi \delete \pi_n\) must be decomposable. We further claim that \(\pi_1 > \pi_n\) must hold. Indeed, since \(\pi\) is indecomposable, there would otherwise be some \(i < \pi^{-1}_n\) such that \(\pi_i > \pi_n\), and then \(\pi_1 \pi_i n \pi_n\) forms a \(1342\)-pattern. Furthermore, \(\pi^{-1}_1 < \pi^{-1}_n\) since \(\pi \delete \pi_n\) is decomposable.

  \begin{figure}[t]
    \begin{minipage}[t]{0.42\textwidth}\vspace{0pt}
      \begin{tikzpicture}[scale=0.8]
        \draw (0,0) rectangle (6,6);
        \fill[black!20] (2,2) rectangle (4,6);

        \foreach \x\y in {0/4,2/0,4/6,6/2} {
          \node[dot] at (\x,\y) {};
          \draw
            (\x,0) -- (\x,6)
            (0,\y) -- (6,\y);
        }

        \node[dot] at (3,1) {};
        \node[circ] at (3,1) {};
      \end{tikzpicture}
    \end{minipage}\hfill
    \begin{minipage}[t]{0.58\textwidth}\vspace{0pt}
      \caption{The structure of an almost decomposable permutation in \(\Av(1324,1342)\) whose image under \(f\) contains \(1342\).}
      \label{fig:1342 f incompatible}
    \end{minipage}
  \end{figure}

  See Figure \ref{fig:1342 f incompatible} for an illustration of the plot of such a permutation \(\pi\). Points can be contained in any of the white regions, but the shaded region must be empty, since a point there would create a \(1342\)-pattern together with \(1\), \(n\) and \(\pi_n\). Note that the four `boundary' points \(\pi_1\), \(1\), \(n\), \(\pi_n\) create \(2 \cdot 4 - 5 = 3\) inversions by themselves. Furthermore, any point in one of the other white regions must create at least two inversions with these boundary points. Lastly, any point \(\pi_i\) in the bottom middle region (containing a circled point in the figure) creates one inversion with \(\pi_1\), and we claim that it must also create an inversion with a non-boundary (`interior') point. Indeed, otherwise all points (except \(\pi_1\)) before \(\pi_i\) are smaller than it, and all points after \(\pi_i\) are larger, so \(\pi \delete \pi_1\) is decomposable. This is a contradiction.

  Construct \(\pi\) from the configuration of its four boundary points by adding the interior points one by one, so that the points in the bottom middle region are added last. The argument above shows that the number of inversions increases by at least two at each step, so \(\inv(\pi) \geq 2n - 5\) by induction.
\end{proof}

To conclude this section, we follow the methods of \cite{linusson_enumerating_2025} to enumerate the row differences
\begin{equation*}
  \av_{n+1}^k(1324, 1342) - \av_n^k(1324, 1342) 
\end{equation*}
for \(n \geq \frac{k+7}{2}\). Since we also know the limit sequence of \(\{1324, 1342\}\), this yields an enumeration of \(\Av_n^k(1324, 1342)\) for all \(n \geq \frac{k+7}{2}\). Recall first from \cite[Section~4]{linusson_enumerating_2025} that
\begin{equation*}
  \Av_{n+1}^k(1324) \setminus f(\Av_n^k(1324)) = R_1 \sqcup R_2 \sqcup R_3,
\end{equation*}
where the collections \(R_i\) are defined as follows.
\begin{itemize}
  \item \(R_1\) contains all permutations \(\sigma\) such that \(\sigma_1 = n+1\) or \(\sigma_{n+1} = 1\).
  \item \(R_2\) contains all permutations \(\sigma\) such that \(\sigma_2 = n+1\) or \(\sigma_{n+1} = 2\).
  \item \(R_3\) contains all permutations \(\sigma\) such that either \(\tau = \sigma\) or \(\tau = \sigma^{-1}\) has the following form: \(\comp(\tau \delete \tau_{n+1}) \geq 3\), \(\tau_1 = \ell+1\) and \(\tau_{n+1} = \ell+2\), where \(\ell\) is the length of the first component of \(\tau \delete \{\tau_1, \tau_{n+1}\}\).
\end{itemize}
We will show that the difference
\begin{equation} \label{eq:1342 row diff}
  \Av_{n+1}^k(1324, 1342) \setminus f(\Av_n^k(1324, 1342))
\end{equation}
consists of exactly the same collections intersected with \(\Av(1342)\).

\begin{lemma} \label{lem:1342 preserved}
  If \(\pi \in \Av_n(1324)\) is decomposable or almost decomposable and contains \(1342\), then \(f(\pi)\) also contains \(1342\).
\end{lemma}

\begin{proof}
  This is clear when \(\pi\) is decomposable, so suppose that \(\pi\) is almost decomposable. Assume first that \(\pi \delete \pi_1\) is decomposable. The entries \(342\) of an occurrence of \(1342\) in \(\pi\) must be contained in the same component of \(\pi \delete \pi_1\), so it is clear that \(f(\pi)\) also contains \(1342\). The case where \(\pi \delete 1\) is decomposable is analogous.

  Now, suppose that \(\pi \delete \pi_1\) and \(\pi \delete 1\) are indecomposable, and that \(\pi \delete \pi_n\) is decomposable. In \(f(\pi)\) there must be an occurrence of \(12\) with both entries larger than \(f(\pi)_{n+1}\). Together with \(1\), these entries form a \(1342\)-pattern in \(f(\pi)\). Finally, assume that \(\pi \delete n\) is decomposable. The entries \(1\) and \(3\) in our \(1342\) must both come from the first component of \(\pi \delete n\), so the \(2\) must also come from the first component. But then the first component contains \(132\), which is impossible.
\end{proof}

Lemma \ref{lem:1342 preserved} implies that the collection \eqref{eq:1342 row diff} consists exactly of \(R_1\), \(R_2\) and \(R_3\) intersected with \(\Av(1342)\), as desired. It remains to enumerate these sets.

\begin{theorem} \label{thm:1342 enum}
  For every nonnegative integer \(k\) and \(n \geq \frac{k+7}{2}\), the difference \(\av_{n+1}^k(1324, 1342) - \av_n^k(1324, 1342)\) is nonnegative, and equals the coefficient of \(x^k\) in the generating function
  \begin{equation*}
    x^{n-1} (2 + 2x) C_{1324, 1342}(x) = x^{n-1} (2+2x) \cdot \prod_{i \geq 1} \frac{1 + x^i}{1 - x^i}.
  \end{equation*}
  In particular, when \(n \geq \frac{k+7}{2}\), 
  \begin{equation*}
    \av_n^k(1324, 1342) = [x^k] \left( \frac{1 - x - x^{n-1}(2 + 2x)}{1 - x} \cdot \prod_{i \geq 1} \frac{1 + x^i}{1 - x^i} \right).
  \end{equation*}
\end{theorem}

\begin{proof}
  The collection \(R_1 \cap \Av(1342)\) consists of permutations \(1 \ominus \pi\) and \(\pi \ominus 1\), where \(\pi \in \Av_n^{k-n}(1324, 1342)\) is arbitrary. Hence, its generating function in \(k\) agrees with \(2 x^n C_{1324,1342}(x)\) when \(k \leq 2n - 7\). In \(R_2 \cap \Av(1342)\), note that the case \(\sigma_{n+1} = 2\) is impossible: there exists a \(12\)-pattern in \(\sigma\) with both entries above \(\sigma_{n+1}\), which together with \(1\) form a \(1342\)-pattern. Hence, as in the previous case, \(R_2 \cap \Av(1342)\) yields the generating function \(x^{n-1} C_{1324, 1342}(x)\). The collection \(R_3 \cap \Av(1342)\) is similar: the case where \(\tau = \sigma\) (see the definition) is impossible, since there exists a \(12\)-pattern with both entries above \(\tau_{n+1}\). On the other hand, when \(\tau = \sigma^{-1}\), the decomposable permutation \(\sigma \delete \{1, n+1\} \in \Av_{n-1}^{k-n+1}(1324, 1342)\) is arbitrary, so the generating function of this collection also agrees with \(x^{n-1} C_{1324, 1342}(x)\) in the \(k \leq 2n - 7\) region. Summing up, we get the generating function
  \begin{equation*}
    x^{n-1} (2 + 2x) C_{1324, 1342}(x) = x^{n-1} (2+2x) \cdot \prod_{i \geq 1} \frac{1 + x^i}{1 - x^i}
  \end{equation*}
  with \(C_{1324, 1342}(x)\) from part \ref{lim pair:1342} of Theorem \ref{thm:1324 pair limits}.
\end{proof}

\section{Conclusions and open problems} \label{sec:conclusions}

\paragraph{Inversion monotone sets}

We presented the first proofs of inversion monotonicity in nontrivial cases, namely for the pair \(\{1324, 231\}\), and through the construction method in Section \ref{sec:building bases}. Together with the results on limit sequences and half-monotonicity, these form a first body of theory for inversion monotonicity outside the principal case. The next natural step would be to prove that a pair \(\{1324, p\}\) with \(p \in S_4\) is inversion monotone. Our injection for \(\{1324, 231\}\) is too intricate to easily be adapted to wider classes, so new ideas are needed.

\begin{problem}
   Prove that \(\{1324, p\}\) is inversion monotone for some pattern \(p \in S_4\).
\end{problem}

A different route is offered by Conjecture \ref{conj:building bases}.

\begin{problem}
  Prove that \(\{1324\} \cup B_2\) is inversion monotone, where \(B_2 = \{213\}^{\ell \ell}\) is the collection of \(20\) patterns of length five obtained from \(213\) by inserting a new first entry twice.
\end{problem}

\paragraph{Limit sequences}

Our analysis of the limit sequences of pairs \(\{1324, p\}\) relies on the identification of a decomposable \(1324\)-avoider with a pair of partitions. The same technique is applicable to any collection of patterns containing \(1324\), as the problem reduces to enumerating restricted partitions. However, the limit sequences of other collections are more mysterious. For example, the pattern \(1243\) has limit sequence
\begin{equation*}
  1, 2, 5, 10, 20, 37, 66, 114, 193, 317, \ldots,
\end{equation*}
which is not in the OEIS. The problem is essentially equivalent to enumerating \(I_k(1243)\), the indecomposable \(1243\)-avoiding permutations with \(k\) inversions.

\begin{problem}
  Determine the limit sequence of the pattern \(1243\).
\end{problem}

\paragraph{Higher limit sequences}

We were able to prove that the pair \(\{1324, 1342\}\) has a secondary limit sequence given by \((2+2x) C_{1324,1342}(x)\). The result follows from an analysis of the injection of Linusson and Verkama \cite{linusson_enumerating_2025} used to prove a similar result for \(1324\). However, it also appears that the pairs \(\{1324, p\}\) with \(p\) in
\begin{equation*}
  1243,\ 1432,\ 2341,\ 3421,\ 4321
\end{equation*}
have secondary limit sequences. The patterns \(1432\) and \(4321\) are \(f\)-compatible, so there is hope of applying the same method as for \(1342\). The other three patterns would require some other method. Similarly, we have no explanation for the tertiary limit sequences of the remaining pairs.

\begin{problem}
  Explain why every pair \(\{1324, p\}\) with \(p \in S_4\) has either a secondary or a tertiary limit sequence.
\end{problem}

\section*{Acknowledgements}

SL and EV are supported by the Swedish Research Council (VR) grant 2022-03875. HU would like to thank SL and EV for their hospitality at KTH.

\paragraph{Statement on AI use}\addcontentsline{toc}{section}{Statement on AI use}

GPT 5.6 Sol (OpenAI) suggested and proved Theorem~\ref{thm:incomp asymptotics} (as well as Corollary \ref{cor:almost all sets comp}) when prompted to find a satisfying payoff to Subsection~\ref{subsec:compatible patterns}; prior to this, the subsection contained only Theorems \ref{thm:incomp necessary} and \ref{thm:incomp sufficient}. All other ideas and results are due to the human authors, who take full responsibility for the content of the manuscript.

\printbibliography

@article{bevan_structural_2020,
	title = {A structural characterisation of Av(1324) and new bounds on its growth rate},
	volume = {88},
	doi = {10.1016/j.ejc.2020.103115},
	note = {103115},
	journaltitle = {Eur. J. Comb.},
	author = {Bevan, David and Brignall, Robert and Elvey Price, Andrew and Pantone, Jay},
	year = {2020},
	mrnumber = {4111727},
}

@article{johansson_using_2014,
	title = {Using functional equations to enumerate 1324-avoiding permutations},
	volume = {56},
	issn = {01968858},
	doi = {10.1016/j.aam.2014.01.006},
	pages = {20--34},
	journaltitle = {Adv. Appl. Math.},
	author = {Johansson, Fredrik and Nakamura, Brian},
	year = {2014},
}

@article{le_wilf_2005,
	title = {Wilf Classes of Pairs of Permutations of Length 4},
	volume = {12},
	doi = {10.37236/1922},
	pages = {\#R25},
	number = {1},
	journaltitle = {Electron. J. Comb.},
	author = {Le, Ian},
	year = {2005},
}

@thesis{verkama_inversions_2026,
	title = {Inversions in the 1324-avoiding permutations},
	type = {Licentiate thesis},
	author = {Verkama, Emil},
	school = {KTH Royal Institute of Technology},
	year = {2026},
	number = {TRITA-SCI-FOU 2026:07},
	isbn = {978-91-8106-559-6},
	pagetotal = {194},
}

@article{dokos_permutation_2012,
	title = {Permutation patterns and statistics},
	volume = {312},
	issn = {0012-365X},
	doi = {10.1016/j.disc.2012.05.014},
	pages = {2760--2775},
	number = {18},
	journaltitle = {Discrete Math.},
	author = {Dokos, Theodore and Dwyer, Tim and Johnson, Bryan P. and Sagan, Bruce E. and Selsor, Kimberly},
	year = {2012},
}

@article{chan_infinite_2015,
	title = {An infinite family of inv-Wilf-equivalent permutation pairs},
	volume = {44},
	issn = {0195-6698},
	doi = {10.1016/j.ejc.2014.08.031},
	pages = {57--76},
	journaltitle = {Eur. J. Comb.},
	author = {Chan, Justin H. C.},
	year = {2015},
}

@book{stanley_enumerative_2012,
  author = {Stanley, Richard P.},
  title = {Enumerative combinatorics. {V}ol. 1},
  edition = {Second edition},
  publisher = {Cambridge University Press},
  year = {2012},
  pages = {xiv+626},
  isbn = {978-1-107-60262-5},
}

@article{linusson_enumerating_2025, 
	title = {Enumerating 1324-Avoiders with Few Inversions}, 
	volume = {32},
	doi = {10.37236/13387},
	number = {3}, 
	journal = {Electron. J. Comb.}, 
	author = {Linusson, Svante and Verkama, Emil}, 
	year = {2025},
	pages = {P3.44} 
}

@thesis{phan_structures_1999,
  title = {Structures ordonnées et dynamiques de piles de sable},
  type = {PhD thesis},
  institution = {Université Paris 7},
  author = {Phan, Thi Ha Duong},
  year = {1999},
  pagetotal = {110},
}

@article{bak_criticality_1988,
  title = {Self-organized criticality},
  author = {Bak, Per and Tang, Chao and Wiesenfeld, Kurt},
  journal = {Phys. Rev. A},
  volume = {38},
  issue = {1},
  pages = {364--374},
  year = {1988},
  publisher = {American Physical Society},
  doi = {10.1103/PhysRevA.38.364}
}

@article{franklin_restricted_2024,
	title = {Restricted Permutations Enumerated by Inversions},
	volume = {403},
	ISSN = {2075-2180},
	DOI = {10.4204/eptcs.403.21},
	journal = {Electron. Proc. Theor. Comput. Sci.},
	publisher = {Open Publishing Association},
	author = {Franklín, Atli Fannar and Claesson, Anders and Bean, Christian and Ulfarsson, Henning and Pantone, Jay},
	year = {2024},
	pages = {96--100} 
}

@article{franklin_pattern_2025,
  title = {Pattern Avoiding Permutations Enumerated by Inversions},
  volume = {27},
  number = {1},
  eid = {5},
  issn = {1365-8050},
  doi = {10.46298/dmtcs.14437},
  issuetitle = {Permutation Patterns 2024},
  journaltitle = {Discrete Math. Theor. Comput. Sci.},
  author = {Franklín, Atli Fannar},
  year = {2025},
}

@article{claesson_upper_2012,
	title = {Upper bounds for the Stanley-Wilf limit of 1324 and other layered patterns},
	volume = {119},
	doi = {10.1016/j.jcta.2012.05.006},
	pages = {1680--1691},
	number = {8},
	journaltitle = {J. Comb. Theory Ser. A},
	author = {Claesson, Anders and Jelínek, Vít and Steingrímsson, Einar},
	year = {2012},
	mrnumber = {2946382},
}

@article{marcus_excluded_2004,
	title = {Excluded permutation matrices and the Stanley-Wilf conjecture},
	volume = {107},
	doi = {10.1016/j.jcta.2004.04.002},
	pages = {153--160},
	number = {1},
	journaltitle = {J. Comb. Theory Ser. A},
	author = {Marcus, Adam and Tardos, Gábor},
	year = {2004},
	mrnumber = {2063960},
}

@article{conway_1324-avoiding_2018,
	title = {1324-avoiding permutations revisited},
	volume = {96},
	doi = {10.1016/j.aam.2018.01.002},
	pages = {312--333},
	journaltitle = {Adv. Appl. Math.},
	author = {Conway, Andrew R. and Guttmann, Anthony J. and Zinn-Justin, Paul},
	year = {2018},
	mrnumber = {3767512},
}

@misc{miner_enumeration_2016,
	title = {Enumeration of several two-by-four classes}, 
	author = {Sam Miner},
	year = {2016},
	eprint = {1610.01908},
	archivePrefix = {arXiv},
	primaryClass = {math.CO},
}

@article{combinatorial-exploration,
  title = {Combinatorial Exploration: An Algorithmic Framework for Enumeration},
  volume = {317},
  number = {1611},
  doi = {10.1090/memo/1611},
  journaltitle = {Mem. Am. Math. Soc.},
  author = {Albert, Michael H. and Bean, Christian and Claesson, Anders and Nadeau, Émile and Pantone, Jay and Ulfarsson, Henning},
  year = {2026},
  pagetotal = {109},
}

@article{lakshmibai_criterion_1990,
	title = {Criterion for smoothness of Schubert varieties in \(\mathrm{Sl}(n)/B\)},
	volume = {100},
	issn = {0973-7685},
	doi = {10.1007/BF02881113},
	pages = {45--52},
	number = {1},
	journaltitle = {Proc. Indian Acad. Sci. (Math. Sci.)},
	author = {Lakshmibai, V. and Sandhya, B.},
	year = {1990},
	langid = {english},
}

@article{bona_smooth_1998,
	title={The Permutation Classes Equinumerous to the Smooth Class}, 
	volume={5},
	DOI={10.37236/1369},
	number={1}, 
	journal={Electron. J. Comb.}, 
	author={Bóna, Miklós}, 
	year={1998},
	pages={R31} 
}

@article{albert_inflations_2014,
	title = {Inflations of geometric grid classes: three case studies},
	volume = {58},
	issn = {1034-4942,2202-3518},
	shorttitle = {Inflations of geometric grid classes},
	pages = {24--47},
	journaltitle = {Australas. J. Comb.},
	author = {Albert, Michael and Atkinson, Mike and Vatter, Vincent},
	year = {2014},
	mrnumber = {3211768},
}

@misc{albert_enumeration_2011,
	title = {The enumeration of permutations avoiding 2143 and 4231}, 
	author = {Albert, Michael and Atkinson, Mike and Brignall, Robert},
	year = {2011},
	eprint = {1108.0989},
	archivePrefix = {arXiv},
	primaryClass = {math.CO}
}

@article{albert_enumeration_2012, 
	title = {The Enumeration of Three Pattern Classes using Monotone Grid Classes}, 
	volume = {19},
	doi = {10.37236/2442},
	number = {3}, 
	journal = {Electron. J. Comb.}, 
	author={Albert, Michael and Atkinson, Mike and Brignall, Robert}, 
	year={2012}, 
	pages={P20} 
}

@article{albert_counting_2009, 
	title={Counting 1324, 4231-Avoiding Permutations}, 
	volume={16},
	DOI={10.37236/225},
	number={1}, 
	journal={Electron. J. Comb.}, 
	author={Albert, Michael and Atkinson, M. and Vatter, Vincent}, 
	year={2009},
	pages={R136} 
}

@article{vatter_finding_2012,
	title = {Finding regular insertion encodings for permutation classes},
	volume = {47},
	issn = {0747-7171},
	doi = {10.1016/j.jsc.2011.11.002},
	pages = {259--265},
	number = {3},
	journaltitle = {J. Symb. Comput.},
	author = {Vatter, Vincent},
	year = {2012},
}

@article{corteel_enumeration_2002,
	title = {Enumeration of sand piles},
	volume = {256},
	issn = {0012-365X},
	doi = {10.1016/S0012-365X(02)00339-4},
	series = {{LaCIM} 2000 Conference on Combinatorics, Computer Science and Applications},
	pages = {625--643},
	number = {3},
	journaltitle = {Discrete Math.},
	author = {Corteel, Sylvie and Gouyou-Beauchamps, Dominique},
	year = {2002}
}

\end{document}

\maketitle

\thispagestyle{empty}

\noindent
This supplement collects the numerical data referred to in the article ``Inversion monotonicity in subclasses of the 1324-avoiders''. Notation and terminology are as in the article.

Table \ref{tab:invmono pairs} records the status of inversion monotonicity for all pairs of length-four patterns. Table \ref{tab:F witnesses} shows where inversion monotonicity fails for the pairs that are marked `false' in the former table.

For each of the eleven patterns \(p \in S_4\) analysed in the article, namely
\begin{equation*}
  1243,\ 1342,\ 1432,\ 2143,\ 2341,\ 2413,\ 2431,\ 3412,\ 3421,\ 4231,\ 4321,
\end{equation*}
we give two tables: the values of \(\av_n^k(1324, p)\), and the row differences 
\begin{equation*}
  \av_{n+1}^k(1324, p) - \av_n^k(1324, p),
\end{equation*}
in both cases for \(n, k \leq 15\). An entry is left blank when no permutation of length \(n\) has \(k\) inversions, that is, when \(k > \binom{n}{2}\).

The limit sequence of \(\{1324, p\}\) can be read off from the columns of the first table for \(p\), which are eventually constant, and the secondary and tertiary limit sequences discussed in the article are visible in the second table. Tables are numbered \(\mathrm{S}1, \mathrm{S}2, \ldots\) and are cited by these numbers in the article.

\clearpage

\begin{table}[t]
  \centering
  \tiny
  \caption{Inversion monotonicity of the \(276\) pairs of length-four patterns. T: trivially inversion monotone (\(215\) pairs). F: \emph{not} inversion monotone (\(33\) pairs), each certified by Table \ref{tab:F witnesses}. \(\frac12\): inversion monotone whenever \(n \geq \frac{k+7}{2}\) (\(8\) pairs, Section~7 of the article); status open in general. O: status open (\(20\) pairs); no failure of inversion monotonicity occurs for \(n \leq 15\).}
  \begin{tblr}{ 
    colsep=1.8pt,
    rowsep=2pt,
    cell{-}{-} = {c},
    hline{2-25}={2-25}{},
    vline{2-26}={2-24}{},
    hline{2}={2-25}{1pt},
    vline{2}={2-24}{1pt}
    }
      ~ & 1243 & 1324 & 1342 & 1423 & 1432 & 2134 & 2143 & 2314 & 2341 & 2413 & 2431 & 3124 & 3142 & 3214 & 3241 & 3412 & 3421 & 4123 & 4132 & 4213 & 4231 & 4312 & 4321 \\ 
      1234 & F & F & F & F & F & F & F & F & F & F & F & F & F & F & F & F & F & F & F & F & F & F & F \\ 
      1243 & ~ & F & T & T & T & F & T & F & T & T & T & F & T & F & T & T & T & T & T & T & T & T & T \\ 
      1324 & ~ & ~ & \(\frac12\) & \(\frac12\) & \(\frac12\) & F & F & \(\frac12\) & O & O & O & \(\frac12\) & O & \(\frac12\) & O & O & O & O & O & O & \(\frac12\) & O & \(\frac12\) \\ 
      1342 & ~ & ~ & ~ & T & T & F & T & O & T & T & T & O & T & O & T & T & T & T & T & T & T & T & T \\ 
      1423 & ~ & ~ & ~ & ~ & T & F & T & O & T & T & T & O & T & O & T & T & T & T & T & T & T & T & T \\ 
      1432 & ~ & ~ & ~ & ~ & ~ & F & T & O & T & T & T & O & T & O & T & T & T & T & T & T & T & T & T \\ 
      2134 & ~ & ~ & ~ & ~ & ~ & ~ & T & T & T & T & T & T & T & T & T & T & T & T & T & T & T & T & T \\ 
      2143 & ~ & ~ & ~ & ~ & ~ & ~ & ~ & T & T & T & T & T & T & T & T & T & T & T & T & T & T & T & T \\ 
      2314 & ~ & ~ & ~ & ~ & ~ & ~ & ~ & ~ & T & T & T & T & T & T & T & T & T & T & T & T & T & T & T \\ 
      2341 & ~ & ~ & ~ & ~ & ~ & ~ & ~ & ~ & ~ & T & T & T & T & T & T & T & T & T & T & T & T & T & T \\ 
      2413 & ~ & ~ & ~ & ~ & ~ & ~ & ~ & ~ & ~ & ~ & T & T & T & T & T & T & T & T & T & T & T & T & T \\ 
      2431 & ~ & ~ & ~ & ~ & ~ & ~ & ~ & ~ & ~ & ~ & ~ & T & T & T & T & T & T & T & T & T & T & T & T \\ 
      3124 & ~ & ~ & ~ & ~ & ~ & ~ & ~ & ~ & ~ & ~ & ~ & ~ & T & T & T & T & T & T & T & T & T & T & T \\ 
      3142 & ~ & ~ & ~ & ~ & ~ & ~ & ~ & ~ & ~ & ~ & ~ & ~ & ~ & T & T & T & T & T & T & T & T & T & T \\ 
      3214 & ~ & ~ & ~ & ~ & ~ & ~ & ~ & ~ & ~ & ~ & ~ & ~ & ~ & ~ & T & T & T & T & T & T & T & T & T \\ 
      3241 & ~ & ~ & ~ & ~ & ~ & ~ & ~ & ~ & ~ & ~ & ~ & ~ & ~ & ~ & ~ & T & T & T & T & T & T & T & T \\ 
      3412 & ~ & ~ & ~ & ~ & ~ & ~ & ~ & ~ & ~ & ~ & ~ & ~ & ~ & ~ & ~ & ~ & T & T & T & T & T & T & T \\ 
      3421 & ~ & ~ & ~ & ~ & ~ & ~ & ~ & ~ & ~ & ~ & ~ & ~ & ~ & ~ & ~ & ~ & ~ & T & T & T & T & T & T \\ 
      4123 & ~ & ~ & ~ & ~ & ~ & ~ & ~ & ~ & ~ & ~ & ~ & ~ & ~ & ~ & ~ & ~ & ~ & ~ & T & T & T & T & T \\ 
      4132 & ~ & ~ & ~ & ~ & ~ & ~ & ~ & ~ & ~ & ~ & ~ & ~ & ~ & ~ & ~ & ~ & ~ & ~ & ~ & T & T & T & T \\ 
      4213 & ~ & ~ & ~ & ~ & ~ & ~ & ~ & ~ & ~ & ~ & ~ & ~ & ~ & ~ & ~ & ~ & ~ & ~ & ~ & ~ & T & T & T \\ 
      4231 & ~ & ~ & ~ & ~ & ~ & ~ & ~ & ~ & ~ & ~ & ~ & ~ & ~ & ~ & ~ & ~ & ~ & ~ & ~ & ~ & ~ & T & T \\ 
      4312 & ~ & ~ & ~ & ~ & ~ & ~ & ~ & ~ & ~ & ~ & ~ & ~ & ~ & ~ & ~ & ~ & ~ & ~ & ~ & ~ & ~ & ~ & T
  \end{tblr}
  \label{tab:invmono pairs}
\end{table}

\begin{table}[t]
  \centering
  \caption{Failures of inversion monotonicity for the pairs marked F in Table \ref{tab:invmono pairs}. Each entry gives values \(n\) and \(k\) with \(\av_n^k > \av_{n+1}^k\) for the pair in question. The remaining \(23\) entries marked F are the pairs \(\{1234, p\}\) with \(p \neq 1234\), for which \(\av_4^1 \geq 2 > 0 = \av_5^1\) in every case.}
  \begin{booktabs}{
    rows = {mode=math},
    column{2,4} = {leftsep=14pt},
    column{3} = {leftsep=22pt},
    rowsep = 2pt
  } \toprule
    \text{Pair} & \text{Witness} & \text{Pair} & \text{Witness} \\ \midrule
    1243, 1324 & \av_3^1 = 2 > 1 & 1324, 2134 & \av_3^1 = 2 > 1 \\
    1243, 2134 & \av_3^1 = 2 > 1 & 1324, 2143 & \av_5^4 = 12 > 10 \\
    1243, 2314 & \av_4^2 = 4 > 3 & 1342, 2134 & \av_4^2 = 4 > 3 \\
    1243, 3124 & \av_4^2 = 4 > 3 & 1423, 2134 & \av_4^2 = 4 > 3 \\
    1243, 3214 & \av_6^6 = 27 > 20 & 1432, 2134 & \av_6^6 = 27 > 20 \\ \bottomrule
  \end{booktabs}
  \label{tab:F witnesses}
\end{table}

\clearpage

%1243
\begin{table}[p]
  \centering
  \scriptsize
  \caption{The values of \(\av_n^k(1324, 1243)\) for \(n, k \leq 15\).}
  \begin{tblr}{
    rows = {mode=math},
    colsep = 5pt,
    cell{-}{-} = {c},
    hline{2,17} = {2-17}{},
    vline{2,18} = {2-16}{},
  }
    n \backslash k & 0 & 1 & 2 & 3 & 4 & 5 & 6 & 7 & 8 & 9 & 10 & 11 & 12 & 13 & 14 & 15 \\ 
    1 & 1 & ~ & ~ & ~ & ~ & ~ & ~ & ~ & ~ & ~ & ~ & ~ & ~ & ~ & ~ & ~ \\ 
    2 & 1 & 1 & ~ & ~ & ~ & ~ & ~ & ~ & ~ & ~ & ~ & ~ & ~ & ~ & ~ & ~ \\ 
    3 & 1 & 2 & 2 & 1 & ~ & ~ & ~ & ~ & ~ & ~ & ~ & ~ & ~ & ~ & ~ & ~ \\ 
    4 & 1 & 1 & 5 & 6 & 5 & 3 & 1 & ~ & ~ & ~ & ~ & ~ & ~ & ~ & ~ & ~ \\ 
    5 & 1 & 1 & 2 & 7 & 12 & 18 & 20 & 15 & 9 & 4 & 1 & ~ & ~ & ~ & ~ & ~ \\ 
    6 & 1 & 1 & 2 & 3 & 9 & 14 & 27 & 43 & 61 & 70 & 65 & 49 & 29 & 14 & 5 & 1 \\ 
    7 & 1 & 1 & 2 & 3 & 5 & 11 & 17 & 26 & 50 & 82 & 122 & 177 & 226 & 262 & 263 & 223 \\ 
    8 & 1 & 1 & 2 & 3 & 5 & 7 & 15 & 21 & 30 & 48 & 80 & 125 & 198 & 290 & 429 & 593 \\ 
    9 & 1 & 1 & 2 & 3 & 5 & 7 & 11 & 19 & 28 & 38 & 56 & 80 & 124 & 185 & 272 & 401 \\ 
    10 & 1 & 1 & 2 & 3 & 5 & 7 & 11 & 15 & 26 & 36 & 50 & 70 & 97 & 133 & 195 & 273 \\ 
    11 & 1 & 1 & 2 & 3 & 5 & 7 & 11 & 15 & 22 & 34 & 48 & 64 & 91 & 121 & 163 & 222 \\ 
    12 & 1 & 1 & 2 & 3 & 5 & 7 & 11 & 15 & 22 & 30 & 46 & 62 & 85 & 115 & 155 & 204 \\ 
    13 & 1 & 1 & 2 & 3 & 5 & 7 & 11 & 15 & 22 & 30 & 42 & 60 & 83 & 109 & 149 & 196 \\ 
    14 & 1 & 1 & 2 & 3 & 5 & 7 & 11 & 15 & 22 & 30 & 42 & 56 & 81 & 107 & 143 & 190 \\ 
    15 & 1 & 1 & 2 & 3 & 5 & 7 & 11 & 15 & 22 & 30 & 42 & 56 & 77 & 105 & 141 & 184
  \end{tblr}
  \label{tab:1324 1243}
\end{table}

\begin{table}[p]
  \centering
  \scriptsize
  \caption{The values of \(\av_{n+1}^k(1324, 1243) - \av_n^k(1324, 1243)\) for \(n, k \leq 15\).}
  \begin{tblr}{
    rows = {mode=math},
    colsep = 5pt,
    cell{-}{-} = {c},
    hline{2,16} = {2-17}{},
    vline{2,18} = {2-15}{},
  }
    n \backslash k & 0 & 1 & 2 & 3 & 4 & 5 & 6 & 7 & 8 & 9 & 10 & 11 & 12 & 13 & 14 & 15 \\ 
    1 & 0 & 1 & ~ & ~ & ~ & ~ & ~ & ~ & ~ & ~ & ~ & ~ & ~ & ~ & ~ & ~ \\ 
    2 & 0 & 1 & 2 & 1 & ~ & ~ & ~ & ~ & ~ & ~ & ~ & ~ & ~ & ~ & ~ & ~ \\ 
    3 & 0 & -1 & 3 & 5 & 5 & 3 & 1 & ~ & ~ & ~ & ~ & ~ & ~ & ~ & ~ & ~ \\ 
    4 & 0 & 0 & -3 & 1 & 7 & 15 & 19 & 15 & 9 & 4 & 1 & ~ & ~ & ~ & ~ & ~ \\ 
    5 & 0 & 0 & 0 & -4 & -3 & -4 & 7 & 28 & 52 & 66 & 64 & 49 & 29 & 14 & 5 & 1 \\ 
    6 & 0 & 0 & 0 & 0 & -4 & -3 & -10 & -17 & -11 & 12 & 57 & 128 & 197 & 248 & 258 & 222 \\ 
    7 & 0 & 0 & 0 & 0 & 0 & -4 & -2 & -5 & -20 & -34 & -42 & -52 & -28 & 28 & 166 & 370 \\ 
    8 & 0 & 0 & 0 & 0 & 0 & 0 & -4 & -2 & -2 & -10 & -24 & -45 & -74 & -105 & -157 & -192 \\ 
    9 & 0 & 0 & 0 & 0 & 0 & 0 & 0 & -4 & -2 & -2 & -6 & -10 & -27 & -52 & -77 & -128 \\ 
    10 & 0 & 0 & 0 & 0 & 0 & 0 & 0 & 0 & -4 & -2 & -2 & -6 & -6 & -12 & -32 & -51 \\ 
    11 & 0 & 0 & 0 & 0 & 0 & 0 & 0 & 0 & 0 & -4 & -2 & -2 & -6 & -6 & -8 & -18 \\ 
    12 & 0 & 0 & 0 & 0 & 0 & 0 & 0 & 0 & 0 & 0 & -4 & -2 & -2 & -6 & -6 & -8 \\ 
    13 & 0 & 0 & 0 & 0 & 0 & 0 & 0 & 0 & 0 & 0 & 0 & -4 & -2 & -2 & -6 & -6 \\ 
    14 & 0 & 0 & 0 & 0 & 0 & 0 & 0 & 0 & 0 & 0 & 0 & 0 & -4 & -2 & -2 & -6
  \end{tblr}
  \label{tab:1324 1243 diffs}
\end{table}

% 2143
\begin{table}[p]
  \centering
  \scriptsize
  \caption{The values of \(\av_n^k(1324, 2143)\) for \(n, k \leq 15\).}
  \begin{tblr}{
    rows = {mode=math},
    colsep = 5pt,
    cell{-}{-} = {c},
    hline{2,17} = {2-17}{},
    vline{2,18} = {2-16}{},
  }
    n \backslash k & 0 & 1 & 2 & 3 & 4 & 5 & 6 & 7 & 8 & 9 & 10 & 11 & 12 & 13 & 14 & 15 \\ 
    1 & 1 & ~ & ~ & ~ & ~ & ~ & ~ & ~ & ~ & ~ & ~ & ~ & ~ & ~ & ~ & ~ \\ 
    2 & 1 & 1 & ~ & ~ & ~ & ~ & ~ & ~ & ~ & ~ & ~ & ~ & ~ & ~ & ~ & ~ \\ 
    3 & 1 & 2 & 2 & 1 & ~ & ~ & ~ & ~ & ~ & ~ & ~ & ~ & ~ & ~ & ~ & ~ \\ 
    4 & 1 & 2 & 4 & 6 & 5 & 3 & 1 & ~ & ~ & ~ & ~ & ~ & ~ & ~ & ~ & ~ \\ 
    5 & 1 & 2 & 4 & 6 & 12 & 16 & 18 & 15 & 9 & 4 & 1 & ~ & ~ & ~ & ~ & ~ \\ 
    6 & 1 & 2 & 4 & 6 & 10 & 18 & 26 & 36 & 50 & 60 & 58 & 46 & 29 & 14 & 5 & 1 \\ 
    7 & 1 & 2 & 4 & 6 & 10 & 14 & 28 & 36 & 52 & 70 & 104 & 135 & 168 & 200 & 212 & 193 \\ 
    8 & 1 & 2 & 4 & 6 & 10 & 14 & 22 & 38 & 52 & 70 & 96 & 130 & 184 & 245 & 310 & 400 \\ 
    9 & 1 & 2 & 4 & 6 & 10 & 14 & 22 & 30 & 54 & 70 & 96 & 126 & 174 & 224 & 318 & 403 \\ 
    10 & 1 & 2 & 4 & 6 & 10 & 14 & 22 & 30 & 44 & 72 & 96 & 126 & 170 & 224 & 294 & 386 \\ 
    11 & 1 & 2 & 4 & 6 & 10 & 14 & 22 & 30 & 44 & 60 & 98 & 126 & 170 & 220 & 294 & 378 \\ 
    12 & 1 & 2 & 4 & 6 & 10 & 14 & 22 & 30 & 44 & 60 & 84 & 128 & 170 & 220 & 290 & 378 \\ 
    13 & 1 & 2 & 4 & 6 & 10 & 14 & 22 & 30 & 44 & 60 & 84 & 112 & 172 & 220 & 290 & 374 \\ 
    14 & 1 & 2 & 4 & 6 & 10 & 14 & 22 & 30 & 44 & 60 & 84 & 112 & 154 & 222 & 290 & 374 \\ 
    15 & 1 & 2 & 4 & 6 & 10 & 14 & 22 & 30 & 44 & 60 & 84 & 112 & 154 & 202 & 292 & 374
  \end{tblr}
  \label{tab:1324 2143}
\end{table}

\begin{table}[p]
  \centering
  \scriptsize
  \caption{The values of \(\av_{n+1}^k(1324, 2143) - \av_n^k(1324, 2143)\) for \(n, k \leq 15\).}
  \begin{tblr}{
    rows = {mode=math},
    colsep = 5pt,
    cell{-}{-} = {c},
    hline{2,16} = {2-17}{},
    vline{2,18} = {2-15}{},
  }
    n \backslash k & 0 & 1 & 2 & 3 & 4 & 5 & 6 & 7 & 8 & 9 & 10 & 11 & 12 & 13 & 14 & 15 \\ 
    1 & 0 & 1 & ~ & ~ & ~ & ~ & ~ & ~ & ~ & ~ & ~ & ~ & ~ & ~ & ~ & ~ \\ 
    2 & 0 & 1 & 2 & 1 & ~ & ~ & ~ & ~ & ~ & ~ & ~ & ~ & ~ & ~ & ~ & ~ \\ 
    3 & 0 & 0 & 2 & 5 & 5 & 3 & 1 & ~ & ~ & ~ & ~ & ~ & ~ & ~ & ~ & ~ \\ 
    4 & 0 & 0 & 0 & 0 & 7 & 13 & 17 & 15 & 9 & 4 & 1 & ~ & ~ & ~ & ~ & ~ \\ 
    5 & 0 & 0 & 0 & 0 & -2 & 2 & 8 & 21 & 41 & 56 & 57 & 46 & 29 & 14 & 5 & 1 \\ 
    6 & 0 & 0 & 0 & 0 & 0 & -4 & 2 & 0 & 2 & 10 & 46 & 89 & 139 & 186 & 207 & 192 \\ 
    7 & 0 & 0 & 0 & 0 & 0 & 0 & -6 & 2 & 0 & 0 & -8 & -5 & 16 & 45 & 98 & 207 \\ 
    8 & 0 & 0 & 0 & 0 & 0 & 0 & 0 & -8 & 2 & 0 & 0 & -4 & -10 & -21 & 8 & 3 \\ 
    9 & 0 & 0 & 0 & 0 & 0 & 0 & 0 & 0 & -10 & 2 & 0 & 0 & -4 & 0 & -24 & -17 \\ 
    10 & 0 & 0 & 0 & 0 & 0 & 0 & 0 & 0 & 0 & -12 & 2 & 0 & 0 & -4 & 0 & -8 \\ 
    11 & 0 & 0 & 0 & 0 & 0 & 0 & 0 & 0 & 0 & 0 & -14 & 2 & 0 & 0 & -4 & 0 \\ 
    12 & 0 & 0 & 0 & 0 & 0 & 0 & 0 & 0 & 0 & 0 & 0 & -16 & 2 & 0 & 0 & -4 \\ 
    13 & 0 & 0 & 0 & 0 & 0 & 0 & 0 & 0 & 0 & 0 & 0 & 0 & -18 & 2 & 0 & 0 \\ 
    14 & 0 & 0 & 0 & 0 & 0 & 0 & 0 & 0 & 0 & 0 & 0 & 0 & 0 & -20 & 2 & 0
  \end{tblr}
  \label{tab:1324 2143 diffs}
\end{table}

% 1342 
\begin{table}[p]
  \centering
  \scriptsize
  \caption{The values of \(\av_n^k(1324, 1342)\) for \(n, k \leq 15\).}
  \begin{tblr}{
    rows = {mode=math},
    colsep = 5pt,
    cell{-}{-} = {c},
    hline{2,17} = {2-17}{},
    vline{2,18} = {2-16}{},
  }
    n \backslash k & 0 & 1 & 2 & 3 & 4 & 5 & 6 & 7 & 8 & 9 & 10 & 11 & 12 & 13 & 14 & 15 \\ 
    1 & 1 & ~ & ~ & ~ & ~ & ~ & ~ & ~ & ~ & ~ & ~ & ~ & ~ & ~ & ~ & ~ \\ 
    2 & 1 & 1 & ~ & ~ & ~ & ~ & ~ & ~ & ~ & ~ & ~ & ~ & ~ & ~ & ~ & ~ \\ 
    3 & 1 & 2 & 2 & 1 & ~ & ~ & ~ & ~ & ~ & ~ & ~ & ~ & ~ & ~ & ~ & ~ \\ 
    4 & 1 & 2 & 4 & 6 & 5 & 3 & 1 & ~ & ~ & ~ & ~ & ~ & ~ & ~ & ~ & ~ \\ 
    5 & 1 & 2 & 4 & 8 & 12 & 16 & 18 & 15 & 9 & 4 & 1 & ~ & ~ & ~ & ~ & ~ \\ 
    6 & 1 & 2 & 4 & 8 & 14 & 22 & 32 & 44 & 54 & 60 & 58 & 46 & 29 & 14 & 5 & 1 \\ 
    7 & 1 & 2 & 4 & 8 & 14 & 24 & 38 & 56 & 80 & 110 & 142 & 175 & 204 & 220 & 218 & 193 \\ 
    8 & 1 & 2 & 4 & 8 & 14 & 24 & 40 & 62 & 92 & 134 & 188 & 256 & 338 & 431 & 534 & 642 \\ 
    9 & 1 & 2 & 4 & 8 & 14 & 24 & 40 & 64 & 98 & 146 & 212 & 300 & 416 & 564 & 746 & 967 \\ 
    10 & 1 & 2 & 4 & 8 & 14 & 24 & 40 & 64 & 100 & 152 & 224 & 324 & 460 & 640 & 876 & 1180 \\ 
    11 & 1 & 2 & 4 & 8 & 14 & 24 & 40 & 64 & 100 & 154 & 230 & 336 & 484 & 684 & 952 & 1308 \\ 
    12 & 1 & 2 & 4 & 8 & 14 & 24 & 40 & 64 & 100 & 154 & 232 & 342 & 496 & 708 & 996 & 1384 \\ 
    13 & 1 & 2 & 4 & 8 & 14 & 24 & 40 & 64 & 100 & 154 & 232 & 344 & 502 & 720 & 1020 & 1428 \\ 
    14 & 1 & 2 & 4 & 8 & 14 & 24 & 40 & 64 & 100 & 154 & 232 & 344 & 504 & 726 & 1032 & 1452 \\ 
    15 & 1 & 2 & 4 & 8 & 14 & 24 & 40 & 64 & 100 & 154 & 232 & 344 & 504 & 728 & 1038 & 1464
  \end{tblr}
  \label{tab:1324 1342}
\end{table}

\begin{table}[p]
  \centering
  \scriptsize
  \caption{The values of \(\av_{n+1}^k(1324, 1342) - \av_n^k(1324, 1342)\) for \(n, k \leq 15\).}
  \begin{tblr}{
    rows = {mode=math},
    colsep = 5pt,
    cell{-}{-} = {c},
    hline{2,16} = {2-17}{},
    vline{2,18} = {2-15}{},
  }
    n \backslash k & 0 & 1 & 2 & 3 & 4 & 5 & 6 & 7 & 8 & 9 & 10 & 11 & 12 & 13 & 14 & 15 \\ 
    1 & 0 & 1 & ~ & ~ & ~ & ~ & ~ & ~ & ~ & ~ & ~ & ~ & ~ & ~ & ~ & ~ \\ 
    2 & 0 & 1 & 2 & 1 & ~ & ~ & ~ & ~ & ~ & ~ & ~ & ~ & ~ & ~ & ~ & ~ \\ 
    3 & 0 & 0 & 2 & 5 & 5 & 3 & 1 & ~ & ~ & ~ & ~ & ~ & ~ & ~ & ~ & ~ \\ 
    4 & 0 & 0 & 0 & 2 & 7 & 13 & 17 & 15 & 9 & 4 & 1 & ~ & ~ & ~ & ~ & ~ \\ 
    5 & 0 & 0 & 0 & 0 & 2 & 6 & 14 & 29 & 45 & 56 & 57 & 46 & 29 & 14 & 5 & 1 \\ 
    6 & 0 & 0 & 0 & 0 & 0 & 2 & 6 & 12 & 26 & 50 & 84 & 129 & 175 & 206 & 213 & 192 \\ 
    7 & 0 & 0 & 0 & 0 & 0 & 0 & 2 & 6 & 12 & 24 & 46 & 81 & 134 & 211 & 316 & 449 \\ 
    8 & 0 & 0 & 0 & 0 & 0 & 0 & 0 & 2 & 6 & 12 & 24 & 44 & 78 & 133 & 212 & 325 \\ 
    9 & 0 & 0 & 0 & 0 & 0 & 0 & 0 & 0 & 2 & 6 & 12 & 24 & 44 & 76 & 130 & 213 \\ 
    10 & 0 & 0 & 0 & 0 & 0 & 0 & 0 & 0 & 0 & 2 & 6 & 12 & 24 & 44 & 76 & 128 \\ 
    11 & 0 & 0 & 0 & 0 & 0 & 0 & 0 & 0 & 0 & 0 & 2 & 6 & 12 & 24 & 44 & 76 \\ 
    12 & 0 & 0 & 0 & 0 & 0 & 0 & 0 & 0 & 0 & 0 & 0 & 2 & 6 & 12 & 24 & 44 \\ 
    13 & 0 & 0 & 0 & 0 & 0 & 0 & 0 & 0 & 0 & 0 & 0 & 0 & 2 & 6 & 12 & 24 \\ 
    14 & 0 & 0 & 0 & 0 & 0 & 0 & 0 & 0 & 0 & 0 & 0 & 0 & 0 & 2 & 6 & 12
  \end{tblr}
  \label{tab:1324 1342 diffs}
\end{table}

% 1432
\begin{table}[p]
  \centering
  \scriptsize
  \caption{The values of \(\av_n^k(1324, 1432)\) for \(n, k \leq 15\).}
  \begin{tblr}{
    rows = {mode=math},
    colsep = 5pt,
    cell{-}{-} = {c},
    hline{2,17} = {2-17}{},
    vline{2,18} = {2-16}{},
  }
    n \backslash k & 0 & 1 & 2 & 3 & 4 & 5 & 6 & 7 & 8 & 9 & 10 & 11 & 12 & 13 & 14 & 15 \\ 
    1 & 1 & ~ & ~ & ~ & ~ & ~ & ~ & ~ & ~ & ~ & ~ & ~ & ~ & ~ & ~ & ~ \\ 
    2 & 1 & 1 & ~ & ~ & ~ & ~ & ~ & ~ & ~ & ~ & ~ & ~ & ~ & ~ & ~ & ~ \\ 
    3 & 1 & 2 & 2 & 1 & ~ & ~ & ~ & ~ & ~ & ~ & ~ & ~ & ~ & ~ & ~ & ~ \\ 
    4 & 1 & 2 & 5 & 5 & 5 & 3 & 1 & ~ & ~ & ~ & ~ & ~ & ~ & ~ & ~ & ~ \\ 
    5 & 1 & 2 & 5 & 9 & 13 & 15 & 17 & 13 & 9 & 4 & 1 & ~ & ~ & ~ & ~ & ~ \\ 
    6 & 1 & 2 & 5 & 9 & 17 & 23 & 36 & 43 & 52 & 55 & 50 & 41 & 26 & 14 & 5 & 1 \\ 
    7 & 1 & 2 & 5 & 9 & 17 & 27 & 42 & 59 & 87 & 112 & 140 & 163 & 189 & 195 & 187 & 163 \\ 
    8 & 1 & 2 & 5 & 9 & 17 & 27 & 46 & 65 & 98 & 136 & 194 & 253 & 333 & 408 & 494 & 580 \\ 
    9 & 1 & 2 & 5 & 9 & 17 & 27 & 46 & 69 & 104 & 148 & 212 & 287 & 402 & 527 & 694 & 883 \\ 
    10 & 1 & 2 & 5 & 9 & 17 & 27 & 46 & 69 & 108 & 154 & 224 & 309 & 432 & 575 & 783 & 1026 \\ 
    11 & 1 & 2 & 5 & 9 & 17 & 27 & 46 & 69 & 108 & 158 & 230 & 321 & 454 & 613 & 833 & 1100 \\ 
    12 & 1 & 2 & 5 & 9 & 17 & 27 & 46 & 69 & 108 & 158 & 234 & 327 & 466 & 635 & 871 & 1162 \\ 
    13 & 1 & 2 & 5 & 9 & 17 & 27 & 46 & 69 & 108 & 158 & 234 & 331 & 472 & 647 & 893 & 1200 \\ 
    14 & 1 & 2 & 5 & 9 & 17 & 27 & 46 & 69 & 108 & 158 & 234 & 331 & 476 & 653 & 905 & 1222 \\ 
    15 & 1 & 2 & 5 & 9 & 17 & 27 & 46 & 69 & 108 & 158 & 234 & 331 & 476 & 657 & 911 & 1234
  \end{tblr}
  \label{tab:1324 1432}
\end{table}

\begin{table}[p]
  \centering
  \scriptsize
  \caption{The values of \(\av_{n+1}^k(1324, 1432) - \av_n^k(1324, 1432)\) for \(n, k \leq 15\).}
  \begin{tblr}{
    rows = {mode=math},
    colsep = 5pt,
    cell{-}{-} = {c},
    hline{2,16} = {2-17}{},
    vline{2,18} = {2-15}{},
  }
    n \backslash k & 0 & 1 & 2 & 3 & 4 & 5 & 6 & 7 & 8 & 9 & 10 & 11 & 12 & 13 & 14 & 15 \\ 
    1 & 0 & 1 & ~ & ~ & ~ & ~ & ~ & ~ & ~ & ~ & ~ & ~ & ~ & ~ & ~ & ~ \\ 
    2 & 0 & 1 & 2 & 1 & ~ & ~ & ~ & ~ & ~ & ~ & ~ & ~ & ~ & ~ & ~ & ~ \\ 
    3 & 0 & 0 & 3 & 4 & 5 & 3 & 1 & ~ & ~ & ~ & ~ & ~ & ~ & ~ & ~ & ~ \\ 
    4 & 0 & 0 & 0 & 4 & 8 & 12 & 16 & 13 & 9 & 4 & 1 & ~ & ~ & ~ & ~ & ~ \\ 
    5 & 0 & 0 & 0 & 0 & 4 & 8 & 19 & 30 & 43 & 51 & 49 & 41 & 26 & 14 & 5 & 1 \\ 
    6 & 0 & 0 & 0 & 0 & 0 & 4 & 6 & 16 & 35 & 57 & 90 & 122 & 163 & 181 & 182 & 162 \\ 
    7 & 0 & 0 & 0 & 0 & 0 & 0 & 4 & 6 & 11 & 24 & 54 & 90 & 144 & 213 & 307 & 417 \\ 
    8 & 0 & 0 & 0 & 0 & 0 & 0 & 0 & 4 & 6 & 12 & 18 & 34 & 69 & 119 & 200 & 303 \\ 
    9 & 0 & 0 & 0 & 0 & 0 & 0 & 0 & 0 & 4 & 6 & 12 & 22 & 30 & 48 & 89 & 143 \\ 
    10 & 0 & 0 & 0 & 0 & 0 & 0 & 0 & 0 & 0 & 4 & 6 & 12 & 22 & 38 & 50 & 74 \\ 
    11 & 0 & 0 & 0 & 0 & 0 & 0 & 0 & 0 & 0 & 0 & 4 & 6 & 12 & 22 & 38 & 62 \\ 
    12 & 0 & 0 & 0 & 0 & 0 & 0 & 0 & 0 & 0 & 0 & 0 & 4 & 6 & 12 & 22 & 38 \\ 
    13 & 0 & 0 & 0 & 0 & 0 & 0 & 0 & 0 & 0 & 0 & 0 & 0 & 4 & 6 & 12 & 22 \\ 
    14 & 0 & 0 & 0 & 0 & 0 & 0 & 0 & 0 & 0 & 0 & 0 & 0 & 0 & 4 & 6 & 12
  \end{tblr}
  \label{tab:1324 1432 diffs}
\end{table}

% 4231
\begin{table}[p]
  \centering
  \scriptsize
  \caption{The values of \(\av_n^k(1324, 4231)\) for \(n, k \leq 15\).}
  \begin{tblr}{
    rows = {mode=math},
    colsep = 5pt,
    cell{-}{-} = {c},
    hline{2,17} = {2-17}{},
    vline{2,18} = {2-16}{},
  }
    n \backslash k & 0 & 1 & 2 & 3 & 4 & 5 & 6 & 7 & 8 & 9 & 10 & 11 & 12 & 13 & 14 & 15 \\ 
    1 & 1 & ~ & ~ & ~ & ~ & ~ & ~ & ~ & ~ & ~ & ~ & ~ & ~ & ~ & ~ & ~ \\ 
    2 & 1 & 1 & ~ & ~ & ~ & ~ & ~ & ~ & ~ & ~ & ~ & ~ & ~ & ~ & ~ & ~ \\ 
    3 & 1 & 2 & 2 & 1 & ~ & ~ & ~ & ~ & ~ & ~ & ~ & ~ & ~ & ~ & ~ & ~ \\ 
    4 & 1 & 2 & 5 & 6 & 5 & 2 & 1 & ~ & ~ & ~ & ~ & ~ & ~ & ~ & ~ & ~ \\ 
    5 & 1 & 2 & 5 & 10 & 16 & 18 & 16 & 10 & 5 & 2 & 1 & ~ & ~ & ~ & ~ & ~ \\ 
    6 & 1 & 2 & 5 & 10 & 20 & 30 & 45 & 55 & 55 & 45 & 30 & 20 & 10 & 5 & 2 & 1 \\ 
    7 & 1 & 2 & 5 & 10 & 20 & 34 & 55 & 82 & 114 & 146 & 172 & 172 & 146 & 114 & 82 & 55 \\ 
    8 & 1 & 2 & 5 & 10 & 20 & 34 & 59 & 92 & 137 & 190 & 262 & 350 & 441 & 510 & 532 & 510 \\ 
    9 & 1 & 2 & 5 & 10 & 20 & 34 & 59 & 96 & 147 & 216 & 304 & 412 & 559 & 738 & 950 & 1188 \\ 
    10 & 1 & 2 & 5 & 10 & 20 & 34 & 59 & 96 & 151 & 226 & 332 & 462 & 627 & 842 & 1110 & 1448 \\ 
    11 & 1 & 2 & 5 & 10 & 20 & 34 & 59 & 96 & 151 & 230 & 342 & 492 & 681 & 924 & 1236 & 1618 \\ 
    12 & 1 & 2 & 5 & 10 & 20 & 34 & 59 & 96 & 151 & 230 & 346 & 502 & 713 & 982 & 1324 & 1768 \\ 
    13 & 1 & 2 & 5 & 10 & 20 & 34 & 59 & 96 & 151 & 230 & 346 & 506 & 723 & 1016 & 1386 & 1862 \\ 
    14 & 1 & 2 & 5 & 10 & 20 & 34 & 59 & 96 & 151 & 230 & 346 & 506 & 727 & 1026 & 1422 & 1928 \\ 
    15 & 1 & 2 & 5 & 10 & 20 & 34 & 59 & 96 & 151 & 230 & 346 & 506 & 727 & 1030 & 1432 & 1966
  \end{tblr}
  \label{tab:1324 4231}
\end{table}

\begin{table}[p]
  \centering
  \scriptsize
  \caption{The values of \(\av_{n+1}^k(1324, 4231) - \av_n^k(1324, 4231)\) for \(n, k \leq 15\).}
  \begin{tblr}{
    rows = {mode=math},
    colsep = 5pt,
    cell{-}{-} = {c},
    hline{2,16} = {2-17}{},
    vline{2,18} = {2-15}{},
  }
    n \backslash k & 0 & 1 & 2 & 3 & 4 & 5 & 6 & 7 & 8 & 9 & 10 & 11 & 12 & 13 & 14 & 15 \\ 
    1 & 0 & 1 & ~ & ~ & ~ & ~ & ~ & ~ & ~ & ~ & ~ & ~ & ~ & ~ & ~ & ~ \\ 
    2 & 0 & 1 & 2 & 1 & ~ & ~ & ~ & ~ & ~ & ~ & ~ & ~ & ~ & ~ & ~ & ~ \\ 
    3 & 0 & 0 & 3 & 5 & 5 & 2 & 1 & ~ & ~ & ~ & ~ & ~ & ~ & ~ & ~ & ~ \\ 
    4 & 0 & 0 & 0 & 4 & 11 & 16 & 15 & 10 & 5 & 2 & 1 & ~ & ~ & ~ & ~ & ~ \\ 
    5 & 0 & 0 & 0 & 0 & 4 & 12 & 29 & 45 & 50 & 43 & 29 & 20 & 10 & 5 & 2 & 1 \\ 
    6 & 0 & 0 & 0 & 0 & 0 & 4 & 10 & 27 & 59 & 101 & 142 & 152 & 136 & 109 & 80 & 54 \\ 
    7 & 0 & 0 & 0 & 0 & 0 & 0 & 4 & 10 & 23 & 44 & 90 & 178 & 295 & 396 & 450 & 455 \\ 
    8 & 0 & 0 & 0 & 0 & 0 & 0 & 0 & 4 & 10 & 26 & 42 & 62 & 118 & 228 & 418 & 678 \\ 
    9 & 0 & 0 & 0 & 0 & 0 & 0 & 0 & 0 & 4 & 10 & 28 & 50 & 68 & 104 & 160 & 260 \\ 
    10 & 0 & 0 & 0 & 0 & 0 & 0 & 0 & 0 & 0 & 4 & 10 & 30 & 54 & 82 & 126 & 170 \\ 
    11 & 0 & 0 & 0 & 0 & 0 & 0 & 0 & 0 & 0 & 0 & 4 & 10 & 32 & 58 & 88 & 150 \\ 
    12 & 0 & 0 & 0 & 0 & 0 & 0 & 0 & 0 & 0 & 0 & 0 & 4 & 10 & 34 & 62 & 94 \\ 
    13 & 0 & 0 & 0 & 0 & 0 & 0 & 0 & 0 & 0 & 0 & 0 & 0 & 4 & 10 & 36 & 66 \\ 
    14 & 0 & 0 & 0 & 0 & 0 & 0 & 0 & 0 & 0 & 0 & 0 & 0 & 0 & 4 & 10 & 38
  \end{tblr}
  \label{tab:1324 4231 diffs}
\end{table}

% 4321
\begin{table}[p]
  \centering
  \scriptsize
  \caption{The values of \(\av_n^k(1324, 4321)\) for \(n, k \leq 15\).}
  \begin{tblr}{
    rows = {mode=math},
    colsep = 5pt,
    cell{-}{-} = {c},
    hline{2,17} = {2-17}{},
    vline{2,18} = {2-16}{},
  }
    n \backslash k & 0 & 1 & 2 & 3 & 4 & 5 & 6 & 7 & 8 & 9 & 10 & 11 & 12 & 13 & 14 & 15 \\ 
    1 & 1 & ~ & ~ & ~ & ~ & ~ & ~ & ~ & ~ & ~ & ~ & ~ & ~ & ~ & ~ & ~ \\ 
    2 & 1 & 1 & ~ & ~ & ~ & ~ & ~ & ~ & ~ & ~ & ~ & ~ & ~ & ~ & ~ & ~ \\ 
    3 & 1 & 2 & 2 & 1 & ~ & ~ & ~ & ~ & ~ & ~ & ~ & ~ & ~ & ~ & ~ & ~ \\ 
    4 & 1 & 2 & 5 & 6 & 5 & 3 & 0 & ~ & ~ & ~ & ~ & ~ & ~ & ~ & ~ & ~ \\ 
    5 & 1 & 2 & 5 & 10 & 16 & 20 & 18 & 11 & 3 & 0 & 0 & ~ & ~ & ~ & ~ & ~ \\ 
    6 & 1 & 2 & 5 & 10 & 20 & 32 & 49 & 61 & 65 & 50 & 26 & 10 & 1 & 0 & 0 & 0 \\ 
    7 & 1 & 2 & 5 & 10 & 20 & 36 & 59 & 90 & 130 & 168 & 192 & 189 & 153 & 96 & 48 & 15 \\ 
    8 & 1 & 2 & 5 & 10 & 20 & 36 & 63 & 100 & 153 & 218 & 307 & 394 & 483 & 525 & 531 & 477 \\ 
    9 & 1 & 2 & 5 & 10 & 20 & 36 & 63 & 104 & 163 & 242 & 349 & 478 & 640 & 820 & 1012 & 1177 \\ 
    10 & 1 & 2 & 5 & 10 & 20 & 36 & 63 & 104 & 167 & 252 & 373 & 524 & 720 & 946 & 1233 & 1530 \\ 
    11 & 1 & 2 & 5 & 10 & 20 & 36 & 63 & 104 & 167 & 256 & 383 & 548 & 766 & 1034 & 1365 & 1738 \\ 
    12 & 1 & 2 & 5 & 10 & 20 & 36 & 63 & 104 & 167 & 256 & 387 & 558 & 790 & 1080 & 1453 & 1882 \\ 
    13 & 1 & 2 & 5 & 10 & 20 & 36 & 63 & 104 & 167 & 256 & 387 & 562 & 800 & 1104 & 1499 & 1970 \\ 
    14 & 1 & 2 & 5 & 10 & 20 & 36 & 63 & 104 & 167 & 256 & 387 & 562 & 804 & 1114 & 1523 & 2016 \\ 
    15 & 1 & 2 & 5 & 10 & 20 & 36 & 63 & 104 & 167 & 256 & 387 & 562 & 804 & 1118 & 1533 & 2040
  \end{tblr}
  \label{tab:1324 4321}
\end{table}

\begin{table}[p]
  \centering
  \scriptsize
  \caption{The values of \(\av_{n+1}^k(1324, 4321) - \av_n^k(1324, 4321)\) for \(n, k \leq 15\).}
  \begin{tblr}{
    rows = {mode=math},
    colsep = 5pt,
    cell{-}{-} = {c},
    hline{2,16} = {2-17}{},
    vline{2,18} = {2-15}{},
  }
    n \backslash k & 0 & 1 & 2 & 3 & 4 & 5 & 6 & 7 & 8 & 9 & 10 & 11 & 12 & 13 & 14 & 15 \\ 
    1 & 0 & 1 & ~ & ~ & ~ & ~ & ~ & ~ & ~ & ~ & ~ & ~ & ~ & ~ & ~ & ~ \\ 
    2 & 0 & 1 & 2 & 1 & ~ & ~ & ~ & ~ & ~ & ~ & ~ & ~ & ~ & ~ & ~ & ~ \\ 
    3 & 0 & 0 & 3 & 5 & 5 & 3 & 0 & ~ & ~ & ~ & ~ & ~ & ~ & ~ & ~ & ~ \\ 
    4 & 0 & 0 & 0 & 4 & 11 & 17 & 18 & 11 & 3 & 0 & 0 & ~ & ~ & ~ & ~ & ~ \\ 
    5 & 0 & 0 & 0 & 0 & 4 & 12 & 31 & 50 & 62 & 50 & 26 & 10 & 1 & 0 & 0 & 0 \\ 
    6 & 0 & 0 & 0 & 0 & 0 & 4 & 10 & 29 & 65 & 118 & 166 & 179 & 152 & 96 & 48 & 15 \\ 
    7 & 0 & 0 & 0 & 0 & 0 & 0 & 4 & 10 & 23 & 50 & 115 & 205 & 330 & 429 & 483 & 462 \\ 
    8 & 0 & 0 & 0 & 0 & 0 & 0 & 0 & 4 & 10 & 24 & 42 & 84 & 157 & 295 & 481 & 700 \\ 
    9 & 0 & 0 & 0 & 0 & 0 & 0 & 0 & 0 & 4 & 10 & 24 & 46 & 80 & 126 & 221 & 353 \\ 
    10 & 0 & 0 & 0 & 0 & 0 & 0 & 0 & 0 & 0 & 4 & 10 & 24 & 46 & 88 & 132 & 208 \\ 
    11 & 0 & 0 & 0 & 0 & 0 & 0 & 0 & 0 & 0 & 0 & 4 & 10 & 24 & 46 & 88 & 144 \\ 
    12 & 0 & 0 & 0 & 0 & 0 & 0 & 0 & 0 & 0 & 0 & 0 & 4 & 10 & 24 & 46 & 88 \\ 
    13 & 0 & 0 & 0 & 0 & 0 & 0 & 0 & 0 & 0 & 0 & 0 & 0 & 4 & 10 & 24 & 46 \\ 
    14 & 0 & 0 & 0 & 0 & 0 & 0 & 0 & 0 & 0 & 0 & 0 & 0 & 0 & 4 & 10 & 24
  \end{tblr}
  \label{tab:1324 4321 diffs}
\end{table}

% 2341
\begin{table}[p]
  \centering
  \scriptsize
  \caption{The values of \(\av_n^k(1324, 2341)\) for \(n, k \leq 15\).}
  \begin{tblr}{
    rows = {mode=math},
    colsep = 5pt,
    cell{-}{-} = {c},
    hline{2,17} = {2-17}{},
    vline{2,18} = {2-16}{},
  }
    n \backslash k & 0 & 1 & 2 & 3 & 4 & 5 & 6 & 7 & 8 & 9 & 10 & 11 & 12 & 13 & 14 & 15 \\ 
    1 & 1 & ~ & ~ & ~ & ~ & ~ & ~ & ~ & ~ & ~ & ~ & ~ & ~ & ~ & ~ & ~ \\ 
    2 & 1 & 1 & ~ & ~ & ~ & ~ & ~ & ~ & ~ & ~ & ~ & ~ & ~ & ~ & ~ & ~ \\ 
    3 & 1 & 2 & 2 & 1 & ~ & ~ & ~ & ~ & ~ & ~ & ~ & ~ & ~ & ~ & ~ & ~ \\ 
    4 & 1 & 2 & 5 & 5 & 5 & 3 & 1 & ~ & ~ & ~ & ~ & ~ & ~ & ~ & ~ & ~ \\ 
    5 & 1 & 2 & 5 & 8 & 13 & 16 & 15 & 13 & 9 & 4 & 1 & ~ & ~ & ~ & ~ & ~ \\ 
    6 & 1 & 2 & 5 & 8 & 16 & 23 & 32 & 39 & 48 & 51 & 44 & 37 & 26 & 14 & 5 & 1 \\ 
    7 & 1 & 2 & 5 & 8 & 16 & 26 & 39 & 56 & 77 & 98 & 114 & 131 & 150 & 161 & 155 & 133 \\ 
    8 & 1 & 2 & 5 & 8 & 16 & 26 & 42 & 63 & 94 & 129 & 171 & 214 & 266 & 319 & 369 & 411 \\ 
    9 & 1 & 2 & 5 & 8 & 16 & 26 & 42 & 66 & 101 & 146 & 202 & 274 & 366 & 464 & 574 & 693 \\ 
    10 & 1 & 2 & 5 & 8 & 16 & 26 & 42 & 66 & 104 & 153 & 219 & 305 & 426 & 568 & 739 & 940 \\ 
    11 & 1 & 2 & 5 & 8 & 16 & 26 & 42 & 66 & 104 & 156 & 226 & 322 & 457 & 628 & 843 & 1110 \\ 
    12 & 1 & 2 & 5 & 8 & 16 & 26 & 42 & 66 & 104 & 156 & 229 & 329 & 474 & 659 & 903 & 1214 \\ 
    13 & 1 & 2 & 5 & 8 & 16 & 26 & 42 & 66 & 104 & 156 & 229 & 332 & 481 & 676 & 934 & 1274 \\ 
    14 & 1 & 2 & 5 & 8 & 16 & 26 & 42 & 66 & 104 & 156 & 229 & 332 & 484 & 683 & 951 & 1305 \\ 
    15 & 1 & 2 & 5 & 8 & 16 & 26 & 42 & 66 & 104 & 156 & 229 & 332 & 484 & 686 & 958 & 1322
  \end{tblr}
  \label{tab:1324 2341}
\end{table}

\begin{table}[p]
  \centering
  \scriptsize
  \caption{The values of \(\av_{n+1}^k(1324, 2341) - \av_n^k(1324, 2341)\) for \(n, k \leq 15\).}
  \begin{tblr}{
    rows = {mode=math},
    colsep = 5pt,
    cell{-}{-} = {c},
    hline{2,16} = {2-17}{},
    vline{2,18} = {2-15}{},
  }
    n \backslash k & 0 & 1 & 2 & 3 & 4 & 5 & 6 & 7 & 8 & 9 & 10 & 11 & 12 & 13 & 14 & 15 \\ 
    1 & 0 & 1 & ~ & ~ & ~ & ~ & ~ & ~ & ~ & ~ & ~ & ~ & ~ & ~ & ~ & ~ \\ 
    2 & 0 & 1 & 2 & 1 & ~ & ~ & ~ & ~ & ~ & ~ & ~ & ~ & ~ & ~ & ~ & ~ \\ 
    3 & 0 & 0 & 3 & 4 & 5 & 3 & 1 & ~ & ~ & ~ & ~ & ~ & ~ & ~ & ~ & ~ \\ 
    4 & 0 & 0 & 0 & 3 & 8 & 13 & 14 & 13 & 9 & 4 & 1 & ~ & ~ & ~ & ~ & ~ \\ 
    5 & 0 & 0 & 0 & 0 & 3 & 7 & 17 & 26 & 39 & 47 & 43 & 37 & 26 & 14 & 5 & 1 \\ 
    6 & 0 & 0 & 0 & 0 & 0 & 3 & 7 & 17 & 29 & 47 & 70 & 94 & 124 & 147 & 150 & 132 \\ 
    7 & 0 & 0 & 0 & 0 & 0 & 0 & 3 & 7 & 17 & 31 & 57 & 83 & 116 & 158 & 214 & 278 \\ 
    8 & 0 & 0 & 0 & 0 & 0 & 0 & 0 & 3 & 7 & 17 & 31 & 60 & 100 & 145 & 205 & 282 \\ 
    9 & 0 & 0 & 0 & 0 & 0 & 0 & 0 & 0 & 3 & 7 & 17 & 31 & 60 & 104 & 165 & 247 \\ 
    10 & 0 & 0 & 0 & 0 & 0 & 0 & 0 & 0 & 0 & 3 & 7 & 17 & 31 & 60 & 104 & 170 \\ 
    11 & 0 & 0 & 0 & 0 & 0 & 0 & 0 & 0 & 0 & 0 & 3 & 7 & 17 & 31 & 60 & 104 \\ 
    12 & 0 & 0 & 0 & 0 & 0 & 0 & 0 & 0 & 0 & 0 & 0 & 3 & 7 & 17 & 31 & 60 \\ 
    13 & 0 & 0 & 0 & 0 & 0 & 0 & 0 & 0 & 0 & 0 & 0 & 0 & 3 & 7 & 17 & 31 \\ 
    14 & 0 & 0 & 0 & 0 & 0 & 0 & 0 & 0 & 0 & 0 & 0 & 0 & 0 & 3 & 7 & 17
  \end{tblr}
  \label{tab:1324 2341 diffs}
\end{table}

% 2413
\begin{table}[p]
  \centering
  \scriptsize
  \caption{The values of \(\av_n^k(1324, 2413)\) for \(n, k \leq 15\).}
  \begin{tblr}{
    rows = {mode=math},
    colsep = 5pt,
    cell{-}{-} = {c},
    hline{2,17} = {2-17}{},
    vline{2,18} = {2-16}{},
  }
    n \backslash k & 0 & 1 & 2 & 3 & 4 & 5 & 6 & 7 & 8 & 9 & 10 & 11 & 12 & 13 & 14 & 15 \\ 
    1 & 1 & ~ & ~ & ~ & ~ & ~ & ~ & ~ & ~ & ~ & ~ & ~ & ~ & ~ & ~ & ~ \\ 
    2 & 1 & 1 & ~ & ~ & ~ & ~ & ~ & ~ & ~ & ~ & ~ & ~ & ~ & ~ & ~ & ~ \\ 
    3 & 1 & 2 & 2 & 1 & ~ & ~ & ~ & ~ & ~ & ~ & ~ & ~ & ~ & ~ & ~ & ~ \\ 
    4 & 1 & 2 & 5 & 5 & 5 & 3 & 1 & ~ & ~ & ~ & ~ & ~ & ~ & ~ & ~ & ~ \\ 
    5 & 1 & 2 & 5 & 10 & 12 & 15 & 16 & 13 & 9 & 4 & 1 & ~ & ~ & ~ & ~ & ~ \\ 
    6 & 1 & 2 & 5 & 10 & 20 & 25 & 33 & 42 & 49 & 47 & 47 & 39 & 26 & 14 & 5 & 1 \\ 
    7 & 1 & 2 & 5 & 10 & 20 & 36 & 51 & 69 & 86 & 110 & 132 & 146 & 155 & 163 & 157 & 141 \\ 
    8 & 1 & 2 & 5 & 10 & 20 & 36 & 65 & 93 & 135 & 178 & 223 & 276 & 336 & 388 & 442 & 483 \\ 
    9 & 1 & 2 & 5 & 10 & 20 & 36 & 65 & 110 & 165 & 241 & 335 & 444 & 557 & 690 & 826 & 980 \\ 
    10 & 1 & 2 & 5 & 10 & 20 & 36 & 65 & 110 & 185 & 277 & 413 & 582 & 803 & 1056 & 1347 & 1671 \\ 
    11 & 1 & 2 & 5 & 10 & 20 & 36 & 65 & 110 & 185 & 300 & 455 & 675 & 971 & 1354 & 1837 & 2428 \\ 
    12 & 1 & 2 & 5 & 10 & 20 & 36 & 65 & 110 & 185 & 300 & 481 & 723 & 1079 & 1552 & 2195 & 3014 \\ 
    13 & 1 & 2 & 5 & 10 & 20 & 36 & 65 & 110 & 185 & 300 & 481 & 752 & 1133 & 1675 & 2423 & 3432 \\ 
    14 & 1 & 2 & 5 & 10 & 20 & 36 & 65 & 110 & 185 & 300 & 481 & 752 & 1165 & 1735 & 2561 & 3690 \\ 
    15 & 1 & 2 & 5 & 10 & 20 & 36 & 65 & 110 & 185 & 300 & 481 & 752 & 1165 & 1770 & 2627 & 3843
  \end{tblr}
  \label{tab:1324 2413}
\end{table}

\begin{table}[p]
  \centering
  \scriptsize
  \caption{The values of \(\av_{n+1}^k(1324, 2413) - \av_n^k(1324, 2413)\) for \(n, k \leq 15\).}
  \begin{tblr}{
    rows = {mode=math},
    colsep = 5pt,
    cell{-}{-} = {c},
    hline{2,16} = {2-17}{},
    vline{2,18} = {2-15}{},
  }
    n \backslash k & 0 & 1 & 2 & 3 & 4 & 5 & 6 & 7 & 8 & 9 & 10 & 11 & 12 & 13 & 14 & 15 \\ 
    1 & 0 & 1 & ~ & ~ & ~ & ~ & ~ & ~ & ~ & ~ & ~ & ~ & ~ & ~ & ~ & ~ \\ 
    2 & 0 & 1 & 2 & 1 & ~ & ~ & ~ & ~ & ~ & ~ & ~ & ~ & ~ & ~ & ~ & ~ \\ 
    3 & 0 & 0 & 3 & 4 & 5 & 3 & 1 & ~ & ~ & ~ & ~ & ~ & ~ & ~ & ~ & ~ \\ 
    4 & 0 & 0 & 0 & 5 & 7 & 12 & 15 & 13 & 9 & 4 & 1 & ~ & ~ & ~ & ~ & ~ \\ 
    5 & 0 & 0 & 0 & 0 & 8 & 10 & 17 & 29 & 40 & 43 & 46 & 39 & 26 & 14 & 5 & 1 \\ 
    6 & 0 & 0 & 0 & 0 & 0 & 11 & 18 & 27 & 37 & 63 & 85 & 107 & 129 & 149 & 152 & 140 \\ 
    7 & 0 & 0 & 0 & 0 & 0 & 0 & 14 & 24 & 49 & 68 & 91 & 130 & 181 & 225 & 285 & 342 \\ 
    8 & 0 & 0 & 0 & 0 & 0 & 0 & 0 & 17 & 30 & 63 & 112 & 168 & 221 & 302 & 384 & 497 \\ 
    9 & 0 & 0 & 0 & 0 & 0 & 0 & 0 & 0 & 20 & 36 & 78 & 138 & 246 & 366 & 521 & 691 \\ 
    10 & 0 & 0 & 0 & 0 & 0 & 0 & 0 & 0 & 0 & 23 & 42 & 93 & 168 & 298 & 490 & 757 \\ 
    11 & 0 & 0 & 0 & 0 & 0 & 0 & 0 & 0 & 0 & 0 & 26 & 48 & 108 & 198 & 358 & 586 \\ 
    12 & 0 & 0 & 0 & 0 & 0 & 0 & 0 & 0 & 0 & 0 & 0 & 29 & 54 & 123 & 228 & 418 \\ 
    13 & 0 & 0 & 0 & 0 & 0 & 0 & 0 & 0 & 0 & 0 & 0 & 0 & 32 & 60 & 138 & 258 \\ 
    14 & 0 & 0 & 0 & 0 & 0 & 0 & 0 & 0 & 0 & 0 & 0 & 0 & 0 & 35 & 66 & 153
  \end{tblr}
  \label{tab:1324 2413 diffs}
\end{table}

% 2431
\begin{table}[p]
  \centering
  \scriptsize
  \caption{The values of \(\av_n^k(1324, 2431)\) for \(n, k \leq 15\).}
  \begin{tblr}{
    rows = {mode=math},
    colsep = 5pt,
    cell{-}{-} = {c},
    hline{2,17} = {2-17}{},
    vline{2,18} = {2-16}{},
  }
    n \backslash k & 0 & 1 & 2 & 3 & 4 & 5 & 6 & 7 & 8 & 9 & 10 & 11 & 12 & 13 & 14 & 15 \\ 
    1 & 1 & ~ & ~ & ~ & ~ & ~ & ~ & ~ & ~ & ~ & ~ & ~ & ~ & ~ & ~ & ~ \\ 
    2 & 1 & 1 & ~ & ~ & ~ & ~ & ~ & ~ & ~ & ~ & ~ & ~ & ~ & ~ & ~ & ~ \\ 
    3 & 1 & 2 & 2 & 1 & ~ & ~ & ~ & ~ & ~ & ~ & ~ & ~ & ~ & ~ & ~ & ~ \\ 
    4 & 1 & 2 & 5 & 6 & 4 & 3 & 1 & ~ & ~ & ~ & ~ & ~ & ~ & ~ & ~ & ~ \\ 
    5 & 1 & 2 & 5 & 10 & 15 & 17 & 15 & 11 & 7 & 4 & 1 & ~ & ~ & ~ & ~ & ~ \\ 
    6 & 1 & 2 & 5 & 10 & 19 & 30 & 43 & 50 & 50 & 49 & 39 & 29 & 19 & 11 & 5 & 1 \\ 
    7 & 1 & 2 & 5 & 10 & 19 & 34 & 55 & 80 & 114 & 140 & 153 & 165 & 161 & 150 & 132 & 105 \\ 
    8 & 1 & 2 & 5 & 10 & 19 & 34 & 59 & 93 & 140 & 202 & 278 & 352 & 420 & 476 & 518 & 544 \\ 
    9 & 1 & 2 & 5 & 10 & 19 & 34 & 59 & 97 & 154 & 231 & 335 & 468 & 639 & 823 & 1014 & 1218 \\ 
    10 & 1 & 2 & 5 & 10 & 19 & 34 & 59 & 97 & 158 & 246 & 366 & 534 & 754 & 1033 & 1393 & 1799 \\ 
    11 & 1 & 2 & 5 & 10 & 19 & 34 & 59 & 97 & 158 & 250 & 382 & 567 & 825 & 1166 & 1615 & 2187 \\ 
    12 & 1 & 2 & 5 & 10 & 19 & 34 & 59 & 97 & 158 & 250 & 386 & 584 & 860 & 1242 & 1758 & 2440 \\ 
    13 & 1 & 2 & 5 & 10 & 19 & 34 & 59 & 97 & 158 & 250 & 386 & 588 & 878 & 1279 & 1839 & 2593 \\ 
    14 & 1 & 2 & 5 & 10 & 19 & 34 & 59 & 97 & 158 & 250 & 386 & 588 & 882 & 1298 & 1878 & 2679 \\ 
    15 & 1 & 2 & 5 & 10 & 19 & 34 & 59 & 97 & 158 & 250 & 386 & 588 & 882 & 1302 & 1898 & 2720
  \end{tblr}
  \label{tab:1324 2431}
\end{table}

\begin{table}[p]
  \centering
  \scriptsize
  \caption{The values of \(\av_{n+1}^k(1324, 2431) - \av_n^k(1324, 2431)\) for \(n, k \leq 15\).}
  \begin{tblr}{
    rows = {mode=math},
    colsep = 5pt,
    cell{-}{-} = {c},
    hline{2,16} = {2-17}{},
    vline{2,18} = {2-15}{},
  }
    n \backslash k & 0 & 1 & 2 & 3 & 4 & 5 & 6 & 7 & 8 & 9 & 10 & 11 & 12 & 13 & 14 & 15 \\ 
    1 & 0 & 1 & ~ & ~ & ~ & ~ & ~ & ~ & ~ & ~ & ~ & ~ & ~ & ~ & ~ & ~ \\ 
    2 & 0 & 1 & 2 & 1 & ~ & ~ & ~ & ~ & ~ & ~ & ~ & ~ & ~ & ~ & ~ & ~ \\ 
    3 & 0 & 0 & 3 & 5 & 4 & 3 & 1 & ~ & ~ & ~ & ~ & ~ & ~ & ~ & ~ & ~ \\ 
    4 & 0 & 0 & 0 & 4 & 11 & 14 & 14 & 11 & 7 & 4 & 1 & ~ & ~ & ~ & ~ & ~ \\ 
    5 & 0 & 0 & 0 & 0 & 4 & 13 & 28 & 39 & 43 & 45 & 38 & 29 & 19 & 11 & 5 & 1 \\ 
    6 & 0 & 0 & 0 & 0 & 0 & 4 & 12 & 30 & 64 & 91 & 114 & 136 & 142 & 139 & 127 & 104 \\ 
    7 & 0 & 0 & 0 & 0 & 0 & 0 & 4 & 13 & 26 & 62 & 125 & 187 & 259 & 326 & 386 & 439 \\ 
    8 & 0 & 0 & 0 & 0 & 0 & 0 & 0 & 4 & 14 & 29 & 57 & 116 & 219 & 347 & 496 & 674 \\ 
    9 & 0 & 0 & 0 & 0 & 0 & 0 & 0 & 0 & 4 & 15 & 31 & 66 & 115 & 210 & 379 & 581 \\ 
    10 & 0 & 0 & 0 & 0 & 0 & 0 & 0 & 0 & 0 & 4 & 16 & 33 & 71 & 133 & 222 & 388 \\ 
    11 & 0 & 0 & 0 & 0 & 0 & 0 & 0 & 0 & 0 & 0 & 4 & 17 & 35 & 76 & 143 & 253 \\ 
    12 & 0 & 0 & 0 & 0 & 0 & 0 & 0 & 0 & 0 & 0 & 0 & 4 & 18 & 37 & 81 & 153 \\ 
    13 & 0 & 0 & 0 & 0 & 0 & 0 & 0 & 0 & 0 & 0 & 0 & 0 & 4 & 19 & 39 & 86 \\ 
    14 & 0 & 0 & 0 & 0 & 0 & 0 & 0 & 0 & 0 & 0 & 0 & 0 & 0 & 4 & 20 & 41
  \end{tblr}
  \label{tab:1324 2431 diffs}
\end{table}

% 3412
\begin{table}[p]
  \centering
  \scriptsize
  \caption{The values of \(\av_n^k(1324, 3412)\) for \(n, k \leq 15\).}
  \begin{tblr}{
    rows = {mode=math},
    colsep = 5pt,
    cell{-}{-} = {c},
    hline{2,17} = {2-17}{},
    vline{2,18} = {2-16}{},
  }
    n \backslash k & 0 & 1 & 2 & 3 & 4 & 5 & 6 & 7 & 8 & 9 & 10 & 11 & 12 & 13 & 14 & 15 \\ 
    1 & 1 & ~ & ~ & ~ & ~ & ~ & ~ & ~ & ~ & ~ & ~ & ~ & ~ & ~ & ~ & ~ \\ 
    2 & 1 & 1 & ~ & ~ & ~ & ~ & ~ & ~ & ~ & ~ & ~ & ~ & ~ & ~ & ~ & ~ \\ 
    3 & 1 & 2 & 2 & 1 & ~ & ~ & ~ & ~ & ~ & ~ & ~ & ~ & ~ & ~ & ~ & ~ \\ 
    4 & 1 & 2 & 5 & 6 & 4 & 3 & 1 & ~ & ~ & ~ & ~ & ~ & ~ & ~ & ~ & ~ \\ 
    5 & 1 & 2 & 5 & 10 & 14 & 16 & 16 & 11 & 6 & 4 & 1 & ~ & ~ & ~ & ~ & ~ \\ 
    6 & 1 & 2 & 5 & 10 & 18 & 30 & 37 & 46 & 46 & 45 & 38 & 27 & 16 & 8 & 5 & 1 \\ 
    7 & 1 & 2 & 5 & 10 & 18 & 34 & 53 & 74 & 98 & 118 & 134 & 139 & 134 & 123 & 106 & 86 \\ 
    8 & 1 & 2 & 5 & 10 & 18 & 34 & 57 & 92 & 130 & 184 & 237 & 294 & 349 & 373 & 400 & 407 \\ 
    9 & 1 & 2 & 5 & 10 & 18 & 34 & 57 & 96 & 150 & 220 & 313 & 422 & 551 & 690 & 826 & 945 \\ 
    10 & 1 & 2 & 5 & 10 & 18 & 34 & 57 & 96 & 154 & 242 & 353 & 508 & 699 & 934 & 1223 & 1526 \\ 
    11 & 1 & 2 & 5 & 10 & 18 & 34 & 57 & 96 & 154 & 246 & 377 & 552 & 795 & 1102 & 1503 & 1998 \\ 
    12 & 1 & 2 & 5 & 10 & 18 & 34 & 57 & 96 & 154 & 246 & 381 & 578 & 843 & 1208 & 1691 & 2314 \\ 
    13 & 1 & 2 & 5 & 10 & 18 & 34 & 57 & 96 & 154 & 246 & 381 & 582 & 871 & 1260 & 1807 & 2522 \\ 
    14 & 1 & 2 & 5 & 10 & 18 & 34 & 57 & 96 & 154 & 246 & 381 & 582 & 875 & 1290 & 1863 & 2648 \\ 
    15 & 1 & 2 & 5 & 10 & 18 & 34 & 57 & 96 & 154 & 246 & 381 & 582 & 875 & 1294 & 1895 & 2708
  \end{tblr}
  \label{tab:1324 3412}
\end{table}

\begin{table}[p]
  \centering
  \scriptsize
  \caption{The values of \(\av_{n+1}^k(1324, 3412) - \av_n^k(1324, 3412)\) for \(n, k \leq 15\).}
  \begin{tblr}{
    rows = {mode=math},
    colsep = 5pt,
    cell{-}{-} = {c},
    hline{2,16} = {2-17}{},
    vline{2,18} = {2-15}{},
  }
    n \backslash k & 0 & 1 & 2 & 3 & 4 & 5 & 6 & 7 & 8 & 9 & 10 & 11 & 12 & 13 & 14 & 15 \\ 
    1 & 0 & 1 & ~ & ~ & ~ & ~ & ~ & ~ & ~ & ~ & ~ & ~ & ~ & ~ & ~ & ~ \\ 
    2 & 0 & 1 & 2 & 1 & ~ & ~ & ~ & ~ & ~ & ~ & ~ & ~ & ~ & ~ & ~ & ~ \\ 
    3 & 0 & 0 & 3 & 5 & 4 & 3 & 1 & ~ & ~ & ~ & ~ & ~ & ~ & ~ & ~ & ~ \\ 
    4 & 0 & 0 & 0 & 4 & 10 & 13 & 15 & 11 & 6 & 4 & 1 & ~ & ~ & ~ & ~ & ~ \\ 
    5 & 0 & 0 & 0 & 0 & 4 & 14 & 21 & 35 & 40 & 41 & 37 & 27 & 16 & 8 & 5 & 1 \\ 
    6 & 0 & 0 & 0 & 0 & 0 & 4 & 16 & 28 & 52 & 73 & 96 & 112 & 118 & 115 & 101 & 85 \\ 
    7 & 0 & 0 & 0 & 0 & 0 & 0 & 4 & 18 & 32 & 66 & 103 & 155 & 215 & 250 & 294 & 321 \\ 
    8 & 0 & 0 & 0 & 0 & 0 & 0 & 0 & 4 & 20 & 36 & 76 & 128 & 202 & 317 & 426 & 538 \\ 
    9 & 0 & 0 & 0 & 0 & 0 & 0 & 0 & 0 & 4 & 22 & 40 & 86 & 148 & 244 & 397 & 581 \\ 
    10 & 0 & 0 & 0 & 0 & 0 & 0 & 0 & 0 & 0 & 4 & 24 & 44 & 96 & 168 & 280 & 472 \\ 
    11 & 0 & 0 & 0 & 0 & 0 & 0 & 0 & 0 & 0 & 0 & 4 & 26 & 48 & 106 & 188 & 316 \\ 
    12 & 0 & 0 & 0 & 0 & 0 & 0 & 0 & 0 & 0 & 0 & 0 & 4 & 28 & 52 & 116 & 208 \\ 
    13 & 0 & 0 & 0 & 0 & 0 & 0 & 0 & 0 & 0 & 0 & 0 & 0 & 4 & 30 & 56 & 126 \\ 
    14 & 0 & 0 & 0 & 0 & 0 & 0 & 0 & 0 & 0 & 0 & 0 & 0 & 0 & 4 & 32 & 60
  \end{tblr}
  \label{tab:1324 3412 diffs}
\end{table}

% 3421
\begin{table}[p]
  \centering
  \scriptsize
  \caption{The values of \(\av_n^k(1324, 3421)\) for \(n, k \leq 15\).}
  \begin{tblr}{
    rows = {mode=math},
    colsep = 5pt,
    cell{-}{-} = {c},
    hline{2,17} = {2-17}{},
    vline{2,18} = {2-16}{},
  }
    n \backslash k & 0 & 1 & 2 & 3 & 4 & 5 & 6 & 7 & 8 & 9 & 10 & 11 & 12 & 13 & 14 & 15 \\ 
    1 & 1 & ~ & ~ & ~ & ~ & ~ & ~ & ~ & ~ & ~ & ~ & ~ & ~ & ~ & ~ & ~ \\ 
    2 & 1 & 1 & ~ & ~ & ~ & ~ & ~ & ~ & ~ & ~ & ~ & ~ & ~ & ~ & ~ & ~ \\ 
    3 & 1 & 2 & 2 & 1 & ~ & ~ & ~ & ~ & ~ & ~ & ~ & ~ & ~ & ~ & ~ & ~ \\ 
    4 & 1 & 2 & 5 & 6 & 5 & 2 & 1 & ~ & ~ & ~ & ~ & ~ & ~ & ~ & ~ & ~ \\ 
    5 & 1 & 2 & 5 & 10 & 16 & 18 & 16 & 10 & 5 & 2 & 1 & ~ & ~ & ~ & ~ & ~ \\ 
    6 & 1 & 2 & 5 & 10 & 20 & 30 & 47 & 55 & 53 & 43 & 31 & 20 & 10 & 5 & 2 & 1 \\ 
    7 & 1 & 2 & 5 & 10 & 20 & 34 & 57 & 84 & 122 & 152 & 162 & 160 & 138 & 111 & 80 & 55 \\ 
    8 & 1 & 2 & 5 & 10 & 20 & 34 & 61 & 94 & 145 & 208 & 292 & 369 & 437 & 471 & 478 & 453 \\ 
    9 & 1 & 2 & 5 & 10 & 20 & 34 & 61 & 98 & 155 & 232 & 340 & 470 & 645 & 830 & 1003 & 1174 \\ 
    10 & 1 & 2 & 5 & 10 & 20 & 34 & 61 & 98 & 159 & 242 & 364 & 522 & 740 & 1003 & 1344 & 1725 \\ 
    11 & 1 & 2 & 5 & 10 & 20 & 34 & 61 & 98 & 159 & 246 & 374 & 546 & 792 & 1106 & 1517 & 2034 \\ 
    12 & 1 & 2 & 5 & 10 & 20 & 34 & 61 & 98 & 159 & 246 & 378 & 556 & 816 & 1158 & 1620 & 2219 \\ 
    13 & 1 & 2 & 5 & 10 & 20 & 34 & 61 & 98 & 159 & 246 & 378 & 560 & 826 & 1182 & 1672 & 2322 \\ 
    14 & 1 & 2 & 5 & 10 & 20 & 34 & 61 & 98 & 159 & 246 & 378 & 560 & 830 & 1192 & 1696 & 2374 \\ 
    15 & 1 & 2 & 5 & 10 & 20 & 34 & 61 & 98 & 159 & 246 & 378 & 560 & 830 & 1196 & 1706 & 2398
  \end{tblr}
  \label{tab:1324 3421}
\end{table}

\begin{table}[p]
  \centering
  \scriptsize
  \caption{The values of \(\av_{n+1}^k(1324, 3421) - \av_n^k(1324, 3421)\) for \(n, k \leq 15\).}
  \begin{tblr}{
    rows = {mode=math},
    colsep = 5pt,
    cell{-}{-} = {c},
    hline{2,16} = {2-17}{},
    vline{2,18} = {2-15}{},
  }
    n \backslash k & 0 & 1 & 2 & 3 & 4 & 5 & 6 & 7 & 8 & 9 & 10 & 11 & 12 & 13 & 14 & 15 \\ 
    1 & 0 & 1 & ~ & ~ & ~ & ~ & ~ & ~ & ~ & ~ & ~ & ~ & ~ & ~ & ~ & ~ \\ 
    2 & 0 & 1 & 2 & 1 & ~ & ~ & ~ & ~ & ~ & ~ & ~ & ~ & ~ & ~ & ~ & ~ \\ 
    3 & 0 & 0 & 3 & 5 & 5 & 2 & 1 & ~ & ~ & ~ & ~ & ~ & ~ & ~ & ~ & ~ \\ 
    4 & 0 & 0 & 0 & 4 & 11 & 16 & 15 & 10 & 5 & 2 & 1 & ~ & ~ & ~ & ~ & ~ \\ 
    5 & 0 & 0 & 0 & 0 & 4 & 12 & 31 & 45 & 48 & 41 & 30 & 20 & 10 & 5 & 2 & 1 \\ 
    6 & 0 & 0 & 0 & 0 & 0 & 4 & 10 & 29 & 69 & 109 & 131 & 140 & 128 & 106 & 78 & 54 \\ 
    7 & 0 & 0 & 0 & 0 & 0 & 0 & 4 & 10 & 23 & 56 & 130 & 209 & 299 & 360 & 398 & 398 \\ 
    8 & 0 & 0 & 0 & 0 & 0 & 0 & 0 & 4 & 10 & 24 & 48 & 101 & 208 & 359 & 525 & 721 \\ 
    9 & 0 & 0 & 0 & 0 & 0 & 0 & 0 & 0 & 4 & 10 & 24 & 52 & 95 & 173 & 341 & 551 \\ 
    10 & 0 & 0 & 0 & 0 & 0 & 0 & 0 & 0 & 0 & 4 & 10 & 24 & 52 & 103 & 173 & 309 \\ 
    11 & 0 & 0 & 0 & 0 & 0 & 0 & 0 & 0 & 0 & 0 & 4 & 10 & 24 & 52 & 103 & 185 \\ 
    12 & 0 & 0 & 0 & 0 & 0 & 0 & 0 & 0 & 0 & 0 & 0 & 4 & 10 & 24 & 52 & 103 \\ 
    13 & 0 & 0 & 0 & 0 & 0 & 0 & 0 & 0 & 0 & 0 & 0 & 0 & 4 & 10 & 24 & 52 \\ 
    14 & 0 & 0 & 0 & 0 & 0 & 0 & 0 & 0 & 0 & 0 & 0 & 0 & 0 & 4 & 10 & 24
  \end{tblr}
  \label{tab:1324 3421 diffs}
\end{table}